\documentclass{article}

\usepackage{amsmath,graphicx,amssymb,geometry}

\usepackage{amsthm}
\usepackage[T1]{fontenc}
\usepackage{mathrsfs}
\usepackage{enumerate}
\usepackage[utf8]{inputenc} %Clavier français Canadien
\usepackage{tikz}
\usepackage{dsfont}

\usepackage{hyperref}

\numberwithin{equation}{section}
\begin{document}
%\maketitle
%\thispagestyle{empty}
%$ $
%\vspace{8cm}
%\begin{center}
%{\bf\LARGE TITLE}\\
%\end{center}

%\begin{center}
%{AUTHOR}
%\end{center}

%\newpage
%\thispagestyle{empty}

%\newpage

%\begin{abstract} Résumé. \end{abstract}

%\setcounter{page}{1}

%\clearpage
%\tableofcontents

\newtheorem{defi}{Definition}[section]
\newtheorem{nota}[defi]{Notation}
\newtheorem{propo}[defi]{Proposition}
\newtheorem{lemme}[defi]{Lemma}
\newtheorem{thm}[defi]{Theorem}
\newtheorem{ab}[defi]{Abus}
\newtheorem{rem}[defi]{Remark}
\newtheorem{coro}[defi]{Corollary}
\newtheorem{ex}[defi]{Example}
\newtheorem{app}[defi]{Application}
\newtheorem{assu}[defi]{Assumption}
\newtheorem{question}[defi]{Question}

\newtheorem{defis}{Definition}[subsection]
\newtheorem{notas}[defis]{Notation}
\newtheorem{propos}[defis]{Proposition}
\newtheorem{lemmes}[defis]{Lemma}
\newtheorem{thms}[defis]{Theorem}
\newtheorem{abs}[defis]{Abus}
\newtheorem{rems}[defis]{Remark}
\newtheorem{coros}[defis]{Corollary}
\newtheorem{exs}[defis]{Example}
\newtheorem{apps}[defis]{Application}
\newtheorem{assus}[defis]{Assumption}

\newtheorem{defiss}{Definition}[subsubsection]
\newtheorem{notass}[defiss]{Notation}
\newtheorem{proposs}[defiss]{Proposition}
\newtheorem{lemmess}[defiss]{Lemme}
\newtheorem{thmss}[defiss]{Theorem}
\newtheorem{abss}[defiss]{Abus}
\newtheorem{remss}[defiss]{Remark}
\newtheorem{coross}[defiss]{Corollary}
\newtheorem{exss}[defiss]{Example}
\newtheorem{appss}[defiss]{Application}
\newtheorem{assuss}[defiss]{Assumption}

%%%Abréviations
\newcommand{\om}{\omega}
\newcommand{\Om}{\Omega}
\newcommand{\tea}{\theta}
\newcommand{\eps}{\varepsilon}
\newcommand{\ii}{\infty}
\newcommand{\wf}{\widehat{f}}
\newcommand{\wW}{\widehat{W}}
\newcommand{\wpsi}{\widehat{\psi}}
\newcommand{\Geta}{\Z_{(\eta)}^d}
\newcommand{\Getan}[1]{\Gamma^{\eta}(#1)}
\newcommand{\Dn}[1]{\Delta(#1)}
\newcommand{\Getanl}[1]{\Gamma^{\eta}_l(#1)}
\newcommand{\Dnl}[1]{\Delta_l(#1)}
\newcommand{\ind}[1]{\mathds{1}_{#1}}
\newcommand{\KK}{\mathcal{K}}
\newcommand{\logr}[1]{\sqrt{\log{\left(3+#1\right)}}}
\newcommand{\puiss}[1]{\left(1+#1\right)^{1/\alpha+\delta}}
\newcommand{\puissm}[1]{\left(1+#1\right)^{\mu}}
\newcommand{\puissa}[1]{\left(1+#1\right)^{1/\alpha+\lfloor \alpha \rfloor/2+\delta}}
\newcommand{\sumjn}[1]{\mathrm{sumj}^-(#1)}
\newcommand{\sumjp}[1]{\mathrm{sumj}^+(#1)}
\newcommand{\sumk}[1]{\mathrm{sumk}(#1)}
\newcommand{\Ree}{\mathcal{R}e}
\newcommand{\Imm}{\mathcal{I}m}
\newcommand{\majeps}[1]{\mathcal{L}_{\alpha}\!\left(#1\right)}

\newcommand{\wpsiJK}{\widehat{\psi}_{J,K}}
\newcommand{\wpsialphaJK}{\widehat{\psi}_{\alpha,J,K}}
\newcommand{\epsJK}{\eps_{J,K}}
\newcommand{\epsalphaJK}{\eps_{\alpha,J,K}}
\newcommand{\PsiJ}{\Psi_J}
\newcommand{\PsialphaJ}{\Psi_{\alpha,J}}
\newcommand{\PhialphaJ}{\Phi_{\alpha,J}}
\newcommand{\Lalpha}{p_*}
\newcommand{\SG}{\Sigma(G)}
\newcommand{\bJK}{b_{J,K}}
\newcommand{\GJK}{G_{J,K}}
\newcommand{\ZZ}{\Z^d\times\Z^d}

%%%Ensembles
\newcommand{\C}{\mathbb{C}}
\newcommand{\Q}{\mathbb{Q}}
\newcommand{\Z}{\mathbb{Z}}
\newcommand{\N}{\mathbb{N}}
\newcommand{\K}{\mathbb{K}}
\newcommand{\R}{\mathbb{R}}
\newcommand{\Hi}{\mathbb{H}}

\def\L{\mathcal{L}}
\def\ov{\overline}
\def\un{\underline}
\def\wt{\widetilde}
\def\wh{\widehat}
\def\ga{\gamma}
\def\A{\mathcal{A}}
\def\H{\mathcal{H}}
\def\F{\mathcal{F}}
\def\sp{\rm supp}
\def\cad{\rm card}
\def\Upss{\Upsilon^*}
\def\Xnwav{X_n^{\mbox{{\em\tiny wav,T}}}}
\def\Xnmwav{X_{n+m}^{\mbox{{\em\tiny wav,T}}}}
\def\Xnwavabs{X_{n}^{\mbox{{\em\tiny wav,T, abs}}}}
\def\jzero{J_0}

%%%Ensembles de fonctions
\newcommand{\Lp}[2]{{L}^{ #1 }\! \left( #2 \right)}
\newcommand{\Lploc}[2]{{L}_{\mathrm{loc}}^{ #1 }\! \left( #2 \right)}
\newcommand{\lp}[2]{{l}^{ #1 }\! \left( #2 \right)}
\newcommand{\Co}[1]{{\mathcal{C}}\! ( #1 )}
\newcommand{\Coloc}[1]{{C}_{\mathrm{loc}}\! \left( #1 \right)}
\newcommand{\Cn}[2]{{\mathcal{C}}^{#1  }\! ( #2 )}
\newcommand{\Cnloc}[2]{{C}_{\mathrm{loc}}^{ #1 }\! \left( #2 \right)}
\newcommand{\Ho}[2]{{\mathcal{C}}^{ #1 }\! ( #2 )}
\newcommand{\Holoc}[2]{{C}_{\mathrm{loc}}^{ #1 }\! ( #2 )}
\newcommand{\sas}{\mathcal{S}\alpha\mathcal{S}}

%%%Probabilités
\newcommand{\pr}{\mathbb{P}}
\newcommand{\esp}{\mathbb{E}}
\newcommand{\var}{\mathrm{Var}}
\newcommand{\cov}{\mathrm{Cov}}

%%%Lettres rondes
\newcommand{\Kr}{\mathcal{K}}

%Divers
\newcommand{\norm}[1]{\left|\! \left| {#1} \right|\!\right|}
\newcommand{\norminf}[1]{\Big|\! \Big| {#1} \Big|\!\Big|_{T,\ii}}
\newcommand{\norminfT}[2]{\Big|\! \Big| {#1} \Big|\!\Big|_{{#2},\ii}}
\newcommand{\va}[1]{ \left| {#1} \right|}
\newcommand{\supp}{\mathrm{supp}}

\title{Stationary increments harmonizable stable fields: upper estimates on path behaviour}
\author{Antoine Ayache\, and\, Geoffrey Boutard\\
 UMR CNRS 8524 Laboratoire Paul Painlev\'e,\\
Université Lille 1,\\59655 Villeneuve d'Ascq Cedex, France\\
E-mails: antoine.ayache@math.univ-lille1.fr\\
\quad \quad \quad\, geoffrey.boutard@ed.univ-lille1.fr}
%\ead{geoffrey.boutard@ed-unvi-lille1.fr}
%\begin{keyword}
%Gaussian \sep Field
%\end{keyword}
%\address{Laboratoire Paul Painlev\'e,\\Université de Lille 1,\\59655 Villeneuve d'Ascq Cedex, France}

%\clearpage
%\begin{center}\LARGE Wavelet decomposition of the random field $X=\left\{X(t); \,t\in \R^d\right\}$, $d\geq 1.$ \end{center}
%\vspace{1cm}

\date{}
\maketitle
%\newpage
%\tableofcontents
%\newpage

\begin{abstract}
Studying sample path behaviour of stochastic fields/processes is a classical research topic in probability theory and related areas such as fractal geometry. To this end, many methods have been developed since a long time in Gaussian frames. They often rely on some underlying "nice" Hilbertian structure, and can also require finiteness of moments of high order. Therefore, they can hardly be transposed to frames of heavy-tailed stable probability distributions. 

However, in the case of some linear non-anticipative moving average stable fields/processes, such as the linear fractional stable sheet and the linear multifractional stable motion, rather new wavelet strategies have already proved to be successful in order to obtain sharp moduli of continuity and other results on sample path behaviour. The main goal of our article is to show that, despite the difficulties inherent in the frequency domain, such kind of a wavelet methodology can be generalized and improved, so that it also becomes fruitful in a general harmonizable stable setting with stationary increments. Let us point out that there are large differences between this harmonizable setting and the moving average stable one.

The real-valued harmonizable stable stochastic field $X$, we focus on, is defined on $\R^d$ through an arbitrary spectral density belonging to a general and wide class of functions. First, we introduce a wavelet type random series representation of $X$, and express it as the finite sum $X=\sum_\eta X^\eta$, where the fields $X^\eta$ are called the $\eta$-frequency parts, since they extend the usual low-frequency and high-frequency parts. Moreover, we show the continuity of the sample paths of the $X^\eta$'s and $X$; also, we discuss the existence and continuity of their partial derivatives of an arbitrary order. Thereafter, we obtain several almost sure upper estimates related with: $(a)$~the anisotropic behaviour of generalized directional increments of the $X^\eta$'s and $X$, on an arbitrary fixed compact cube of~$\R^d$; 
$(b)$~the behaviour at infinity of the $X^\eta$'s, of $X$, and of their partial derivatives, when they exist. We mention that all the results on sample paths, obtained in the article, are valid on the same event of probability~$1$; furthermore, this event is "universal", in the sense that it does not depend, in any way, on the spectral density associated with~$X$.
\end{abstract}

\medskip

      {\it Running head}:  Behaviour of stationary increments harmonizable stable fields \\



%
%\noindent{\textit{MSC:}} 
%
%\noindent{\textit{Keywords:}}

\section{Introduction}
Many methods have been developed since a long time in order to study sample path behaviour of Gaussian fields/processes (see e.g. \cite{CL04, Adl81, Khoshnevisan02, lif95,let91,xia09,xia13,xue10,MWX13}). Generally speaking, most of these methods can hardly be transposed to frames of heavy-tailed stable distributions. Such distributions are very important in probability and statistics because they are a natural counterpart to the Gaussian ones. They have been widely examined in the literature; a classical reference on them and related topics, including stable random measures and their associated stochastic integrals, is the book of Samorodnitsky and Taqqu \cite{SamTaq}. Throughout our article the underlying probability space is denoted by $(\Om,\mathcal{G},\pr)$. Recall that a real-valued random variable $Z$ is said to have a symmetric stable distribution of stability parameter $\alpha\in (0,2]$ and scale parameter $\sigma\in\R_+$, if its characteristic function can be expressed as $\exp(-\sigma^\alpha |\zeta|^\alpha)$, for any $\zeta\in\R$. Notice that $Z$ reduces to a centered Gaussian random variable when $\alpha=2$. The situation is very different when $\alpha\in (0,2)$ and $\sigma>0$; the distribution of $Z$ becomes heavy-tailed. Namely, the asymptotic behaviour of the probability $\pr(|Z|>z)$ is of the same order as $z^{-\alpha}$ when the real number $z$ goes to $+\infty$. This, in particular, implies that an absolute moment of $Z$ has to be of a small order in order to be finite; more precisely one has that $\esp (|Z|^\ga)=+\infty$, as soon as $\ga \ge \alpha$. 

In the case of some linear non-anticipative moving average stable fields/processes, such as the linear fractional stable sheet and the linear multifractional stable motion, rather new wavelet methods have already proved to be successful in studying sample path behaviour (see \cite{ayache2009linear,hamonier2012lmsm}). Can this methodology be adapted to some harmonizable stable fields/processes? Providing an answer to this question is a non trivial problem, since, generally speaking, there are large differences between an harmonizable stable setting and a moving average one (see for instance \cite{kono1991holder,EmMa,SamTaq}). The main goal of our article is to study this issue in the case of a stationary increments real-valued symmetric harmonizable $\alpha$-stable field $X:=\left\{X(t),t\in\R^d\right\}$ having a general form. Basically, we show that, despite the difficulties inherent in the frequency domain, the wavelet methodology can be generalized and improved in such way that it works well in the case of this general harmonizable stable field $\left\{X(t),t\in\R^d\right\}$.  We mention that when $\left\{X(t),t\in\R^d\right\}$ is a (multi-)operator scaling stable random field satisfying some conditions, interesting results on its H\"older regularity have been obtained in \cite{kono1991holder,bierme2009holder,bierme2011multi}. The methodology employed in these articles relies on a representation of $\left\{X(t),t\in\R^d\right\}$ as a LePage series; it is rather different from the wavelet methodology we use in the present paper.

In order to precisely define $\left\{X(t),t\in\R^d\right\}$, first, we need to introduce some notations and make some brief recalls on stable stochastic integrals. We denote by $\wt{M}_{\alpha}$ a complex-valued rotationally invariant $\alpha$-stable random measure on $\R^d$ with Lebesgue control measure. The related stable stochastic integral is denoted by 
$\int_{\R^d} \big (\cdot\big)\,\mathrm{d}\wt{M}_{\alpha}$. It is a linear map on the Lebesgue space $L^\alpha (\R^d)$ such that, for any deterministic function $g\in L^\alpha (\R^d)$, the real part $\Ree\big\{\int_{\R^d} g(\xi)\, \mathrm{d}\wt{M}_{\alpha}(\xi)\big\}$ is a real-valued symmetric $\alpha$-stable random variable with a scale parameter satisfying
\begin{equation}
\label{eq:isom-stable}
\sigma\Big(\Ree\big\{\int_{\R^d} g(\xi)\, \mathrm{d}\wt{M}_{\alpha}(\xi)\big\}\Big)^\alpha=\int_{\R^d} \big |g(\xi)\big|^\alpha\,d\xi. 
\end{equation}
Observe that the equality (\ref{eq:isom-stable}) is reminiscent of the classical isometry property of Wiener integrals; in particular, it implies 
that $\Ree\big\{\int_{\R^d} g_n (\xi)\, \mathrm{d}\wt{M}_{\alpha}(\xi)\big\}$ converges to $\Ree\big\{\int_{\R^d} g(\xi)\, \mathrm{d}\wt{M}_{\alpha}(\xi)\big\}$ in probability, when a sequence $(g_n)_n$ converges to $g$ in $L^\alpha (\R^d)$. This will be useful for us.

Let us now focus on the definition of $\left\{X(t),t\in\R^d\right\}$. Its main ingredient is $f$, an arbitrary real-valued 
%non-negative even\footnote{A real-valued (or complex-valued) function $f$ on $\R^d$ is said to be even (resp. odd) if, for almost all $\xi\in\R^d$, one has $f(\xi)=f(-\xi)$ (resp. $f(\xi)=-f(-\xi)$).} 
Lebesgue measurable even function on $\R^d$ satisfying the condition:
\begin{equation}
\label{inte}
\int_{\R^d}\min\big(1,\norm{\xi}^{\alpha}\big)\big|f(\xi)\big|^{\alpha}\,\mathrm{d}\xi<+\infty,
\end{equation}
where $\norm{\cdot}$ denotes the Euclidian norm on $\R^d$. Notice that, by analogy with the Gaussian case (see \cite{BE03} for instance), the function $|f|^{\alpha}$ is called \textit{the spectral density of the field $X$}. Thanks to (\ref{inte}), for any $t\in\R^d$, the function $\xi\mapsto \big (e^{it\cdot\xi}-1\big)f(\xi)$ belongs to $L^\alpha (\R^d)$, and thus it is integrable with respect to $\wt{M}_\alpha$. The field $\left\{X(t),t\in\R^d\right\}$ is defined, for all $t\in\R^d$, as
\begin{equation}
\label{deffield}
X(t)=\Ree\left\{\int_{\R^d} \big (e^{it\cdot\xi}-1\big)f(\xi)\,\mathrm{d}\wt{M}_{\alpha}(\xi)\right\},
\end{equation}
where $t\cdot\xi$ denotes the usual inner product of $t$ and $\xi$. We mention that not only the study of sample path behaviour of $\left\{X(t),t\in\R^d\right\}$ is interesting in its own right (among other things, for the theoretical reasons given before), but also it may have an impact on future development of new applications related with modelling of anisotropic materials in frames of heavy-tailed stable distributions. It is worthwhile to note that in Gaussian frames such a modelling has already proved to be useful, in particular for detecting osteoporosis in human bones through the analysis of their radiographic images (see \cite{estrade2003,BE03,bierme2009}).

%A classical example of a $\alpha$-stable field, which can be expressed as in \eqref{deffield} is  {\em Fractional $\alpha$-stable Field} with an 
%arbitrary Hurst parameter $H\in(0,1)$; in this case the spectral density is simply given by $f^\alpha(\xi)=\norm{\xi}^{-\alpha H-d}$, for %almost all
%\footnote{With respect to the Lebesgue measure.} 
%$\xi\in\R^d$.

Typically, $X$ is an anisotropic model when the rate of vanishing at infinity of the corresponding spectral density $|f|^\alpha$ changes from one axis of $\R^d$ to another; therefore, we focus on the class of the so-called admissible functions~$f$, defined in the following way.   

\begin{defi}
\label{def:adm}
Let $\lfloor1/\alpha\rfloor$ be the integer part of $1/\alpha$, the inverse of the stability parameter $\alpha\in (0,2]$. We set
\begin{equation}
\label{eq:bs}
p_*:=\max\big\{2,\lfloor 1/\alpha\rfloor+1\big\}.
\end{equation}
The function $f$ in \eqref{deffield} is said to be {\em admissible} when it satisfies the following three conditions.
\begin{itemize}
\item[$(\H_1)$] 
For all multi-index $p:=(p_1,p_2,\ldots,p_d)\in\big\{0,1,2,\ldots, p_*\big\}^d$, the partial derivative function
$$
\partial^p f:=\frac{\partial^{p_1}\partial^{p_2}\ldots \partial^{p_d}}{(\partial \xi_1)^{p_1}(\partial \xi_2)^{p_2}\ldots(\partial \xi_d)^{p_d}}\, f\text{ \,\,\, (with the convention that $\partial^0\!f:=f$)}
$$
is well-defined and continuous on the open set $\big(\R\setminus\{0\}\big)^d$; that is the Cartesian product of $\R\setminus\{0\}$ with itself $d$ times.
%$f$ is locally bounded on $\R^d\setminus\{0\}$; in other words $f$ is bounded on each compact subset of $\R^d$. and that the partial %derivative function
%$$
%\partial^b f:=\frac{\partial^{b_1}\partial^{b_2}\ldots \partial^{b_d}}{(\partial \xi_1)^{b_1}(\partial \xi_2)^{b_2}\ldots(\partial \xi_d)^{b_d}}\! %f\text{ \,\,\,and \,\,\,} \partial^0\!f:=f,
%$$
%of $f$, viewed as tempered distribution, exits and belongs to $\Lploc{\max(1,\alpha)}%{\R^d\setminus\{0\}}$\footnote{$\Lploc{p}%{\R^d\setminus\{0\}}$ denotes the space of complex-valued functions $f$ defined on %$\R^d$ satisfying for each compact set $K$ of $\R^d\setminus\{0\}$,
%$$
%\int_{K}\va{f(\xi)}^{p}\,\mathrm{d}\xi<+\infty.
%$$}. 
\item[$(\H_2)$] 
There are a positive constant $c'$ and an exponent $a'\in (0,1)$ such that, for each $p\in\big\{0,1,2,\ldots,p_*\big\}^d$, and $\xi\in \big(\R\setminus\{0\}\big)^d$,
\begin{equation}
\label{A3}
\norm{\xi}\leq \frac{8\pi}{3}\sqrt{d} \Longrightarrow\big |\partial^p f (\xi)\big|\leq c'\norm{\xi}^{-a'-d/\alpha-\mathrm l(p)},
\end{equation}
where $\mathrm l(p):=p_1+p_2+\dots+p_d$ is the length of the multi-index $p$. 
\item[$(\H_3)$] There exist a positive constant $c$ and $d$ positive exponents $a_1,\dots,a_d$ such that for every  $p\in\big\{0,1,2,\ldots,p_*\big\}^d$, and $\xi\in \big(\R\setminus\{0\}\big)^d$,
\begin{equation}
\label{A2}
\norm{\xi}\geq\frac{2\pi}{3} \Longrightarrow \big|\partial^p f(\xi)\big|\leq c\prod_{l=1}^d(1+\va{\xi_l})^{-a_l-1/\alpha-p_l}.
\end{equation}
\end{itemize}
%The set of all those functions $f$ is denoted by $\A$.
\end{defi}
\begin{rem}
\label{rem1:adm}
It is clear that when $f$ is admissible then it satisfies the condition \eqref{inte}. Also notice that in \eqref{A3} and \eqref{A2}, the quantities $8\pi\sqrt{d}/3$ and $2\pi/3$ can be replaced by any other fixed positive quantities. More importantly, notice that many functions belong to the admissible class, as, for instance, the function 
\[
\xi=(\xi_1,\ldots,\xi_d)\longmapsto\bigg (\sum_{l=1}^d \xi_l^2\bigg)^{-(u+d/\alpha)/2}\times \prod_{l=1}^d\big (1+|\xi_l|\big)^{-v_l},
\]
 where $u\in (0,1)$ and $v_1,\ldots, v_d\in [0,+\infty)$ are arbitrary fixed parameters. 
\end{rem}

%\noindent{\bf{Remark:}} We mention that the methods we employ for answering to Question~1 and Question~2 rely on an anisotropic wavelet %type series representation of $X$ (see \cite{DA92}, \cite{ME89}, and \cite{AX05}). These methods offer the following three advantages:
%\begin{enumerate}[(1)]
%\item The results are valid on an universal event $\Om_1^*$ of probability~1 not depending on $f$.
%\item We show that the low frequency part of $X$ is infinitely pathwise differentiable. Also, we show that, some other parts of $X$, as well as %$X$ itself, can sometimes be partially differentiable.
%\item Not only we obtain results on usual increments but also on rectangular increments and other kinds of increments. For instance, when %$d=2$, it follows from Theorem \ref{thm:etaincr} in our article that,
%\begin{equation}
%\sup_{(t'_1,t'_2),(t''_1,t''_2)\in[-T,T]^2}\left\{\frac{\big |X(t''_1,t''_2,\om)-X(t''_1,t'_2,\om)-X(t'_1,t''_2,\om)+X(t'_1,t'_2,%\om)\big|}%{\va{t_1''-%t_1'}^{b_1}\va{t_2''-t_2'}^{b_2}}\right\}<+\ii,
%\end{equation} 
%for all $b_1\in(0,a_1)$, $b_2\in(0,a_2)$, $\om\in\Om_1^*$ and $T>0$.
%\end{enumerate}

The rest of the article is organized in the following way. In section~\ref{secwavrep}, we introduce a wavelet type random series representation of $X$, and express it as the finite sum $X=\sum_\eta X^\eta$, where the fields $X^\eta$ are called the $\eta$-frequency parts, since they extend the usual low-frequency and high-frequency parts. Then, we show that the sample paths of all the $X^\eta$'s are continuous on $\R^d$, and we connect the existence and continuity of their partial derivative, of an arbitrary order, with the rates of vanishing at infinity of the spectral density along the axes i.e. with the exponents $a_1,\ldots, a_d$ in \eqref{A2}. Notice that, in order to avoid this section~\ref{secwavrep} being very long, the proofs of some results in it have been postponed to the appendices~\ref{app:psijalpha}, ~\ref{app:decompo} and \ref{app:lepage}. In section~\ref{sechold}, we obtain, in terms of $a_1,\ldots, a_d$, almost sure upper estimates of the anisotropic behaviour of generalized directional increments of the $X^\eta$'s and $X$, on an arbitrary compact cube of $\R^d$. In section~\ref{secbehav}, we are concerned with the behaviour in the vicinity of infinity of the $X^\eta$'s, of $X$, and of their partial derivatives, when they exist. Mainly, we show that $X$ and its low-frequency part $X^0$ are, up to a logarithmic factor, dominated by the power function $\norm{t}^{a'}$, where $a'$ is the same exponent as in \eqref{A3}. Also, we show that the other $\eta$-frequency parts and all the partial derivatives, that exist, have at most a logarithmic behaviour. 

Before ending the present introductory section, we mention that all the results on sample paths, obtained in our article, are valid on the same event of probability~$1$, namely, the event $\Om_1 ^*$ introduced in Lemma~\ref{le:eps-JK}. Notice that $\Om_1 ^*$ is "universal", in the sense that it does not depend, in any way, on the admissible function $f$ associated with the field $X$ through \eqref{deffield}.

\section{Wavelet type random series representation}
\label{secwavrep}

In the general case, where the stability parameter $\alpha\in (0,2]$ is arbitrary, the strategy, allowing to obtain the wavelet type random series representation of $\{X(t), t\in\R^d\}$, that we are looking for, follows, more or less, the main steps as in the Gaussian case, where $\alpha=2$; yet, the arguments of their proofs have to be significantly modified in order to fit with the general case.  First, we intend to present these main steps in a rather heuristic way, by avoiding, as far as possible, to be technical. This is why we restrict, for the time being, our presentation to the Gaussian case which is less difficult to understand than the general one. 

We denote by
$
\big\{\psi_{J,K}:(J,K)\in \Z^d\times\Z^d\big\}
$ 
the orthonormal basis of $L^2 (\R^d)$ defined in the following way: for all $(J,K):=(j_1,\ldots,j_d,k_1,\ldots,k_d)\in\Z^d\times\Z^d$ and $x:=(x_1,\ldots, x_d)\in\R^d$ 
\begin{equation}
\label{eq1:blm}
\psi_{J,K}(x):=\prod_{l=1}^d 2^{j_l/2}\psi^1 (2^{j_l} x_l-k_l),
\end{equation} 
where $\psi^1$ denotes an usual 1D Lemari\'e-Meyer mother wavelet. We refer to the books of Meyer \cite{ME89, Meyer92} and to that of Daubechies \cite{DA92} for a complete description of the wavelet tools used in the present section. It is worthwhile noting that $\psi^1$ is a real-valued function belonging to the Schwartz class $S(\R)$; that is the space of complex-valued $C^\infty$ functions on $\R$ having rapidly decreasing derivatives at any order. Also, we mention that the Fourier transform of $\psi^1$, denoted by $\widehat{\psi^1}$, is a compactly supported  $C^\infty$ function on $\R$, such that
\begin{equation}
\label{eq2:blm}
\sp\,\wh{\psi^1}\subseteq \KK:=\left\{\lambda\in\R : \frac{2\pi}{3}\leq\va{\lambda}\leq\frac{8\pi}{3}\right\}.
\end{equation} 
Observe that it follows from (\ref{eq1:blm}) and elementary properties of the Fourier transform that, for any $\xi\in\R^d$,
\begin{equation}
\label{eq3:blm} 
\wpsiJK(\xi)=\prod_{l=1}^d 2^{-j_l/2}e^{-i2^{-j_l}k_l\xi_l}\,\wh{\psi^1}(2^{-j_l}\xi_l).
\end{equation}
Therefore combining (\ref{eq2:blm}) and (\ref{eq3:blm}) one gets that
\begin{equation}
\label{eq4:blm}
 \sp\,\wh{\psi}_{\text{$J,K$}}\subset\left\{\xi\in\R^d : \mbox{{\small for all $l=1,\ldots, d$ one has}}\,\, \frac{2^{\text{$j_l$}+1}\pi}{3}\leq\text{$\va{\xi_l}$}\leq\frac{2^{\text{$j_l$}+3}\pi}{3}\right\};
\end{equation}
this inclusion will be very useful for us.

Next notice that (\ref{inte}) and the assumption $\alpha=2$ imply that, for any fixed $t\in\R^d$, the function $\xi\mapsto \big(e^{it\cdot\xi}-1\big)f(\xi)$ belongs to $L^2 (\R^d)$. Therefore, it can be expressed as  
%Using the fact that the sequence $\big\{\overline{\wpsiJK}:(J,K)\in \Z^d\times\Z^d\big\},$ is an orthonormal basis of $\Lp{2}{\R^d}$, when %the function $f$ in $\eqref{deffield}$ satisfies \eqref{inte} with $\alpha=2$, we have, for any $t\in\R^d$ and $\xi\in\R^d$,
\begin{equation}
\label{eq1:decomp}
\big(e^{it\cdot\xi}-1\big)f(\xi)=\sum_{(J,K)\in\Z^d\times\Z^d} s_{J,K}(t)\overline{\wpsiJK(\xi)},
\end{equation}
where 
\begin{equation}
\label{eq:wavs}
s_{J,K}(t):=\int_{\R^d}\left(e^{it\cdot\xi}-1\right)f(\xi)\wpsiJK(\xi)\,\mathrm{d}\xi,
\end{equation}
and $\overline{\wpsiJK(\xi)}$ denotes the complex conjugate of $\wpsiJK(\xi)$; observe that, at this stage, the right-hand side in \eqref{eq1:decomp}, has to be viewed as a series of functions, of the variable $\xi$, which converges in the $\Lp {2}{\R^d}$ norm. Now, denote by $\PsiJ$ the real-valued function defined, for all $x\in\R^d$, as
\begin{equation}
\label{eq:psiJ}
\PsiJ(x):=2^{(j_1+\dots+j_d)/2}\int_{\R^d}e^{ix\cdot \xi}f\left(2^J\xi\right)\wpsi_{0,0}(\xi)\mathrm d\xi,
%\PsialphaJ(x)=2^{(j_1+\dots+j_d)/2}\int_{\R^d}e^{ix\cdot \xi}\wf(2^J\xi)\wpsi(\xi_1)\dots\wpsi(\xi_d)\mathrm d\xi;
\end{equation}
with the convention~\footnote{Notice that such a convention will be extensively used in all the rest of our article, without being recalled.} that $2^J \xi:=(2^{j_1} \xi_1,\dots,2^{j_d} \xi_d)$. It can easily be derived from \eqref{eq3:blm}, \eqref{eq:wavs} and \eqref{eq:psiJ} that
\begin{equation}
\label{eq:wavs2}
s_{J,K}(t)=\PsiJ\left(2^Jt-K\right)-\PsiJ\left(-K\right).
\end{equation} 
Then, it results from (\ref{eq1:decomp}), \eqref{eq:wavs2} and (\ref{deffield}) (with $\alpha=2$) that
\begin{equation}
\label{eq:add1}
X(t)=\Ree\left\{\int_{\R^d}\bigg(\sum_{(J,K)\in\Z^d\times\Z^d}\big (\PsiJ\left(2^Jt-K\right)-\PsiJ\left(-K\right)\big)\overline{\wpsiJK(\xi)}\bigg)\,\mathrm{d}\wt{M}_{2}(\xi)\right\}.
\end{equation}
Finally, in view of (\ref{eq:isom-stable}), it turns out that, roughly speaking, one can interchange in (\ref{eq:add1}) the integration and the summation. Thus, we get that
\begin{equation}
\label{eq:add2}
X(t)=\sum_{(J,K)\in\Z^d\times\Z^d}\big (\PsiJ\left(2^Jt-K\right)-\PsiJ\left(-K\right)\big)\epsJK,
\end{equation}
where the $\epsJK$'s are the centered real-valued Gaussian random variables defined as
\[
\epsJK:=\Ree\left\{\int_{\R^d}\overline{\wpsiJK(\xi)}\,\mathrm d\wt{M}_{2}(\xi)\right\}.
\]

Having presented, in the Gaussian case $\alpha=2$, the main steps of the strategy allowing to obtain the wavelet type random series representation \eqref{eq:add2} of $\{X(t), t\in\R^d\}$; from now on we assume that $\alpha\in (0,2]$ is arbitrary, and that the function $f$ in \eqref{deffield} is any admissible function in the sense of Definition~\ref{def:adm}. Our present goal is to show that the strategy previously employed, in the Gaussian case, for deriving \eqref{eq:add2}, can be extended to the general case. To this end, the arguments, we have used in the "convenient" framework of the Hilbert space $L^2(\R^d)$, have to be adapted to the "more hostile" framework of the space $\Lp{\alpha}{\R^d}$. First we mention that:

\begin{rem}
\label{rem:Lalpha}
The space $\Lp{\alpha}{\R^d}$ is defined as the space of the Lebesgue measurable complex-valued  functions $g$ on~$\R^d$, such that
\begin{equation}
\norm{g}_{\Lp{\alpha}{\R^d}}:=\left(\int_{\R^d}\va{g(\xi)}^{\alpha}\,\mathrm d\xi\right)^{1/\alpha}<+\ii.
\end{equation}
When $\alpha\in[1,2]$, it is well-known that $\norm{\cdot}_{\Lp{\alpha}{\R^d}}$ is a norm on $\Lp{\alpha}{\R^d}$ confering to it the structure of a Banach space; the associated distance is 
\begin{equation}
\label{rem:Lalpha:eq2}
\Delta_{\alpha}(g_1,g_2):=\norm{g_1-g_2}_{\Lp{\alpha}{\R^d}}.
\end{equation}
When $\alpha\in(0,1)$, the definition of the distance $\Delta_{\alpha}$ has to be slightly modified since $\norm{\cdot}_{\Lp{\alpha}{\R^d}}$ is no longer a norm but only a quasi-norm~\footnote{The difference between a norm and a quasi-norm is that for a quasi-norm the triangle inequality is weakened to 
$
\norm{g+h}\le c\big(\norm{g}+\norm{h}\big),
$
where $c$ is a finite constant strictly bigger than $1$.}. More precisely, $\Delta_{\alpha}$ has to be defined as
\begin{equation}
\label{rem:Lalpha:eq1}
\Delta_{\alpha}(g_1,g_2):=\int_{\R^d}\va{g_1(\xi)-g_2(\xi)}^{\alpha}\,\mathrm d\xi,
\end{equation}
and then $\Lp{\alpha}{\R^d}$ equipped with this distance is a complete metric space. Observe that for any $\alpha\in (0,2]$, $\Delta_{\alpha}$ is invariant under translations, that is for all $g_1$, $g_2$, and $g_3$ in $\Lp{\alpha}{\R^d}$, one has
$
\Delta_{\alpha}(g_1+g_3,g_2+g_3)=\Delta_{\alpha}(g_1,g_2).
$
\end{rem}
%First, we intend to explain how the crucial equality \eqref{eq1:decomp} can be extended to the general case where $\alpha\in(0,2]$ is %arbitrary. To this end, 
Let us now come back to our goal. Rather than directly working with the functions $\wpsiJK$ (see \eqref{eq3:blm}), it is more convenient to work with their renormalized versions $\wh{\psi}_{\alpha,J,K}$ defined, for all $(J,K)\in\Z^d\times\Z^d$ and $\xi\in\R^d$, as
\begin{equation}
\label{eq4:blmalpha}
\wh{\psi}_{\alpha,J,K}(\xi):=2^{(j_1+\dots+j_d)(1/2-1/\alpha)}\,\wpsiJK(\xi)=\prod_{l=1}^d 2^{-j_l/\alpha}e^{-i2^{-j_l}k_l\xi_l}\,\wh{\psi^1}(2^{-j_l}\xi_l);
\end{equation}
it is clear that, similarly to $\wpsiJK$, the function $\wh{\psi}_{\alpha,J,K}$ is $C^\infty$ on $\R^d$ with a compact support satisfying
\begin{equation}
\label{eq6:blmalpha}
 \sp\,\wh{\psi}_{\text{$\alpha,J,K$}}\subset\left\{\xi\in\R^d : \mbox{{\small for all $l=1,\ldots, d$ one has}}\,\, \frac{2^{\text{$j_l$}+1}\pi}{3}\leq\text{$\va{\xi_l}$}\leq\frac{2^{\text{$j_l$}+3}\pi}{3}\right\}.
\end{equation}
The advantage offered by this renormalization is that the (quasi)-norm $\big\|\wh{\psi}_{\alpha,J,K}\big\|_{\Lp{\alpha}{\R^d}}$ does not depend on $(J,K)$, in other words,
\begin{equation}
\label{eq5:blmalpha}
\big\|\wh{\psi}_{\alpha,J,K}\big\|_{\Lp{\alpha}{\R^d}}=\big\|\wh{\psi}_{\alpha,0,0}\big\|_{\Lp{\alpha}{\R^d}}=\big\|\wh{\psi^1}\big\|_{\Lp{\alpha}{\R}}^d.
\end{equation}
Therefore, the real-valued symmetric $\alpha$-stable random variables $\epsalphaJK$ defined, for all $(J,K)\in\Z^d\times\Z^d$, as
\begin{equation}
\label{eq:eps-stable}
\epsalphaJK:=\Ree\left\{\int_{\R^d}\overline{\wpsialphaJK(\xi)}\,\mathrm d\wt{M}_{\alpha}(\xi)\right\},
\end{equation}
have the same distribution.

The function $\PsialphaJ$ denotes the renormalized version of $\PsiJ$ (see \eqref{eq:psiJ}), such that, for all $x\in\R^d$,
\begin{equation}
\label{psialpha}
\PsialphaJ(x)=2^{(j_1+\dots+j_d)(1/\alpha-1/2)}\PsiJ(x)=2^{(j_1+\dots+j_d)/\alpha}\int_{\R^d}e^{ix\cdot \xi} f(2^{J}\xi)\wpsi_{0,0}(\xi)\mathrm d\xi.
\end{equation}
In view of \eqref{eq4:blmalpha} and \eqref{psialpha}, it can easily be seen that, for every $(J,K)\in\Z^d\times\Z^d$ and $(t,\xi)\in\R^d\times\R^d$, one has
\begin{equation}
\label{normaliz}
\big(\PsiJ\left(2^Jt-K\right)-\PsiJ\left(-K\right)\big)\overline{\wpsiJK(\xi)}=\big (\PsialphaJ\left(2^Jt-K\right)-\PsialphaJ\left(-K\right)\big)\overline{\wpsialphaJK(\xi)}.
\end{equation}
%
%
%We define $s_{\alpha,J,K}$, for all $t\in\R^d$, as:
%\begin{equation}
%\label{eq:wavs}
%s_{\alpha,J,K}(t)=2^{-(j_1+\dots+j_d)(1/2-1/\alpha)}\int_{\R^d}\left(e^{it\cdot\xi}-1\right)f(\xi)\wpsialphaJK(\xi)\,\mathrm{d}\xi.
%\end{equation}
%Standard computations easily allow to derive that 
%\begin{equation}
%s_{\alpha,J,K}(t)=\PsialphaJ\left(2^Jt-K\right)-\PsialphaJ\left(-K\right),
%\end{equation} 
%where, $2^Jt:=(2^{j_1}t_1,\dots,2^{j_d}t_d)$ and, for all $x\in\R^d$,
%\begin{equation}
%\label{eq:psiJ}
%\PsialphaJ(x)=2^{(j_1+\dots+j_d)/\alpha}\int_{\R^d}e^{ix\cdot \xi}f\left(2^J\xi\right)\wpsi_{0,0}(\xi)\mathrm d\xi.
%%\PsialphaJ(x)=2^{(j_1+\dots+j_d)/2}\int_{\R^d}e^{ix\cdot \xi}\wf(2^J\xi)\wpsi(\xi_1)\dots\wpsi(\xi_d)\mathrm d\xi.
%\end{equation}
The following proposition explains, in a precise way, how the crucial equality \eqref{eq1:decomp} can be extended to the general case where $\alpha\in(0,2]$ is arbitrary. 
\begin{propo}
\label{prop:decompo}
Assume that $f$ is admissible in the sense of Definition~\ref{def:adm}, and denote by $F$ the function defined, for all $(t,\xi)\in\R^d\times\R^d$, as,
\begin{equation}
\label{prop:decompo:eq2}
F(t,\xi):=(e^{it\cdot\xi}-1\big)f(\xi).
\end{equation}
Let $(\mathcal{D}_n)_{n\in\N}$ be an arbitrary increasing (in the sense of the inclusion) sequence of finite subsets of~$\Z^d\times\Z^d$ which satisfies $\bigcup_{n\in\N}\mathcal{D}_n=\Z^d\times\Z^d.$ Then, for every fixed $ t\in\R^d$, one has
\begin{equation}
\label{prop:decompo:eq1}
\lim_{n\to+\ii}\Delta_{\alpha}\left(\sum_{(J,K)\in\mathcal{D}_n}\big(\PsialphaJ(2^Jt-K)-\PsialphaJ(-K)\big)\overline{\wh{\psi}_{\alpha,J,K}(\cdot)},F(t,\cdot)\right)=0,
\end{equation}
where $\PsialphaJ$ and $\wh{\psi}_{\alpha,J,K}$ are as in \eqref{psialpha} and \eqref{eq4:blmalpha}.
%We recall that $\Delta_{\alpha}$ is the distance on $\Lp{\alpha}{\R^d}$ defined through~\eqref{rem:Lalpha:eq2} when $\alpha\in [1,2]$ and %through \eqref{rem:Lalpha:eq1} when  $\alpha\in (0,1)$.
\end{propo}

The following proposition is a straightforward consequence of Proposition~\ref{prop:decompo}, Remark~\ref{rem:Lalpha}, \eqref{eq:isom-stable}, \eqref{deffield} and~\eqref{eq:eps-stable}. In some sense, it shows that similarly to the Gaussian case (see \eqref{eq:add2}), a wavelet type random series representation of the field $\{X(t), t\in\R^d\}$ can be obtained in the general case where $\alpha\in (0,2]$ is arbitrary.
\begin{propo}
\label{prop:wavrep}
Assume that $t\in\R^d$ is arbitrary and fixed. Let $X(t)$ be the real-valued symmetric $\alpha$-stable random variable defined through~\eqref{deffield}, where $f$ is supposed to be any admissible function in the sense of Definition~\ref{def:adm}.
Denote by  $(\mathcal{D}_n)_{n\in\N}$ an arbitrary increasing sequence of finite subsets of~$\Z^d\times\Z^d$ which satisfies $\bigcup_{n\in\N}\mathcal{D}_n=\Z^d\times\Z^d.$ For every fixed $n\in\N$, let $X_n^{\mathcal{D}}(t)$ be the real-valued symmetric $\alpha$-stable random variable defined as 
\begin{equation}
\label{thm:wavrep:eq1}
X_n^{\mathcal{D}}(t):=\sum_{(J,K)\in\mathcal{D}_n}\big(\PsialphaJ(2^Jt-K)-\PsialphaJ(-K)\big)\epsalphaJK,
\end{equation}
where $\PsialphaJ$ and $\epsalphaJK$ are as in \eqref{psialpha} and \eqref{eq:eps-stable}. Then, the sequence $(X_n^{\mathcal{D}}(t))_{n\in\N}$ converges in probability to $X(t)$. 
\end{propo}

Proposition~\ref{prop:decompo} is proved in the appendix~\ref{app:decompo}; we mention that the three main ingredients of its proof are the following two lemmas and Proposition~\ref{prop:psij} given below.
\begin{lemme}
\label{rem:conv}
Let $\alpha\in(0,2]$ be arbitrary and fixed. Assume that $(g_i)_{i\in\Z^d\times\Z^d}$ is a sequence of functions of $\Lp{\alpha}{\R^d}$ which satisfies,
\begin{equation}
\label{app:decompo:eq1}
\sum_{i\in\Z^d\times\Z^d}\Delta_{\alpha}\left({g_i,0}\right)<+\ii.
\end{equation}
Then there exists a function $g\in\Lp{\alpha}{\R^d}$ such that one has,
\begin{equation}
\label{app:decompo:eq2}
\lim_{n\to+\infty}\Delta_{\alpha}\left({\sum_{i\in\mathcal{D}_n}g_i,g}\right)=0,
\end{equation}
where $(\mathcal{D}_n)_{n\in\N}$ denotes any arbitrary increasing sequence of finite subsets of~$\Z^d\times\Z^d$ satisfying $\bigcup_{n\in\N}\mathcal{D}_n=\Z^d\times\Z^d$; observe that $g$ does not depend on the choice of this sequence of subsets.
\end{lemme}

The proof of Lemma~\ref{rem:conv} is rather classical; it mainly relies on the completeness of $L^\alpha (\R^d)$, the triangle inequality and the fact that the distance $\Delta_{\alpha}$ is invariant under translations. It does not present major difficulties, this is why it has been omitted.
\begin{lemme}
\label{le:majjj}
Assume that the real numbers $a'\in(0,1)$, $\alpha\in(0,2]$, and $\delta>0$ are arbitrary and fixed. Then,
for all fixed $r\in\{1,\dots,d\}$, one has
\begin{equation}
\label{le:majjj:eq1}
\sum_{J\in\Z_+^d}2^{-j_r(1-a')}\left(2^{-j_1}+\dots+2^{-j_d}\right)^{-d/\alpha}\prod_{l=1}^d2^{-j_l/\alpha}\logr{j_l}(1+j_l)^{1/\alpha+\delta}<+\ii;
\end{equation}
which clearly implies that
\begin{equation}
\label{le:majjj:eq2}
\sum_{J\in\Z_+^d}2^{-j_r(1-a')}\left(2^{-j_1}+\dots+2^{-j_d}\right)^{-d/\alpha}\prod_{l=1}^d2^{-j_l/\alpha}(1+j_l)^{1/\alpha+\delta}<+\ii
\end{equation}
and 
\begin{equation}
\label{le:majjj:eq2bis}
\sum_{J\in\Z_+^d}2^{-j_r(1-a')}\left(2^{-j_1}+\dots+2^{-j_d}\right)^{-d/\alpha}\prod_{l=1}^d2^{-j_l/\alpha}\logr{j_l}<+\ii.
\end{equation}
\end{lemme}

Lemma~\ref{le:majjj} is proved in the appendix~\ref{app:decompo}. 

For later purposes, we denote by $\Upsilon$ and $\Upss$ the two sets defined as,
\begin{equation}
\label{def:upss}
\Upsilon:=\{0,1\}^d \text{ and } \Upss:=\{0,1\}^d\setminus\{(0,\dots,0)\}.
\end{equation}
Also, for any fixed $\eta=(\eta_1,\ldots,\eta_d)\in\Upsilon$, we denote by $\Geta$ the subset of $\Z^d$ defined as the Cartesian product
\begin{equation}
\label{def:geta:eq1}
\Geta:=\prod_{l=1}^d\Z_{\eta_l},
\end{equation}
where
\begin{equation}
\label{def:geta:eq2}
\Z_{1}:=\N=\{1,2,\dots\} \quad\text{and}\quad\Z_0:=\Z_{-}=\{\dots,-2,-1,0\}.
\end{equation}
Notice that
\begin{equation}
\label{prop:geta}
 \Z^d=\bigcup_{\eta\in\Upsilon}\Z^d_{(\eta)},  \text{ \,\,\,and\,\,\, } Z^d_{(\eta)}\cap\Z^d_{(\eta')}=\emptyset \text{ \,\,\,when } \eta\neq\eta'.
\end{equation}

\begin{propo}
\label{prop:psij}
For all $J\in\Z^d$, let $\PsialphaJ$ be the function defined through \eqref{psialpha}, where $f$ is any admissible function in the sense of Definition~\ref{def:adm}. Then $\PsialphaJ$ is infinitely differentiable on $\R^d$. Also, its partial derivatives are such that, for all $b\in\Z_+^d$ and $x\in\R^d$,
\begin{equation}
\label{prop:psij:eq0}
\partial^{b}\PsialphaJ(x)=2^{(j_1+\dots+j_d)/\alpha}\,i^{\mathrm{l}(b)}\int_{\R^d}e^{ix\cdot \xi}\,\xi^b f(2^{J}\xi)\wpsi_{0,0}(\xi)\mathrm d\xi,
\end{equation}
where $\xi^{b}:=\prod_{l=1}^d\xi_l^{b_l}$ and $\mathrm{l}(b):=\sum_{l=1}^db_l$ is the length of $b$. Moreover, the $\partial^{b}\PsialphaJ$'s, $b\in\Z_+^d$, are well-localized functions, in the sense that they satisfy the following two properties, where $p_*$ is as in \eqref{eq:bs}.
\begin{enumerate}[(i)]
\item For each $T>0$, and  $b\in\Z_+^d$, there is a positive constant $c$, such that for all $J\in\Z_+^d$, and $x=(x_1,\dots,x_d)\in\R^d$,
\begin{equation}
\label{prop:psij:eq1}
\va{\partial^{b}\Psi_{\alpha,-J}(x)}\leq c\,\frac{\left(2^{-j_1}+\dots+2^{-j_d}\right)^{-a'-d/\alpha}\prod_{l=1}^d{2^{-j_l/\alpha}}}{\prod_{l=1}^d\left(1+T+\va{x_l}\right)^{p_*}},
\end{equation}
where the exponent $a'\in(0,1)$ and $\Lalpha$ are as in Definition~\ref{def:adm}.
\item 
For every $T>0$, $\eta\in\Upss$ (see \eqref{def:upss}), and $b\in\Z_+^d$, there exists a positive constant $c$, such that for every $J\in\Geta$ (see \eqref{def:geta:eq1} and \eqref{def:geta:eq2}), and $x=(x_1,\dots,x_d)\in\R^d$,
\begin{equation}
\label{prop:psij:eq2}
\va{\partial^{b}\PsialphaJ(x)}\leq c\prod_{l=1}^d\frac{2^{(1-\eta_l)j_l/\alpha}\,2^{-j_l\eta_l a_l}}{\left(1+T+\va{x_l}\right)^{p_*}},
\end{equation}
where the positive exponents $a_1,\dots,a_d,$ and $\Lalpha$ are as in Definition~\ref{def:adm}.
\end{enumerate}
\end{propo}

Proposition~\ref{prop:psij} is proved in the appendix~\ref{app:psijalpha}.

Having presented the main ingredients of the proof of the important Proposition~\ref{prop:wavrep} which provides the wavelet type random series representation of $\{X(t), t\in\R^d\}$, our present goal is to improve the convergence result concerning this series. First we need to give two useful lemmas. The following one will play a crucial role throughout the rest of the article.
\begin{lemme}
\label{le:eps-JK}
Let $\big\{\epsalphaJK:(J,K)\in \Z^d\times\Z^d\big\}$ be the sequence of the identically distributed real-valued symmetric $\alpha$-stable random variables defined through \eqref{eq:eps-stable}. There exists an event $\Om_1^*$ of probability~1 such that the following three results hold.
\begin{enumerate}
\item Assume that $\alpha\in(0,1)$; then, for all fixed $\delta\in (0,+\infty)$ and $\om\in\Om_1^*$, there is a finite constant $C(\om)>0$ (depending on $\alpha$, $\delta$ and $\om$),  such that, for every $J=(j_1,\dots,j_d)\in\Z^d$ and $K\in\Z^d$, one has
\begin{equation}
\label{le:eps-JK:eq1}
\va{\epsalphaJK(\om)}\leq C(\om)\prod_{l=1}^d(1+\va{j_l})^{1/\alpha+\delta}.
\end{equation}
Observe that in this case $\va{\epsalphaJK(\om)}$ can be bounded independently of $K$.
\item Assume that $\alpha\in [1,2)$; then, for each fixed $\delta\in (0,+\infty)$ and $\om\in\Om_1^*$, there exists a finite constant $C(\om)>0$ (depending on $\alpha$, $\delta$ and $\om$), such that for all $(J,K)=(j_1,\dots,j_d,k_1,\dots,k_d)\in\Z^d\times\Z^d$,
\begin{equation}
\label{le:eps-JK:eq2}
\va{\epsalphaJK(\om)}\leq C(\om)\logr{\sum_{l=1}^d\big ({\va{j_l}+\va{k_l}\big)}}\prod_{l=1}^d(1+\va{j_l})^{1/\alpha+\delta}.
\end{equation}
\item Assume that $\alpha=2$, then, for every fixed $\om\in\Om_1^*$, there is  a finite constant $C(\om)>0$ (depending on $\om$), such that for each $(J,K)=(j_1,\dots,j_d,k_1,\dots,k_d)\in\Z^d\times\Z^d,$
\begin{equation}
\label{le:eps-JK:eq3}
\va{\epsalphaJK(\om)}\leq C(\om)\logr{\sum_{l=1}^d\big ({\va{j_l}+\va{k_l}}\big)}.
\end{equation}
\end{enumerate}
Notice that the event $\Om_1^*$ depends on $\alpha$; yet, it does not depend on the function $f$ associated with the field $X$ through~\eqref{deffield}.  
%where for all non-negative real numbers $x$ and $y$, the function $\majeps{\cdot,\cdot}$ is defined as,
%\begin{equation}
%\label{le:eps-JK:eq2}
%\left\{
%    \begin{array}{r @{\,\,\,=\,\,\,} c @{\,\,\,\text{ if }\,\,\,} l}
%    \majeps{x,y} & \left(1+x\right)^{1/\alpha+\eta}            & \alpha\in(0,1), \\ 
%    \majeps{x,y} & \left(1+x\right)^{1/\alpha+\eta}\logr{x+y} & \alpha\in[1,2), \\
%    \majeps{x,y} & \logr{x+y} & \alpha=2.
%    \end{array}
%  \right.
%\end{equation}
%\begin{enumerate} 
%\item if $\alpha\in(0,1)$, then for all $\om\in\Om_1^{*}$, one has $C(\om)>0$, and for every $(J,K)=(j_1,\dots,j_d,k_1,\dots,k_d)\in\Z^d\times\Z^d,$
%\begin{equation}
%\label{le:eps-JK:eq1}
%\va{\epsalphaJK(\om)}\leq C(\om)\left(1+\sum_{l=1}^d\va{j_l}\right)^{1/\alpha+\eta}.
%\end{equation}
%\item if $\alpha\in[1,2)$, then for all $\om\in\Om_1^{*}$, one has $C(\om)>0$, and for every $(J,K)=(j_1,\dots,j_d,k_1,\dots,k_d)\in\Z^d\times\Z^d,$
%\begin{equation}
%\label{le:eps-JK:eq2}
%\va{\epsalphaJK(\om)}\leq C(\om)\left(1+\sum_{l=1}^d\va{j_l}\right)^{1/\alpha+\eta}\logr{\sum_{l=1}^d\va{j_l}+\va{k_l}}.
%\end{equation}
%\end{enumerate}
\end{lemme}
The third result provided by Lemma~\ref{le:eps-JK} (in other words the inequality \eqref{le:eps-JK:eq3} which holds in the Gaussian case $\alpha=2$) is rather classical; its proof can be found in e.g.~\cite{AT03}. The first two results provided by the lemma (in other words the inequalities \eqref{le:eps-JK:eq1} and \eqref{le:eps-JK:eq2}) are derived  in the appendix~\ref{app:lepage}; we mention that their proofs rely on a LePage series representation of the complex-valued $\alpha$-stable process
\[
\bigg\{\int_{\R^d}\overline{\wpsialphaJK(\xi)}\,\mathrm d\wt{M}_{\alpha}(\xi): (J,K)\in\Z^d\times\Z^d\bigg\}.
\]
On the other hand, it is worth noticing that the elementary inequality
\begin{equation}
\label{subadd}
\text{for all $u',u''\in\R_+$,}\quad\logr{u'+u''}\leq2\logr{u'}\logr{u''},
\end{equation}
will frequently be employed for deriving upper bounds of the logarithmic function in Lemma~\ref{le:eps-JK}.  In particular it allows to show that:
\begin{rem}
\label{re1:eps-JK}
Assume that $\alpha\in (0,2]$ is arbitrary. Let $\big\{\epsalphaJK:(J,K)\in \Z^d\times\Z^d\big\}$ and $\Om_1^*$ be as in Lemma~\ref{le:eps-JK}. Then, for each fixed $\delta\in (0,+\infty)$ and $\om\in\Om_1^*$, there exists a finite constant $C(\om)>0$ (depending on $\alpha$, $\delta$ and $\om$), such that for all $(J,K)=(j_1,\dots,j_d,k_1,\dots,k_d)\in\Z^d\times\Z^d$,
\begin{equation}
\label{re1:eps-JK:eq1}
\va{\epsalphaJK(\om)}\leq C(\om)\prod_{l=1}^d\logr{\va{j_l}}(1+\va{j_l})^{1/\alpha+\delta}\logr{\va{k_l}}.
\end{equation}
\end{rem}

The second useful lemma is the following one:
\begin{lemme}
\label{le:int1}
Assume that $\alpha\in (0,2]$ is arbitrary, and let $p_*=p_*(\alpha)$ be as in \eqref{eq:bs}. Then, there is a positive finite constant $c$ such that, for every $(\tea,v)\in\R_+\times\R$, the following inequality holds:
\begin{equation}
\label{le:int1:eq1}
\sum_{k\in\Z}\frac{\logr{\tea+\va{k}}}{\left(2+\va{v-k}\right)^{p_*}}\leq c \logr{\tea+\va{v}}.
\end{equation}
\end{lemme}
Lemma~\ref{le:int1} is proved at the end of the appendix~\ref{app:psijalpha}. 

The following proposition is an improvement of Proposition~\ref{prop:wavrep}.

% \begin{eqnarray}
%\label{subadd}
%\majeps{x'+x'',y'}&\leq&2\majeps{x',y'}\majeps{x'',y'}\\
%\majeps{x',y'+y''}&\leq&2\majeps{x',y'}\majeps{x',y''}
%\end{eqnarray}
%In the sequel, we denote by $ Y(t) $ the series associated with $ (X_n^{\mathcal{D}}(t))_{n\in\N} $. That is,
%\begin{equation}
%\label{thm:wavrep:eq2}
%Y(t):=\sum_{(J,K)\in\Z^d\times\Z^d}\big(\PsialphaJ(2^Jt-K)-\PsialphaJ(-K)\big)\epsalphaJK(\om).
%\end{equation}

\begin{propo}
\label{thm:wavrep}
We assume that the stability parameter $\alpha\in (0,2]$ is arbitrary, and that $\Om_1^*$ is the event of probability~1 introduced in Lemma~\ref{le:eps-JK}. Then, for all $(t,\om)\in\R^d\times\Om_1^*$, the series of real numbers,
\begin{equation}
\label{thm:wavrep:eq3}
\sum_{(J,K)\in\Z^d\times\Z^d}\big(\PsialphaJ(2^Jt-K)-\PsialphaJ(-K)\big)\epsalphaJK(\om),
\end{equation}
is absolutely convergent~\footnote{Therefore, its finite value does not depend on the way the terms of the series are labelled. Moreover, it follows from the Fubini's theorem that:
\[
\sum_{(J,K)\in\Z^d\times\Z^d}\big(\PsialphaJ(2^Jt-K)-\PsialphaJ(-K)\big)\epsalphaJK(\om)=\sum_{J\in\Z^d}\Big (\sum_{K\in\Z^d}\big(\PsialphaJ(2^Jt-K)-\PsialphaJ(-K)\big)\epsalphaJK(\om)\Big).
\]
}. Thus, in view of Proposition~\ref{prop:wavrep}, the sum in \eqref{thm:wavrep:eq3} is equal to $X(t,\om)$ defined through~\eqref{deffield}, except when $\om$ belongs to a negligible event~\footnote{Notice that this negligible event does not necessarily coincide with the whole set $\Om\setminus\Om_1^*$. On the other hand, this negligible event may depend on $t$.}.
\end{propo}
Before proving Proposition~\ref{thm:wavrep}, we introduce a convenient notation. Let $T$ be any fixed positive real number and let $g$ be any real-valued (or complex-valued) function on~$\R^d$, then the quantity $\|g\|_{T,\infty}$ is defined as:
\[
\|g\|_{T,\infty}:=\sup_{s\in [-T,T]^d} |g(s)|;
\]
observe that $\norm{\cdot}_{T,\infty}$ is almost the uniform semi-norm on the cube $[-T,T]^d$; the only difference is that  one may have $\|g\|_{T,\infty}=+\infty$, since one does not necessarily impose $g$ to be bounded on $[-T,T]^d$. 

\begin{proof}[Proof of Proposition~\ref{thm:wavrep}]
 We assume that $(t,\om)\in\R^d\times\Om_1^*$ is arbitrary and fixed. We have to prove that the series of real numbers in \eqref{thm:wavrep:eq3} is absolutely convergent, that is
\begin{equation}
\label{thm:wavrep:proof:eq0}
Z(t,\om)<+\ii,
\end{equation}
where
\begin{equation}
\label{thm:wavrep:proof:eq1}
Z(t,\om):=\sum_{(J,K)\in\Z^d\times\Z^d}\va{\PsialphaJ(2^Jt-K)-\PsialphaJ(-K)}\va{\epsalphaJK(\om)}.
\end{equation}
Let $\Upsilon$ be as \eqref{def:upss}, and, for each fixed $\eta\in\Upsilon$, let $\Geta$ be as in \eqref{def:geta:eq1} (see also \eqref{def:geta:eq2}). Then, it follows from \eqref{prop:geta} and \eqref{thm:wavrep:proof:eq1} that $Z(t,\om)$ can be decomposed as:
\begin{equation}
\label{thm:wavrep:proof:eq1bis}
Z(t,\om)=\sum_{\eta\in\Upsilon}Z^{\eta}(t,\om),
\end{equation}
where, for all fixed $\eta\in\Upsilon$,
\begin{equation}
\label{thm:wavrep:proof:eq2}
Z^{\eta}(t,\om):=\sum_{(J,K)\in\Geta\times\Z^d}\va{\PsialphaJ(2^Jt-K)-\PsialphaJ(-K)}\va{\epsalphaJK(\om)}.
\end{equation}
Next, using \eqref{thm:wavrep:proof:eq1bis} and the fact that $\Upsilon$ is a finite set, it turns out that~\eqref{thm:wavrep:proof:eq0} is equivalent to:
\begin{equation}
\label{thm:wavrep:proof:eq2bis}
Z^{\eta}(t,\om)<+\ii, \quad\text{for all $\eta\in\Upsilon$.}
\end{equation}
In order to prove \eqref{thm:wavrep:proof:eq2bis}, we will study two cases: $\eta=0:=(0,\dots,0)$ and $\eta\in\Upss:=\Upsilon\setminus\{0\}$.
%\bigbreak
%\vspace{0.5cm}

\noindent\underline{First case: $\eta=0$}. Notice that, in this case, one has $J\in\Z^d_{(0)}:=\Z^d_{-}$, so it can be rewritten as $J=-J'$, where $J'$ belongs to $\Z_+^d$. In the sequel $J'$ is denoted by $J$. Using the Mean Value Theorem and the triangle inequality, we get 
\begin{equation}
\label{thm:wavrep:proof:eqA}
\va{\Psi_{\alpha,-J}(2^{-J}t-K)-\Psi_{\alpha,-J}(-K)}\leq T\sum_{r=1}^d2^{-j_r}\norminf{\frac{\partial\Psi_{\alpha,-J}}{\partial x_r}\big(2^{-J}\cdot-K\big)},
\end{equation}
where $T:=\max_{1\leq l\leq d}{\va{t_l}}$, the $t_l$'s being the coordinates of $t$. 
% it follows from inequalities \eqref{le:eps-JK:eq1} and \eqref{prop:psij:eq1}, for all $(J,K)\in\Z_+^d\times\Z^d$, there exists $\theta_{J,K}(t):=(\theta_1,\dots,\theta_d)\in(-T,T)$ such that:
%\begin{eqnarray}
%\nonumber&&\va{\Psi_{-J}(2^{-J}t-K)-\Psi_{-J}(-K)}\va{\eps_{-J,K}(\om)}\\
%&&\nonumber\leq C_1(\om) T\sum_{r=1}^d2^{-j_r}\va{\frac{\partial\Psi_{-J}}{\partial x_r}\big(2^{-J}\theta_{J,K}(t)-K\big)}\prod_{s=1}^d\logr{\va{j_s}+\va{k_s}}\\
%&&\leq C_3(\om)\sum_{r=1}^d2^{-j_r}\left(\sum_{l=1}^d2^{-j_l}\right)^{{-a'-d/2}}\prod_{s=1}^d\frac{{2^{-j_s/2}\logr{\va{j_s}+\va{k_s}}}}{\left(1+T+\va{2^{-j_s}\theta_s-k_s}\right)^2}\label{thm:wavrep:proof:eq1bis},
%\end{eqnarray}
%where $C_3(\om):=Tc_2C_1(\om)$: $C_1(\om)$ being the constant in \eqref{le:eps-JK:eq1}, and $c_2$, the constant in \eqref{prop:psij:eq1}.
Moreover, combining \eqref{prop:psij:eq1} with the inequality, 
$$1+T+\va{2^{-j_r}s_l-k_l}\geq1+\va{k_l},\quad\mbox{for all $l\in\{1,\dots,d\}$ and $s_l\in [-T,T]$}, $$ 
we obtain, for every $r\in\{1,\dots,d\}$, that
\begin{equation}
\label{thm:wavrep:proof:eqB}
2^{-j_r}\norminf{\frac{\partial\Psi_{\alpha,-J}}{\partial x_r}\big(2^{-J}\cdot-K\big)}\leq c_1\frac{2^{-j_r(1-a')}\left(2^{-j_1}+\dots+2^{-j_d}\right)^{-d/\alpha}\prod_{l=1}^d2^{-j_l/\alpha}}{\prod_{l=1}^d\big(1+\va{k_l}\big)^{\Lalpha}},
\end{equation}
where $c_1$ is a positive finite constant not depending on $(J,K)$. Next, putting together \eqref{thm:wavrep:proof:eq2}, \eqref{thm:wavrep:proof:eqA}, \eqref{thm:wavrep:proof:eqB}, \eqref{eq:bs}, \eqref{re1:eps-JK:eq1}, and \eqref{le:majjj:eq1}, it follows that \eqref{thm:wavrep:proof:eq2bis} holds when $\eta=0$.
% and, the fact that for all $x\in[-T,T]$, $j>0$ and $k\in\Z$, 
%\begin{equation}
%\label{thm:wavrep:proof:eq1ter}
%1+T+\va{2^{-j}x-k}\geq 1+T+\va{k}-2^{-j}\va{x}\geq1+\va{k},
%\end{equation}
%implies that,
%\begin{equation}
%\label{thm:wavrep:proof:eq1terbis}
%\va{\PsialphaJ(2^Jt-K)-\PsialphaJ(-K)}\va{\epsalphaJK(\om)}\leq C_4(\om)\prod_{s=1}^d\frac{{2^{-j_s(1-a')/(2d)}\logr{\va{j_s}+\va{k_s}}}}{\left(1+\va{k_s}\right)^2},
%\end{equation}
%where $C_4(\om):=dC_3(\om)$ does not depend on $J$ and $K$. Therefore,
%\begin{equation*}
%Z^{\eta}(t,\om)\leq C_2(\om)\left(\sum_{j\in\Z_+}\sum_{k\in\Z}2^{-j(1-a')/(2d)}\frac{\logr{\va{j}+\va{k}}}{\left(1+\va{k}\right)^2}\right)^{d}<+\ii,
%\end{equation*}
%where $C_2(\om)$ is a positive finite constant.\\
%where $C_3(\om):=dc_1C_2(\om)$; $C_2(\om)$ being the constant $C(\om)$ in \eqref{le:eps-JK:eq1}.\\
%\bigbreak

\noindent\underline{Second case: $\eta\in\Upss$}. It results from \eqref{thm:wavrep:proof:eq2} and the triangle inequality that
\begin{equation*}
Z^{\eta}(t,\om)\leq \sum_{(J,K)\in\Geta\times\Z^d}\va{\PsialphaJ(2^Jt-K)}\va{\epsalphaJK(\om)}+\sum_{(J,K)\in\Geta\times\Z^d}\va{\PsialphaJ(-K)}\va{\epsalphaJK(\om)}.
\end{equation*}
Thus, in order to obtain \eqref{thm:wavrep:proof:eq2bis}, it is enough to show that,
\begin{eqnarray}
\label{thm:wavrep:proof:eq1qua}&\displaystyle\sum_{(J,K)\in\Geta\times\Z^d}\va{\PsialphaJ(2^Jt-K)}\va{\epsalphaJK(\om)}<+\ii,&\\
\nonumber\text{ and }&&\\
\label{thm:wavrep:proof:eq1cin}&\displaystyle\sum_{(J,K)\in\Geta\times\Z^d}\va{\PsialphaJ(-K)}\va{\epsalphaJK(\om)}<+\ii.&
\end{eqnarray}
Notice that \eqref{thm:wavrep:proof:eq1cin} is nothing else than \eqref{thm:wavrep:proof:eq1qua} where $t=0$. The proof of \eqref{thm:wavrep:proof:eq1qua} can be done in the following way.
%It is clear that \eqref{thm:wavrep:proof:eq1cin} is a straightforward consequence of \eqref{thm:wavrep:proof:eq1qua} (it is enough to take $t=(0,\dots,0)$ in what follows).
Using \eqref{re1:eps-JK:eq1}, \eqref{prop:psij:eq2} (with $T=1$), \eqref{le:int1:eq1} (with $(\tea,v)=(0,2^{j_l}t_l)$), \eqref{def:geta:eq1} and \eqref{def:geta:eq2}, one gets that,
\begin{eqnarray*}
&&\sum_{(J,K)\in\Geta\times\Z^d}\va{\PsialphaJ(2^Jt-K)}\va{\epsalphaJK(\om)}\\
&&\leq C_2(\om)\sum_{(J,K)\in\Geta\times\Z^d}\,\prod_{l=1}^d 2^{(1-\eta_l)j_l/2}2^{-j_l\eta_l a_l}\logr{\va{j_l}}(1+\va{j_l})^{1/\alpha+\delta}\frac{\sqrt{\log\big(3+\va{k_l}\big)}}{\big(2+\va{2^{j_l}t_l-k_l}\big)^{p_*}}\\
%&&\leq C_3(\om)\sum_{J\in\Geta}\prod_{l=1}^d 2^{(1-\eta_l)j_l/2}2^{-j_l\eta_l a_l}\logr{\va{j_l}}(1+\va{j_l})^{1/\alpha+\delta}\logr{2^{j_l}\va{t_l}}\\
&&\leq C_3 (\om)\prod_{l=1}^d\left(\sum_{j_l\in\Z_{\eta_l}} 2^{(1-\eta_l)j_l/2}2^{-j_l\eta_l a_l}\logr{\va{j_l}}(1+\va{j_l})^{1/\alpha+\delta}\logr{2^{j_l}\va{t_l}}\right) <+\ii,
\end{eqnarray*}
where $C_2(\om)$ and $C_3(\om)$ are two positive finite constants.
%where $C_5(\om):=c_4C_1(\om)\mathbf{c}_1(\va{t_1})\dots\mathbf{c}_1(\va{t_d})$; $c_4$ being the constant $c$ in \eqref{prop:psij:eq2}, and $\mathbf{c}_1(\cdot)$, the constant defined in Lemma~\ref{le:int1}.
%The proof of the fact that $\tilde{X}$ is a modification of $X$ is a straightforward consequence of the isometry property of the Wiener integral and the decomposition of the kernel $\xi\mapsto \left(e^{it\cdot\xi}-1\right)f(\xi)$ on the wavelet basis~$\big\{\overline{\Psi_{J,K}}: (J,K)\in\Z^d\times\Z^d\big\}$ for all $t\in\R^d$.
\end{proof}

\begin{rem}
\label{rem:deffield2}
From now on, for the sake of simplicity, "we forget" the definition of the real-valued symmetric $\alpha$-stable field $\{X(t), t\in\R^d\}$ given by \eqref{deffield}, and we systematically identify this field with its modification provided by Proposition~\ref{thm:wavrep}. More precisely, we assume that, for all $(t,\om)\in\R^d\times\Om_1^*$, one has
\begin{equation}
\label{deffield2}
X(t,\om):=\sum_{(J,K)\in\Z^d\times\Z^d}\big(\PsialphaJ(2^Jt-K)-\PsialphaJ(-K)\big)\epsalphaJK(\om);
\end{equation}
also, we assume that the field $X$ vanishes outside of the event $\Om_1^*$.
\end{rem}

Thanks to \eqref{deffield2}, for any $\eta\in\Upsilon$, the $\eta$-frequency part $\left\{X^{\eta}(t),t\in\R^d\right\}$ of the field $\{X(t), t\in\R^d\}$ can be precisely defined. 

\begin{defi}
\label{def:wavrepeta}
For all $\eta\in\Upsilon:=\{0,1\}^d$, the $\eta$-frequency part of the field $\{X(t), t\in\R^d\}$ is the real-valued symmetric $\alpha$-stable field denoted by $X^{\eta}:=\left\{X^{\eta}(t),t\in\R^d\right\}$, and defined, for any~$(t,\om)\in\R^d\times\Om_1 ^*$, as:
\begin{equation}
\label{wavrepeta}
X^{\eta}(t,\om):=\sum_{(J,K)\in\Geta\times\Z^d}\left(\PsialphaJ\big(2^Jt-K\big)-\PsialphaJ\big(-K\big)\right)\epsalphaJK(\om),
\end{equation}
where $\Geta$ is as in \eqref{def:geta:eq1} (see also \eqref{def:geta:eq2}); moreover, it is assumed that the field $X^\eta$ vanishes outside of the event~$\Om_1^*$. Notice that we know from \eqref{thm:wavrep:proof:eq2} and \eqref{thm:wavrep:proof:eq2bis} that the series of real numbers in \eqref{wavrepeta} is absolutely convergent.
\end{defi}

\begin{rem}
\label{rem:wavrepeta}
In view of Remark~\ref{rem:deffield2} and Definition~\ref{def:wavrepeta}, it is clear that the field $X$ can be expressed as the finite sum of all its $\eta$-frequency parts: for each~$(t,\om)\in\R^d\times\Om$ one has
\begin{equation}
\label{deffield3}
X(t,\om)=\sum_{\eta\in\Upsilon}X^{\eta}(t,\om).
\end{equation}
In some sense, the two extremes, that is the fields $X^0:=X^{(0,\dots,0)}$ and $X^1:=X^{(1,\dots,1)}$, can respectively be viewed as the low-frequency and high-frequency parts. While, for any $\eta\in\{0,1\}^d\setminus\{(0,\dots,0),(1,\dots,1)\}$, the field $X^{\eta}$ can be viewed as an intermediary part between low-frequency and high-frequency.
\end{rem}

\begin{rem} 
\label{re:split}
For the sake of convenience, when $\eta\neq 0$ and $(t,\om)\in\R^d\times\Om_1^*$, we sometimes decompose $X^{\eta}(t,\om)$  as:
\begin{equation}
\label{split}
\displaystyle X^{\eta}(t,\om)=Y^{\eta}(t,\om)-Y^{\eta}(0,\om),
\end{equation}
where,
\begin{equation}
\label{split2}
Y^{\eta}(t,\om):=\sum_{(J,K)\in\Geta\times\Z^d}\PsialphaJ(2^Jt-K)\epsalphaJK(\om).
\end{equation}
Notice that we know from \eqref{thm:wavrep:proof:eq1qua} that the series of real numbers in \eqref{split2} is absolutely convergent.
\end{rem}

%So far, we know that the series expansion of $X^{\eta}$ (see \eqref{wavrepeta}) and $X$ (see \eqref{deffield2}) converge in probability for each fixed %$t\in\R^d$, when $ \alpha\in(0,2) $, and converge absolutely for each fixed $t\in\R^d$, and $\om\in\Om_1^*$, when $\alpha=2$. Now, we are going to %see that they are, in some sense, almost surely uniformly convergent in $t$, on any compact subset of $\R^d$. Also, we will see that, sometimes, their term %by term partial derivatives share the same convergence property, more precisely, the following two propositions hold.

Now, we are going to study some smoothness properties of the sample paths of the $\eta$-frequency parts $X^\eta$ of the field $X$. Mainly, we will show that they are always continuous functions, and may even have partial derivatives in some cases; for instance, they are infinitely differentiable in the particular case of the low-frequency part $X^0$. Notice that, in view of \eqref{deffield3}, the continuity property of the $X^\eta$'s implies that the sample paths of $X$, itself, are continuous as well. 

More precisely, we will show that the following three propositions hold.

\begin{propo}
\label{thm:lfcvu}
For any $\alpha\in (0,2]$, for each $b=(b_1,\dots,b_d)\in\Z_+^d$, and for all $(T,\om)\in (0,+\ii)\times\Om_1^*$,
%there exists a random variable $A_1:=A_1(T)$ of finite moments of any order such that for every 
one has
\begin{equation}
\label{thm:lfcvu:eq1}
\sum_{(J,K)\in\Z_+^d\times\Z^d}\norm{\partial^{b}\!\left(\Psi_{\alpha,-J}(2^{-J}\cdot-K)-\Psi_{\alpha,-J}(-K)\right)}_{T,\ii}\va{\eps_{\alpha,,-J,K}(\om)} <+\ii.
\end{equation}
Thus, when $\eta=0$, the series in \eqref{wavrepeta} and all its term by term partial derivatives of any order are uniformly convergent in $t$, on each compact subset of $\R^d$. Therefore, the function $X^{0}(\cdot,\om):t\mapsto X^{0}(t,\om)$ is infinitely differentiable on $\R^d$, with partial derivatives satisfying, for all $b\in\Z_+^d$ and $t\in\R^d$,
\begin{equation}
\label{thm:lfcvu:eq2}
\big(\partial^bX^{0}\big)\!(t,\om)=\sum_{(J,K)\in\Z_+^d\times\Z^d}\partial^{b}\!\Big (\Psi_{\alpha,-J}(2^{-J}\cdot-K)-\Psi_{\alpha,-J}(-K)\Big)\!(t)\,\eps_{\alpha,-J,K}(\om) .
\end{equation}
\end{propo}

\begin{propo}
\label{pro:phi}
Let $\alpha\in (0,2]$, $\eta=(\eta_1,\dots,\eta_d)\in\Upss$, $J=(j_1,\dots, j_d)\in\Geta$, $b=(b_1,\dots,b_d)\in\Z_+^d$ and $(T,\om)\in (0,+\infty)\times\Om_1^*$ be arbitrary and fixed. One has
\begin{equation}
\label{le:phi:eq2}
\sum_{K\in\Z^d}\norminf{\left(\partial^{b}\PsialphaJ\right)(\cdot-K)}\va{\epsalphaJK(\om)}<+\ii.
\end{equation}
Thus, the series
% $x\in\R^d$,
\begin{equation}
\label{le:phi:eq1}
\PhialphaJ(x,\om):=\sum_{K\in\Z^d}\PsialphaJ(x-K)\epsalphaJK(\om),
\end{equation}
and all its term by term partial derivatives of any order are uniformly convergent in $x$, on each compact subset of $\R^d$. Therefore, 
the real-valued function $\PhialphaJ(\cdot,\om):x\mapsto\PhialphaJ(x,\om)$ is infinitely differentiable on $\R^d$, with partial derivatives satisfying, for all $b\in\Z_+^d$ and $x\in\R^d$,
\begin{equation}
\label{le:phi:eq3}
\left(\partial^b\PhialphaJ\right)\!(x,\om)=\sum_{K\in\Z^d}(\partial^b\PsialphaJ)(x-K)\epsalphaJK(\om).
\end{equation}
\end{propo}

\begin{propo}
\label{thm:etacvu}
Assume that $\eta=(\eta_1,\dots,\eta_d)\in\Upss$ and $b=(b_1,\dots,b_d)\in\Z_+^d$ satisfy
\begin{equation}
\label{thm:etacvu:eq1}
\eta_l b_l <a_l, \quad\text{for all $l\in\{1,\dots,d\}$,}
\end{equation}
where the positive exponents $a_1,\dots,a_d$ are as in Definition~\ref{def:adm}. Let $\alpha\in (0,2]$ and $(T,\om)\in (0,+\infty)\times\Om_1^*$ be arbitrary and fixed. Then, one has
\begin{equation}
\label{thm:etacvu:eq2}
\sum_{J\in\Geta}\norm{\partial^{b}\!\left(\PhialphaJ(2^J\cdot,\om)\right)}_{T,\ii} <+\infty.
\end{equation}
Thus, the series
$
\sum_{J\in\Geta} \PhialphaJ(2^Jt,\om),
$
and any of its term by term partial derivatives, of an order $b$ satisfying \eqref{thm:etacvu:eq1}, are uniformly convergent in $t$ on each compact subset of $\R^d$. Therefore,  the function $Y^{\eta}(\cdot,\om):t\mapsto Y^{\eta}(t,\om)$, defined on $\R^d$ through \eqref{split2}, is continuous and has a continuous partial derivative $\big(\partial^bY^{\eta}\big)\!(\cdot,\om)$ such that, for all $t\in\R^d$, 
\begin{equation}
\label{thm:etacvu:eq3}
\big(\partial^bY^{\eta}\big)\!(t,\om)=\sum_{J\in\Geta}\partial^{b}\!\left(\PhialphaJ(2^J\cdot,\om)\right)\!(t)=\sum_{J\in\Geta} 2^{j_1b_1+\ldots+ j_d b_d}\left(\partial^b\PhialphaJ\right)\!(2^J t,\om).
\end{equation}
Notice that, these continuity and differentiability properties are also satisfied by the function $X^{\eta}(\cdot,\om)$ (see Definition~\ref{def:wavrepeta}) because of the equality \eqref{split}.
\end{propo}

%First, we prove Proposition~\ref{thm:lfcvu}.
\begin{proof}[Proof of Proposition~\ref{thm:lfcvu}]
%Let $(T,\om,b)\in\R_+\times\Om_1^*\times\Z_+^d$ be fixed. Let $(T,\om)\in(0,\ii)\times\Om_{1}^*$ be fixed. 
We will study two cases: $b=0$ and $b\neq0$.\\
\underline{First case: $b=0$}. Similarly to \eqref{thm:wavrep:proof:eqA} and \eqref{thm:wavrep:proof:eqB}, we can show that, for some finite constant $c_1$ and for all $(J,K)\in\Z_+^d\times\Z^d$, one has
\begin{eqnarray}
\nonumber
\norminf{\Psi_{\alpha,-J}(2^{-J}\cdot-K)-\Psi_{\alpha,-J}(-K)}&\leq& T\sum_{r=1}^d2^{-j_r}\norminf{\frac{\partial\Psi_{\alpha,-J}}{\partial x_r}\big(2^{-J}\cdot-K\big)}\\
&\leq&\label{thm:lfcvu:proof:eq1} c_1\sum_{r=1}^d\frac{2^{-j_r(1-a')}\left(2^{-j_1}+\dots+2^{-j_d}\right)^{-d/\alpha}\prod_{l=1}^d2^{-j_l/\alpha}}{\prod_{l=1}^d\big(1+\va{k_l}\big)^{p_*}}.
\end{eqnarray}
%where $c_1$ is the constant $c$ in~\eqref{prop:psij:eq1}.
Next putting together \eqref{thm:lfcvu:proof:eq1}, \eqref{re1:eps-JK:eq1}, \eqref{eq:bs} and~\eqref{le:majjj:eq1}, we get \eqref{thm:lfcvu:eq1} when $b=0$.\\
%entail that for every $t\in[-T,T]^d$, \eqref{thm:wavrep:proof:eq1terbis} holds; thus we get that
%\begin{equation}
%\label{le:majphi:proof:eq0}
%\norminf{\PsialphaJ(2^Jt-K)-\PsialphaJ(-K)}\va{\epsalphaJK(\om)}\leq C_1(\om)\prod_{s=1}^d\frac{{2^{-\frac{j_s(1-a')}{2d}}\logr{\va{j_s}+\va{k_s}}}}{\left(1+\va{k_s}\right)^2},
%\end{equation}
%where $C_1(\om)$ is a positive constant which does not depend on $t$, $J$ and $K$. Notice that the function $x\mapsto\logr{x}$ is sub-additive, that is the inequality: for all positive real number $x$ and $y$,
%\begin{equation}
%\label{le:majphi:proof:eq00}
%\logr{x+y}\leq{\logr{x}+\logr{y}}.
%\end{equation}
%Therefore, combining \eqref{le:majphi:proof:eq0} and \eqref{le:majphi:proof:eq00}, we get \eqref{le:majphi:eq2}.
\underline{Second case: $b\neq0$}. Notice that in this case the multi-index $b$ has at least one positive coordinate, let us say $b_{r_0}$. Standard computations and~\eqref{prop:psij:eq1} allow to show that, for some finite constant $c_2$, and for all $(J,K)\in\Z_+^d\times\Z^d$, one has
\begin{eqnarray}
\label{le:majphi:proof:eq1} 
&& \norminf{\partial^{b}\!\big(\Psi_{-J}(2^{-J}\cdot-K)-\Psi_{-J}(-K)\big)}=\Big(\prod_{l=1}^d2^{-j_lb_l}\Big)\norminf{\partial^{b}\Psi_{-J}(2^{-J}\cdot-K)}\nonumber\\
&&\leq2^{-j_{r_0}}\norminf{\partial^{b}\Psi_{-J}(2^{-J}\cdot-K)}\leq c_2\frac{2^{-j_{r_0}(1-a')}\left(2^{-j_1}+\dots+2^{-j_d}\right)^{-d/2}\prod_{l=1}^d2^{-j_l/2}}{\prod_{l=1}^d\left(1+\va{k_l}\right)^{p_*}}.
\end{eqnarray}
%On the other hand, the fact that $J\in\Z_+^d$ implies that, for all $u\in[-T,T]$ and 
%$k\in\Z$,
%\begin{equation}
%\label{le:majphi:proof:eq2}
%1+T+|2^{-j}u-k|\geq 1+|k|.
%\end{equation}
Next putting together \eqref{le:majphi:proof:eq1}, \eqref{re1:eps-JK:eq1}, \eqref{eq:bs} and~\eqref{le:majjj:eq1}, we get \eqref{thm:lfcvu:eq1} when $b\ne 0$.
\end{proof}

\begin{proof}[Proof of Proposition~\ref{pro:phi}]
% and define for all $\beta\in(\Z_+)^d$, $T\in\N$ and $t\in R^d$,
%\begin{eqnarray*}
%\phi_J^{\beta}(t,\om)&:=&\sum_{K\in\Z^d}\partial^{\beta}\!\!\left(\PsialphaJ(2^J\cdot-K)\right)\!\!(t)\epsalphaJK(\om)\\
%\phi_{J,N,T}^{\beta}(t,\om)&:=&\sum_{K\in[-2^{N+1}T,2^{N+1}T]^d}\partial^{\beta}\!\!\left(\PsialphaJ(2^J\cdot-K)\right)\!\!(t)\epsalphaJK(\om),
%\end{eqnarray*}
%where $C_N:=[-2^{N+1}T,2^{N+1}T]^d$.
%Since the functions $\PsialphaJ$ are infinitely differentiable on $\R^d$ (see Proposition~\ref{prop:psij}), to show that the function $\Phi_J(\cdot,\om)$ is infinitely differentiable on $\R^d$, it is sufficient to prove that for all $T\in\N$ and $\beta\in\Z_+^d,$
%\begin{equation}
%\label{le:phi:proof:eq1}
%\lim_{N\to+\ii}\sup_{t\in[-T,T]^d}\left\{\va{\phi_{J,N,T}^{\beta}(t,\om)-\phi_{J}^{\beta}(t,\om)}\right\}=0.
%\end{equation}
It follows from~\eqref{re1:eps-JK:eq1}, \eqref{prop:psij:eq2}  and the triangle inequality that, for all $x=(x_1,\dots,x_d)\in[-T,T]^d$ and $K=(k_1,\dots,k_d)\in\Z_+^d$, one has
\begin{equation}
\label{le:phi:proof:eq2}
\va{\left(\partial^{b}\PsialphaJ\right)\!(x-K)}\va{\epsalphaJK(\om)}\leq C_1(\om,T, J)\prod_{l=1}^d\frac{\logr{\va{k_l}}}{\left(1+T+\va{x_l-k_l}\right)^{p_*}}\leq C_1(\om,T,J)\prod_{l=1}^d\frac{\logr{\va{k_l}}}{\left(1+\va{k_l}\right)^{p_*}},
\end{equation}
where $C_1(\om,T,J)$ is a finite constant depending on $T$ and $J$, but not on $K$. In view of \eqref{eq:bs}, it is clear that~\eqref{le:phi:proof:eq2} entails that~\eqref{le:phi:eq2} holds.
\end{proof}

In order to derive Proposition~\ref{thm:etacvu}, we need the following lemma.

\begin{lemme}
\label{le:phi}
Assume that $a_1,\dots,a_d$ are the same positive exponents as in Definition~\ref{def:adm}.
Let $\alpha\in (0,2]$, $\eta=(\eta_1,\dots,\eta_d)\in\Upss$, $J\in\Geta$,  $b=(b_1,\dots,b_d)\in\Z_+^d$ and $(T,\delta,\om)\in (0,+\infty)^2\times\Om_1^*$ be arbitrary and fixed. The following three results are satisfied; notice that $C(\om)$, in each one of them, is a finite constant not depending on $J$ and $T$.
\begin{enumerate}
\item When $\alpha\in (0,1)$, one has
\begin{equation}
\label{le:majphi:eq3}
\norminf{\partial^{b}\left(\PhialphaJ(2^J\cdot,\om)\right)}\leq C(\om)\prod_{l=1}^{d}
2^{j_l((1-\eta_l)(1/\alpha+b_l)-\eta_l(a_l-b_l))}\puiss{\va{j_l}}.
\end{equation}
\item When $\alpha\in[1,2)$, one has 
\begin{equation}
\label{le:majphi:eq3bis}
\norminf{\partial^{b}\left(\PhialphaJ(2^J\cdot,\om)\right)}\leq C(\om)\prod_{l=1}^{d}
2^{j_l((1-\eta_l)(1/\alpha+b_l)-\eta_l(a_l-b_l))}\puiss{\va{j_l}}\logr{\va{j_l}+2^{j_l}T}.
\end{equation}
\item When $\alpha=2$, one has
\begin{equation}
\label{le:majphi:eq3ter}
\norminf{\partial^{b}\left(\PhialphaJ(2^J\cdot,\om)\right)}\leq C(\om)\prod_{l=1}^{d}2^{j_l((1-\eta_l)(1/\alpha+b_l)-\eta_l(a_l-b_l))}\logr{\va{j_l}+2^{j_l}T}.
\end{equation}
\end{enumerate}
\end{lemme}

\begin{proof}[Proof of Lemma~\ref{le:phi}] We give the proof only in the case where $\alpha\in [1,2)$; the other two cases, $\alpha\in (0,1)$ and $\alpha=2$, can be treated similarly except that one has to use \eqref{le:eps-JK:eq1} and \eqref{le:eps-JK:eq3} instead of \eqref{le:eps-JK:eq2}. 
 It follows from \eqref{le:phi:eq3}, the triangle inequality, \eqref{prop:psij:eq2} (with $T=1$), \eqref{le:eps-JK:eq2}, \eqref{subadd} and \eqref{le:int1:eq1}, that, for every $t\in[-T,T]^d$ and $J\in\Geta$, one has
\begin{eqnarray*}
&&\va{\partial^b\big(\PhialphaJ\big(2^J\cdot,\om\big)\big)(t)}\\
&&\leq\sum_{K\in\Z^d}\va{\left(\prod_{l=1}^d2^{j_lb_l}\right)\!(\partial^{b}\PsialphaJ)(2^Jt-K)\epsalphaJK(\om)}\\
&&\leq C_1(\om)\prod_{l=1}^d2^{(1-\eta_l)j_l(1/\alpha+b_l)}2^{-\eta_lj_l (a_l-b_l)}\puiss{\va{j_l}}\sum_{k_l\in\Z}\frac{\logr{\va{j_l}+\va{k_l}}}{\left(2+\va{2^{j_l}t_l-k_l}\right)^{p_*}}\\
&&\leq C_2(\om)\prod_{l=1}^d2^{(1-\eta_l)j_l(1/2+b_l)}2^{-\eta_lj_l (a_l-b_l)}\puiss{\va{j_l}}\logr{\va{j_l}+2^{j_l}T},
\end{eqnarray*}
where $C_1(\om)$ and $C_2(\om)$ are two positive and finite constants not depending on $J$, $t$ and $T$.
%where $C_4(\om):=C_1(\om)c_2\mathbf{c}_1(T)^d$ is a positive constant which does not depend on $t$ and $J$; $\mathbf{c}_1(T)$ being the constant in \eqref{le:int1:eq2}.
\end{proof}

We are now ready to prove Proposition~\ref{thm:etacvu}.

\begin{proof}[Proof of Proposition~\ref{thm:etacvu}] 
Using Lemma~\ref{le:phi}, \eqref{def:geta:eq1}, \eqref{def:geta:eq2} and standard computations, one 
can easily obtain~\eqref{thm:etacvu:eq2}.
\end{proof}

Before ending this section let us state the following theorem which easily results from Remark~\ref{rem:wavrepeta}, Proposition~\ref{thm:lfcvu} and Proposition~\ref{thm:etacvu}.
\begin{thm}
\label{thm:cvu}
Assume that $f$ is an admissible function in the sense of Definition~\ref{def:adm}, and that the positive exponents $a_1,\ldots, a_d$ are as in this definition. Then, the field $X$ associated with $f$ (see \eqref{deffield} and Remark~\ref{rem:deffield2}) has the following property. For any fixed $\om\in\Om_1 ^*$ (see Lemma~\ref{le:eps-JK}), the sample path $X(\cdot,\om):t\mapsto X(t,\om)$ is continuous on $\R^d$; moreover, when $b=(b_1,\ldots, b_d)\in\Z_+^d$ satisfies $b_l<a_l$, for all $l\in\{1,\ldots, d\}$, the partial derivative $\big(\partial^bX\big)(\cdot,\om)$ exists and is continuous on~$\R^d$. 
\end{thm}
Last but not least, we point out that $\Om_1 ^*$ is an event of probability~$1$ not depending on $f$; so, in some sense, $\Om_1 ^*$ is "universal".

\section{Generalized directional increments on a compact cube}
\label{sechold}
Let $f$ be an admissible function, $X$ the field associated with $f$, and $X^{\eta}$ an arbitrary $\eta$-frequency part of $X$, where $\eta=(\eta_1,\ldots, \eta_d)\in\Upsilon:=\{0,1\}^d$ (see Definition~\ref{def:adm}, \eqref{deffield}, Definition~\ref{def:wavrepeta} and Remark~\ref{rem:wavrepeta}). The directional rates of vanishing at infinity of $f$ along the axes of $\R^d$ are governed by the positive exponents $a_1,\ldots, a_d$ through the inequality~\eqref{A2}. The main goal of the present section is to draw connections between these exponents and the anisotropic behaviour of the generalized directional increments of $X^\eta$ and $X$, on an arbitrary compact cube of $\R^d$. The methodology we use is based on the wavelet type random series representations~\eqref{wavrepeta} and \eqref{deffield2} of $X^{\eta}$ and $X$. It is worth mentioning that all the results we obtain are valid on $\Om_1^*$, the "universal" event of probability~$1$  which was introduced in Lemma~\ref{le:eps-JK}; we recall that "universal" means that $\Om_1^*$ does not depend on $f$. In order to precisely state our results, first, we need to introduce some notations.

For every fixed $k\in\{1,\dots,d\}$ and $h_k\in\R$, we denote by~$\Delta_{h_k}^k$, the operator from the space of the real-valued functions on $\R^d$, into itself; so that, when $g$ is such a function, $\Delta_{h_k}^k g$ is then the function defined, for all $x\in\R^d$, as
\begin{equation}
\label{def:acc:eq1}
\big(\Delta_{h_k}^k g\big)(x)= g(x+h_ke_k)-g(x),
\end{equation}
where $e_k$ denotes the vector of $\R^d$ whose $k$-th coordinate equals 1 and the others vanish. Clearly $\Delta_{h_k}^k g$ is at least as much regular as $g$ is; in particular, when $g$ belongs to the space $\Cn{\ii}{\R^d}$ of the infinitely differentiable real-valued functions defined on $\R^d$, then $\Delta_{h_k}^k g$ shares the same property. On the other hand, notice that the operators $\Delta_{h_k}^k$ are commutative, in the sense that, for all $(k,k')\in\{1,\dots,d\}^2$ and $(h_k,h_{k'}')\in\R^2$, one has
\begin{equation*}
\Delta_{h_{k'}'}^{k'}\circ\Delta_{h_{k}}^{k}=\Delta_{h_{k}}^{k}\circ\Delta_{h_{k'}'}^{k'},
\end{equation*}
where the symbol "$\circ$" denotes the usual composition of operators. For every $h=(h_1,\dots,h_d)\in\R^d$ and multi-index $B=(b_1,\dots,b_d)\in\Z_+^d$, we denote by $\Delta_{(h)}^{B}$, the operator from the space of the real-valued functions on $\R^d$ into itself, defined by
\begin{equation}
\label{def:acc:eq1bis}
\Delta_{(h)}^{B}:=\Delta_{h_1}^{1,b_1}\circ\dots\circ\Delta_{h_d}^{d,b_d},
\end{equation}
where, for all $k\in\{1,\dots,d\}$, $\Delta_{h_k}^{k,b_k}$ is $\Delta_{h_k}^{k}$ composed with itself $b_k$ times, with the convention that $\Delta_{h_k}^{k,0}$ is the identity. 

\begin{defi} 
\label{def1:ant}
$ $
\begin{itemize}
\item[(i)] We denote by $\L_2$ the function defined, for each $(a,b)\in\R_+^2$, as
\begin{equation}
\label{eq:ant-defL2}
\L_2(a,b):=1/2\,\ind{\{b \ge a\}} +\ind{\{b =a\}}.
\end{equation}
More precisely, one has: 
\[
\text{$\L_2(a,b)=0$ if $a>b$, $\L_2(a,b)=3/2$ if $a=b$, and $\L_2(a,b)=1/2$ if $a<b$.}
\]
\item[(ii)] For any fixed $\alpha\in (0,2)$, we denote by $\L_\alpha $ the function defined, for each $(a,b,\delta)\in\R_+^3$, as
\begin{equation}
\label{eq:ant-defL3}
\L_\alpha (a,b,\delta):=\big(1/\alpha+\lfloor\alpha\rfloor/2+\delta\big)\ind{\{b \ge a\}} +\ind{\{b =a\}},
\end{equation}
where $\lfloor\alpha\rfloor$ is the integer part of $\alpha$. More precisely, 
\begin{itemize}
\item when $\alpha\in (0,1)$, one has: 
\[
\text{$\L_\alpha (a,b,\delta)=0$ if $a>b$, $\L_\alpha(a,b,\delta)=1/\alpha+1+\delta$ if $a=b$, and $\L_\alpha(a,b,\delta)=1/\alpha+\delta$ if $a<b$;}
\]
\item when $\alpha\in [1,2)$, one has: 
\[
\text{$\L_\alpha (a,b,\delta)=0$ if $a>b$, $\L_\alpha(a,b,\delta)=1/\alpha+3/2+\delta$ if $a=b$, and $\L_\alpha (a,b,\delta)=1/\alpha+1/2+\delta$ if $a<b$.}
\]
\end{itemize}
\end{itemize}
\end{defi}

We are now ready to state the first main result of this section.
\begin{thm}
\label{thm:etaincr} 
The positive exponents $a_1,\ldots, a_d$ are the same as in Definition~\ref{def:adm}. 
Moreover we assume that $\eta=(\eta_1,\ldots,\eta_d)\in\Upsilon:=\{0,1\}^d$, $B=(b_1,\ldots,b_d)\in\Z_+^d$, $T\in (0,+\infty)$ and $\om\in\Om_1 ^*$ are arbitrary and fixed. Then, the following two results hold (with the convention that $0/0=0$).
\begin{itemize}
\item[(i)] When $\alpha=2$, one has 
\begin{equation}
\label{thm:etaincr:eq3bis}
\sup_{h\in[-T,T]^d}\left\{\frac{\norminf{\Delta_{(h)}^BX^{\eta}(\cdot,\om)}} {\displaystyle\prod_{l=1}^d{\va{h_l}^{b_l(1-\eta_l)}\va{h_l}^{\min(b_l,a_l)\eta_l}}\left(\log\left(3+\va{h_l}^{-1}\right)\right)^{\eta_l\L_2(a_l,b_l)}}\right\}<+\infty.
\end{equation}
\item[(ii)]  When $\alpha\in (0,2)$, for all arbitrarily small positive real numbers $\delta$, one has  
\begin{equation}
\label{thm:etaincr:eq3}
\sup_{h\in[-T,T]^d}\left\{\frac{\norminf{\Delta_{(h)}^BX^{\eta}(\cdot,\om)}} {\displaystyle\prod_{l=1}^d{\va{h_l}^{b_l(1-\eta_l)}\va{h_l}^{\min(b_l,a_l)\eta_l}}\left(\log\left(3+\va{h_l}^{-1}\right)\right)^{\eta_l\L_\alpha(a_l,b_l,\delta)}}\right\}<+\infty.
\end{equation}
\end{itemize}
\end{thm}

It easily follows from Remark~\ref{rem:wavrepeta} and Theorem~\ref{thm:etaincr} that:

\begin{coro}
\label{cor:etaincr} 
The positive exponents $a_1,\ldots, a_d$ are the same as in Definition~\ref{def:adm}. 
Moreover we assume that $B=(b_1,\ldots,b_d)\in\Z_+^d$, $T\in (0,+\infty)$ and $\om\in\Om_1 ^*$ are arbitrary and fixed. Then, the following two results hold (with the convention that $0/0=0$).
\begin{itemize}
\item[(i)] When $\alpha=2$, one has 
\begin{equation}
\label{cor:etaincr:eq3bis}
\sup_{h\in[-T,T]^d}\left\{\frac{\norminf{\Delta_{(h)}^BX(\cdot,\om)}} {\displaystyle\prod_{l=1}^d\va{h_l}^{\min(b_l,a_l)}\left(\log\left(3+\va{h_l}^{-1}\right)\right)^{\L_2(a_l,b_l)}}\right\}<+\infty.
\end{equation}
\item[(ii)]  When $\alpha\in (0,2)$, for all arbitrarily small positive real numbers $\delta$, one has  
\begin{equation}
\label{cor:etaincr:eq3}
\sup_{h\in[-T,T]^d}\left\{\frac{\norminf{\Delta_{(h)}^BX(\cdot,\om)}} {\displaystyle\prod_{l=1}^d\va{h_l}^{\min(b_l,a_l)}\left(\log\left(3+\va{h_l}^{-1}\right)\right)^{\L_\alpha(a_l,b_l,\delta)}}\right\}<+\infty.
\end{equation}
\end{itemize}
\end{coro}

The following proposition is the main ingredient of the proof of Theorem~\ref{thm:etaincr}.

\begin{propo} 
\label{propo:etaincr}
The positive exponents $a_1,\ldots, a_d$ are the same as in Definition~\ref{def:adm}. 
Moreover, we assume that $\eta=(\eta_1,\ldots,\eta_d)\in\Upss:=\{0,1\}^d\setminus\{(0,\ldots,0)\}$, $B=(b_1,\ldots,b_d)\in\Z_+^d$, $T\in (0,+\infty)$ and $\om\in\Om_1 ^*$ are arbitrary and fixed. Then, the following two results hold (with the convention that $0/0=0$); notice that the notations used in their statements are the same as in \eqref{def:geta:eq1}, \eqref{def:acc:eq1bis}, \eqref{le:phi:eq1}, and Definition~\ref{def1:ant}. 
\begin{itemize}
\item[(i)] When $\alpha=2$, one has
\begin{equation}
\label{thm:etaincr:eq2bis}
\sup_{h\in[-T,T]^d}\left\{\frac{\displaystyle\sum_{J\in\Geta}\norminf{\Delta_{(h)}^B\big(\PhialphaJ\big(2^J\cdot,\om\big)\big)}} {\displaystyle\prod_{l=1}^d{\va{h_l}^{b_l(1-\eta_l)}\va{h_l}^{\min(b_l,a_l)\eta_l}}\left(\log\left(3+\va{h_l}^{-1}\right)\right)^{\eta_l\L_2(a_l,b_l)}}\right\}<+\infty.
\end{equation}
\item[(ii)]  When $\alpha\in (0,2)$, for all arbitrarily small positive real numbers $\delta$, one has  
\begin{equation}
\label{thm:etaincr:eq2}
\sup_{h\in[-T,T]^d}\left\{\frac{\displaystyle\sum_{J\in\Geta}\norminf{\Delta_{(h)}^B\big(\PhialphaJ\big(2^J\cdot,\om\big)\big)}} {\displaystyle\prod_{l=1}^d{\va{h_l}^{b_l(1-\eta_l)}\va{h_l}^{\min(b_l,a_l)\eta_l}}\left(\log\left(3+\va{h_l}^{-1}\right)\right)^{\eta_l\L_\alpha (a_l,b_l,\delta)}}\right\}<+\infty.
\end{equation}
\end{itemize}
\end{propo}

We now show that Proposition~\ref{propo:etaincr} holds; to this end, we need the three following lemmas.

\begin{lemme}
\label{le:incrgen}
Denote by $B=(b_1,\dots,b_d)\in\Z_+^d$ an arbitrary multi-index, and by $\mathrm l(B):=b_1+\ldots+b_d$ its length.
Then for all functions $g\in\Cn{\ii}{\R^d}$, for any positive real number $T$, and for each $h=(h_1,\ldots, h_d)\in [-T,T]^d$, the following inequality holds:
\begin{equation}
\label{le:incrgen:eq1}
\norminf{\Delta^B_{(h)}g}\leq 2^{\mathrm l(B)}\times\min_{B'\in I(B)}\left\{\norminfT{\partial^{B'}\!g}{T\, 2^{\mathrm l(B)}}\times\prod_{l=1}^d\va{h_l}^{b_l'}\right\},
\end{equation}
with the convention that $0^0=1$, and where the set $I(B)$ is defined as
\begin{equation}
\label{le:incrgen:eq2}
I(B):=\big\{B'=(b'_1,\dots,b'_d)\in\Z_+^d: \text{ for each $l\in\{1,\dots,d\}$, } b_l'\leq b_l\big\}.
\end{equation}
\end{lemme}

\begin{lemme}
\label{le:majincrphi-bis}
Assume that the real numbers $T>0$, $\alpha>0$, $\mu\ge 0$ and $b\ge 0$ are arbitrary and fixed. Then, one has
\begin{equation}
\label{le:majincrphi:eq2}
\sup_{z\in[-T,T]}\left\{\frac{\displaystyle\sum_{j=-\ii}^{0} 2^{j/\alpha}\puissm{\va{j}}\min\left(\va{2^{j}z}^{b},1\right)}{{{\displaystyle{\big| z\big|^{b}}}}}\right\}<+\ii.
\end{equation}
with the conventions that $0/0=0$ and $0^0=1$.
\end{lemme}

\begin{lemme}
\label{le:majincrphi}
Assume that the real numbers $T>0$, $a>0$, $\mu\ge 0$ and $b\ge 0$ are arbitrary and fixed. Then, the following three results hold (with the conventions that $0/0=0$ and $0^0=1$).
\begin{enumerate}
\item When $b<a$, one has
\begin{equation}
\label{le:majincrphi:eq1}
\sup_{z\in[-T,T]}\left\{\frac{\displaystyle\sum_{j=1}^{+\ii}2^{-ja}\puissm{j}\min\left(\va{2^{j}z}^{b},1\right)}{{{\displaystyle{\big| z\big|^{b}}}}}\right\}<+\ii.
\end{equation}
\item When $b=a$, one has
\begin{equation}
\label{le:majincrphi:eq1ter}
\sup_{z\in[-T,T]}\left\{\frac{\displaystyle\sum_{j=1}^{+\ii}2^{-ja}\puissm{j}\min\left(\va{2^{j}z}^{b},1\right)}{{{\displaystyle{\va{z}^{a}\left(\log\left(3+\va{z}^{-1}\right)\right)^{\mu+1}}}}}\right\}<+\ii.
\end{equation}
\item When $b>a$, one has
\begin{equation}
\label{le:majincrphi:eq1bis}
\sup_{z\in[-T,T]}\left\{\frac{\displaystyle\sum_{j=1}^{+\ii}2^{-ja}\puissm{j}\min\left(\va{2^{j}z}^{b},1\right)}{{{\displaystyle{\va{z}^{a}\left(\log\left(3+\va{z}^{-1}\right)\right)^{\mu}}}}}\right\}<+\ii.
\end{equation}
\end{enumerate}
\end{lemme}

\begin{proof}[Proof of Lemma~\ref{le:incrgen}]
We intend to proceed by induction on $\mathrm l(B)$. More precisely, the proof is structured as follows. {\em In the Part~1}, we establish the lemma in the particular case where $\mathrm l(B)=0$. {\em In the Part~2}, we denote by $n$ an arbitrary fixed non-negative integer, and we assume that the lemma holds when $\mathrm l(B)=n$ (such a $B$ is denoted by~$\wt{B}$), then the goal is to derive it in the case where $\mathrm l(B)=n+1$. 

{\em\underline{Part~1:}} In view of \eqref{le:incrgen:eq2} and of the assumption $\mathrm l(B)=0$, the set $I(B)$ reduces to $\left\{0\right\}$. Then, in view of the equalities $\Delta_{(h)}^{0}g=g$, for all $h\in\R^d$, and $\partial^{0}\!g=g$, it is clear that the lemma is true.

{\em\underline{Part~2:}} Let $B\in\Z_+^d$ be arbitrary and satisfying $\mathrm l(B)=n+1$. One has to show that, for all $g\in\Cn{\ii}{\R^d}$, for any positive real number $T$, and for each $h=(h_1,\ldots, h_d)\in [-T,T]^d$, the following inequality holds:
\begin{equation}
\label{incr:eq3bis}
\norminf{\Delta^B_{(h)}g}\le 2^{\mathrm l(B)}\times\min_{B'\in I(B)}\left\{\norminfT{\partial^{B'}\!g}{T\, 2^{\mathrm l(B)}}\times\prod_{l=1}^d\va{h_l}^{b_l'}\right\}.
\end{equation}
Observe that there exists $\wt{B}\in\Z_+^d$ satisfying $\mathrm l(\wt{B})=n$, and there exists $k\in\{1,\dots,d\}$, such that $B$ can be expressed as
\begin{equation}
\label{btilde}
B=\wt{B}+e_k,
\end{equation}
where $e_k\in\Z^d$ is the multi-index whose $k$-th coordinate equals $1$ and the others vanish. Next, it follows from~\eqref{btilde}, \eqref{def:acc:eq1bis} and \eqref{def:acc:eq1} that 
\begin{equation}
\label{le:incrgen:proof:eq1}
\norminf{\Delta^B_{(h)}g}=\sup_{x\in[-T,T]^d}\Big|\big (\Delta^{\wt{B}}_{(h)}g\big) (x+h_ke_k)-\big(\Delta^{\wt{B}}_{(h)}g\big)(x)\Big|.
\end{equation}
Therefore, using the triangle inequality one has that
\begin{equation}
\label{le:incrgen:proof:eq1bis}
\norminf{\Delta^B_{(h)}g}\le 2\norminfT{\Delta^{\wt{B}}_{(h)}g}{2T}\le 2^{\mathrm l(B)}\times\min_{B'\in I(\wt{B})}\left\{\norminfT{\partial^{B'}\!g}{T\, 2^{\mathrm l(B)}}\times|h|_\pi ^{B'}\right\},
\end{equation}
where the convenient notation $|h|_\pi ^{B'}$ is defined by
\begin{equation}
\label{le:incrgen:proof:eq1ter}
|h|_\pi ^{B'}:=\prod_{l=1}^d\va{h_l}^{b_l'};
\end{equation}
notice that the last inequality in \eqref{le:incrgen:proof:eq1bis} results from the induction hypothesis~\footnote{In which $B$ is replaced by $\wt{B}$ and $T$ by $2T$.} and the equality $\mathrm l(B)=\mathrm l(\wt{B})+1$. On the other hand, one can derive from \eqref{le:incrgen:proof:eq1}, the Mean Value Theorem, and the equality $\partial^{e_k}\big(\Delta^{\wt{B}}_{(h)}g\big)=\Delta^{\wt{B}}_{(h)}\big(\partial^{e_k}  g\big)$ that
\begin{equation}
\label{incr:eq1}
\norminf{\Delta^B_{(h)} g}\leq \va{h_k}\norminfT{\Delta^{\wt{B}}_{(h)}\big(\partial^{e_k} g\big)}{2T}.
\end{equation}
Moreover, applying the induction hypothesis~\footnote{In which $B$ is replaced by $\wt{B}$, $g$ by $\partial^{e_k} g$, and $T$ by $2T$.} and using \eqref{le:incrgen:proof:eq1ter}, one gets that
\begin{equation}
\label{incr:eq2}
\norminfT{\Delta^{\wt{B}}_{(h)}\big(\partial^{e_k} g\big)}{2T}\leq  2^{\mathrm l(\wt{B})}\times\min_{B'\in I(\wt{B})}\left\{\norminfT{\partial^{B'+e_k} g}{T\, 2^{\mathrm l(B)}}\times |h|_\pi ^{B'}\right\}.
\end{equation}
Next, putting together~\eqref{incr:eq1}, \eqref{incr:eq2}, \eqref{le:incrgen:proof:eq1ter} and the inequality $\mathrm l(\wt{B})<\mathrm l(B)$, we obtain that
\begin{equation}
\label{incr:eq3}
\norminf{\Delta^B_{(h)}g}\leq  2^{\mathrm l(B)}\times\min_{B'\in I(\wt{B})}\left\{\norminfT{\partial^{B'+e_k}g}{T\,2^{\mathrm l(B)}}\times |h|_\pi ^{B'+e_k}\right\}.
\end{equation}
Finally, in view of the fact 
\[
I(B)=I(\wt{B})\cup\big\{B'+e_k: B'\in I(\wt{B})\big\},
\]
one can derive from \eqref{le:incrgen:proof:eq1ter}, \eqref{le:incrgen:proof:eq1bis} and \eqref{incr:eq3} that \eqref{incr:eq3bis} holds.
\end{proof}

\begin{proof}[Proof of Lemma~\ref{le:majincrphi-bis}] Observe that for all $z\in [-T,T]$ and $j\in\Z_{-}$, one has $\va{2^{j}T^{-1}z}^{b}\le 1$. Therefore, one obtains that 
\begin{eqnarray*}
&& \sum_{j=-\ii}^{0} 2^{j/\alpha}\puissm{\va{j}}\min\left(\va{2^{j}z}^{b},1\right)=\sum_{j=-\ii}^{0} 2^{j/\alpha}\puissm{\va{j}}\min\left(T^b\va{2^{j}T^{-1}z}^{b},1\right)\\
&& \le (1+T)^b \sum_{j=-\ii}^{0} 2^{j/\alpha}\puissm{\va{j}}\min\left(\va{2^{j}T^{-1}z}^{b},1\right) =c\va{z}^b,
\end{eqnarray*}
where the finite constant $c$ is equal to
\[
c:=(1+T)^b\, T^{-b}\sum_{j=-\ii}^{0} 2^{j(1/\alpha+b)}\puissm{\va{j}}.
\]
\end{proof}

\begin{proof}[Proof of Lemma~\ref{le:majincrphi}] Let $z\in [-T,T]$ be arbitrary and fixed; there is no restriction to assume that $z\ne 0$. One sets
\begin{equation}
\label{le:majincrphi:proof:eqan1}
\mathrm{j}_0(z):=\min\left\{j\in\N : \va{2^{j}z} > 1\right\}.
\end{equation}
It can easily be shown that there are two constants $0<c_1<c_2<+\infty$, not depending on $z$, such that
\begin{equation}
\label{le:majincrphi:proof:eqan0}
c_1\log\left(3+\va{z}^{-1}\right)\le \mathrm{j}_0(z)\le c_2 \log\left(3+\va{z}^{-1}\right).
\end{equation}
Observe that, for any arbitrary fixed real numbers $a>0$, $\mu\ge 0$ and $b\ge 0$, one has that 
\begin{equation}
\label{le:majincrphi:proof:eqan2}
\sum_{j=\mathrm{j}_0(z)}^{+\ii}2^{-ja}\puissm{j}\min\left(\va{2^{j}z}^{b},1\right)=\sum_{j=\mathrm{j}_0(z)}^{+\ii}2^{-ja}\puissm{j}
\end{equation}
and 
\begin{equation}
\label{le:majincrphi:proof:eqan2bis}
\sum_{j=1}^{\mathrm{j}_0(z)-1}2^{-ja}\puissm{j}\min\left(\va{2^{j}z}^{b},1\right)=|z|^b \sum_{j=1}^{\mathrm{j}_0(z)-1}2^{-j(a-b)}\puissm{j},
\end{equation}
with the convention that $\sum_{j=1}^{0}\ldots=0$. We are going to conveniently bound from above the right-hand side in \eqref{le:majincrphi:proof:eqan2} and the right-hand side in \eqref{le:majincrphi:proof:eqan2bis}. First, we show that there exists a finite constant $c_3$, not depending on $z$,  such that 
\begin{equation}
\label{le:majincrphi:proof:eqan3}
\sum_{j=\mathrm{j}_0(z)}^{+\ii}2^{-ja}\puissm{j}\le c_3 \va{z}^{a}\left(\log\left(3+\va{z}^{-1}\right)\right)^{\mu}.
\end{equation}
This is indeed the case since one has that 
\begin{eqnarray*}
 \sum_{j=\mathrm{j}_0(z)}^{+\ii}2^{-ja}\puissm{j}&=&\sum_{j=0}^{+\ii}2^{-\mathrm{j}_0(z)a-ja}\puissm{\mathrm{j}_0(z)+j}\\ 
&=& 2^{-\mathrm{j}_0(z)a}\,\mathrm{j}_0(z)^\mu\sum_{j=0}^{+\ii}2^{-ja}\puissm{\frac{1+j}{\mathrm{j}_0(z)}} \le  c_3 \va{z}^{a}\,\left(\log\left(3+\va{z}^{-1}\right)\right)^{\mu},
\end{eqnarray*}
where the last inequality results from \eqref{le:majincrphi:proof:eqan1} and \eqref{le:majincrphi:proof:eqan0}; notice that the finite 
constant $c_3$ is defined as
\[
c_3:=c_2 ^\mu\,\sum_{j=0}^{+\ii}2^{-ja}(2+j)^\mu.
\]
Let us now study the right-hand side in \eqref{le:majincrphi:proof:eqan2bis}. In the case where $b<a$, the constant
\[
c_4:=\sum_{j=1}^{+\infty}2^{-j(a-b)}\puissm{j}
\]
is finite, and we have that 
\begin{equation}
\label{le:majincrphi:proof:eqan3bis}
|z|^b \sum_{j=1}^{\mathrm{j}_0(z)-1}2^{-j(a-b)}\puissm{j}\le c_4 |z|^b.
\end{equation}
 In the second case where $b=a$, one has
\begin{eqnarray}
\label{le:majincrphi:proof:eqan4}
|z|^b \sum_{j=1}^{\mathrm{j}_0(z)-1}2^{-j(a-b)}\puissm{j}=|z|^a \sum_{j=1}^{\mathrm{j}_0(z)-1}\puissm{j}\le |z|^a\,\mathrm{j}_0(z)^{\mu+1}\le c_2 ^{\mu+1} \va{z}^{a}\left(\log\left(3+\va{z}^{-1}\right)\right)^{\mu+1},
\end{eqnarray}
where the last inequality results from \eqref{le:majincrphi:proof:eqan0}. In the third and last case  where $b>a$, letting $c_5$ and $c_6$ be the finite constants defined as $c_5:=2^{b-a}\big/\big(2^{b-a}-1\big)$ and $c_6:=c_5c_2 ^{\mu}$, one has
\begin{eqnarray}
\label{le:majincrphi:proof:eqan5}
|z|^b \sum_{j=1}^{\mathrm{j}_0(z)-1}2^{-j(a-b)}\puissm{j}\le c_5 |z|^b \, 2^{(\mathrm{j}_0(z)-1)(b-a)}\,\mathrm{j}_0(z)^{\mu}\le c_6 \va{z}^{a}\left(\log\left(3+\va{z}^{-1}\right)\right)^{\mu},
\end{eqnarray}
where the last inequality follows from \eqref{le:majincrphi:proof:eqan1} and \eqref{le:majincrphi:proof:eqan0}.

Finally, putting together \eqref{le:majincrphi:proof:eqan2} to \eqref{le:majincrphi:proof:eqan5} one gets the lemma.

\end{proof}

We are now in the position to prove Proposition~\ref{propo:etaincr}.

\begin{proof}[Proof of Proposition~\ref{propo:etaincr}] We only give the proof of \eqref{thm:etaincr:eq2}, since that of \eqref{thm:etaincr:eq2bis} can be done 
in the same way, except that one has to make use of \eqref{le:majphi:eq3ter}, instead of \eqref{le:majphi:eq3} and \eqref{le:majphi:eq3bis}. So, in the rest of the proof we assume that $\alpha\in (0,2)$.

We know from Proposition~\ref{pro:phi} that, for all fixed $J\in\Geta$, the function $\PhialphaJ\big (2^J\cdot,\om\big )$ belongs to the space $\Cn{\ii}{\R^d}$. Thus, it follows from Lemma~\ref{le:incrgen} that 
\begin{equation}
\label{incrphi:eq1}
\norminf{\Delta^B_{(h)}\big(\PhialphaJ\big(2^J\cdot,\om\big)\big)}\le  c_1\times\min_{B'\in I(B)}\left\{\norminfT{\partial^{B'}\big(\PhialphaJ\big(2^J\cdot,\om\big)\big)}{T_1}\times\prod_{l=1}^d\va{h_l}^{b_l'}\right\}, 
\end{equation}
where $I(B)$ is the same finite set as in \eqref{le:incrgen:eq2}, and the finite constants $c_1$ and $T_1$ are defined as $c_1:=2^{\mathrm l(B)}$ and $T_1:=T\,2^{\mathrm l(B)}$. Moreover, we know from \eqref{le:majphi:eq3} and \eqref{le:majphi:eq3bis} that, for all fixed positive real numbers $\delta$, and for any $B'\in I(B)$, one has
\begin{equation}
\label{incrphi:eq2}
\norminfT{\partial^{B'}\left(\PhialphaJ(2^J\cdot,\om)\right)}{T_1}\leq C_2(\om)\prod_{l=1}^{d}2^{(1-\eta_l)j_l(1/\alpha +b'_l)}2^{-\eta_lj_l(a_l-b'_l)}\puissa{\va{j_l}},
\end{equation}
where $\lfloor \alpha \rfloor$ is the integer part of $\alpha$. Notice that the finite constant $C_2(\om)$ does not depend on $J$ and $h$; also, it can be chosen in such a way that it does not depend on $B'$, since $I(B)$ is a finite set. Next setting $C_3(\om):=c_1C_2(\om)$ and using the fact that $\eta_l\in\{0,1\}$, for all $l\in\{1,\ldots,d\}$, one can derive from \eqref{incrphi:eq1}, \eqref{incrphi:eq2}, and \eqref{le:incrgen:eq2}, that 
\begin{eqnarray*}
\norminf{\Delta^B_{(h)}\big(\PhialphaJ\big(2^J\cdot,\om\big)\big)} &\le & C_3 (\om)\prod_{l=1}^d2^{(1-\eta_l)j_l/\alpha}\,2^{-\eta_lj_la_l}\puissa{\va{j_l}} \min_{B'\in I(B)}\left\{\va{2^{j_l}h_l}^{b'_l}\right\}\nonumber\\
&\leq &\ C_3 (\om)\prod_{l=1}^d2^{(1-\eta_l)j_l/\alpha}\,2^{-\eta_lj_la_l}\puissa{\va{j_l}}\min\left\{\va{2^{j_l}h_l}^{b_l},1\right\}.
\end{eqnarray*}
Then, \eqref{thm:etaincr:eq2} can be obtained by using \eqref{def:geta:eq1}, \eqref{def:geta:eq2}, Lemmas~\ref{le:majincrphi-bis} and \ref{le:majincrphi}, as well as Definition~\ref{def1:ant}. 
\end{proof}

We are now in the position to prove Theorem~\ref{thm:etaincr}.

\begin{proof}[Proof of Theorem~\ref{thm:etaincr}] When $\eta=0=(0,\ldots,0)$ the theorem easily results from Proposition~\ref{thm:lfcvu}
and Lemma~\ref{le:incrgen}. When $\eta\ne 0$ the theorem can easily be derived from~\eqref{wavrepeta}, \eqref{le:phi:eq1}, the triangle inequality and Proposition~\ref{propo:etaincr}.
\end{proof}

%%%%%%%%%%%%%%%%%%%%%%%%%%%%%%%%%%%%%%%%%%%%%%%%
%%%%%%%%%%%%%%%%%%%%%%%%%%%%%%%%%%%%%%%%%%%%%%%%
%%%%%%%%%%%%%%%%%%%%%%%%%%%%%%%%%%%%%%%%%%%%%%%%

In order to state the second main result of this section, we need to introduce some additional notations.
\begin{defi} 
\label{def1:Ltilde}
$ $
\begin{itemize}
\item[(i)] We denote by $\wt{\L_2}$ the function defined, for each $a\in\R_+$, as
\begin{equation}
\label{eq:Ltilde2}
\wt{\L_2}(a):=1/2 +\ind{\{a\in\N\}}.
\end{equation}
More precisely, one has: 
\[
\text{ $\wt{\L_2}(a)=3/2$ if $a\in\N$, and $\wt{\L_2}(a)=1/2$ if $a\notin\N$.}
\]
\item[(ii)] For any fixed $\alpha\in (0,2)$, we denote by $\wt{\L_\alpha} $ the function defined, for each $(a,\delta)\in\R_+^2$, as
\begin{equation}
\label{eq:Ltilde3}
\wt{\L_\alpha} (a,\delta):=1/\alpha+\lfloor\alpha\rfloor/2+\delta+\ind{\{a\in\N\}},
\end{equation}
where $\lfloor\alpha\rfloor$ is the integer part of $\alpha$. More precisely, 
\begin{itemize}
\item when $\alpha\in (0,1)$, one has: 
\[
\text{ $\wt{\L_\alpha}(a,\delta)=1/\alpha+1+\delta$ if $a\in\N$, and $\wt{\L_\alpha}(a,\delta)=1/\alpha+\delta$ if $a\notin\N$;}
\]
\item when $\alpha\in [1,2)$, one has: 
\[
\text{ $\wt{\L_\alpha}(a,\delta)=1/\alpha+3/2+\delta$ if $a\in\N$, and $\wt{\L_\alpha}(a,\delta)=1/\alpha+1/2+\delta$ if $a\notin\N$.}
\]
\end{itemize}
\end{itemize}
\end{defi}

For any fixed $h\in\R^d$, we denote by $\mathbf{\Delta}_h$, the operator from the space of the real-valued functions on $\R^d$, into itself; so that, when $g$ is such a function, $\mathbf{\Delta}_h g$ is then the function defined, for all $x\in\R^d$, as
\begin{equation}
\label{def:acc:eq1ter}
\big(\mathbf{\Delta}_hg\big)(x):=g(x+h)-g(x).
\end{equation}
Moreover, for each positive integer $n$, we denote by $\mathbf{\Delta}^n_h$ the operator $\mathbf{\Delta}_h$ composed $n$ times with itself. 

We are now ready to state the second main result of this section.

\begin{thm}
\label{thm:rect}
The positive exponents $a_1,\ldots, a_d$ are the same as in Definition~\ref{def:adm}, and we set 
\[
n_0:=1-d+\sum_{l=1}^d\lceil a_l\rceil, 
\]
where $\lceil a_l\rceil:=\min\{m\in\N: m\geq a_l\}$, for any $l\in\{1,\dots,d\}$.
Moreover, we assume that $\eta=(\eta_1,\ldots,\eta_d)\in\Upsilon:=\{0,1\}^d$, $T\in (0,+\infty)$ and $\om\in\Om_1 ^*$ are arbitrary and fixed. Let $n$ be an arbitrary integer such that $n\geq n_0$. Then, the following two results hold (with the convention that $0/0=0$). %; notice that the notations used in their statements are the same as in \eqref{def:acc:eq1ter}, and Definition~\ref{def1:Ltilde}. 
\begin{itemize}
\item[(i)]  When $\alpha=2$, one has  
\begin{equation}
\label{thm:rect:eq1bis}
\sup_{h\in[-T,T]^d}\left\{\frac{\displaystyle\norminf{\mathbf{\Delta}_{h}^n{X}^{\eta}(\cdot,\om)}} {\displaystyle\sum_{l=1}^d\va{h_{l}}^{  \eta_{l}a_{l}+(1-\eta_{l})\lceil a_{l}\rceil}\left(\log\left(3+\va{h_l}^{-1}\right)\right)^{\eta_l \wt{\L_2}(a_l)}}\right\}<+\infty.
\end{equation}
\item[(ii)] When $\alpha\in(0,2)$, for all arbitrarily small positive real numbers $\delta$, one has 
\begin{equation}
\label{thm:rect:eq1}
\sup_{h\in[-T,T]^d}\left\{\frac{\displaystyle\norminf{\mathbf{\Delta}_{h}^n{X}^{\eta}(\cdot,\om)}} {\displaystyle\sum_{l=1}^d\va{h_{l}}^{  \eta_{l}a_{l}+(1-\eta_{l})\lceil a_{l}\rceil}\left(\log\left(3+\va{h_l}^{-1}\right)\right)^{\eta_l\wt{\L_\alpha}(a_l,\delta)}}\right\}<+\infty.
\end{equation}
\end{itemize}
\end{thm}

It easily follows from Remark~\ref{rem:wavrepeta} and Theorem~\ref{thm:rect} that:

\begin{coro}
\label{cor:incr} 
The positive exponents $a_1,\ldots, a_d$ are the same as in Definition~\ref{def:adm}, and the positive integer $n_0=n_0(a_1,\dots,a_d,d)$ is the same as in Theorem \ref{thm:rect}. 
Moreover, we assume that $T\in (0,+\infty)$ and $\om\in\Om_1 ^*$ are arbitrary and fixed.  Let $n$ be an arbitrary integer such that $n\geq n_0$.  Then, the following two results hold (with the convention that $0/0=0$).%; notice that the notations used in their statements are the same as in \eqref{def:geta:eq1}, \eqref{def:acc:eq1ter}, \eqref{le:phi:eq1}, and Definition~\ref{def1:ant}. 
\begin{itemize}
\item[(i)]  When $\alpha=2$, one has  
\begin{equation}
\label{thm:incr:eq1bis}
\sup_{h\in[-T,T]^d}\left\{\frac{\displaystyle\norminf{\mathbf{\Delta}_{h}^n{X(\cdot,\om)}}} {\displaystyle\sum_{l=1}^d\va{h_{l}}^{ a_{l}}\left(\log\left(3+\va{h_l}^{-1}\right)\right)^{\wt{\L_2}(a_l)}}\right\}<+\infty.
\end{equation}
\item[(ii)] When $\alpha\in(0,2)$, for all arbitrarily small positive real numbers $\delta$, one has 
\begin{equation}
\label{thm:incr:eq1}
\sup_{h\in[-T,T]^d}\left\{\frac{\displaystyle\norminf{\mathbf{\Delta}_{h}^n{X(\cdot,\om)}}} {\displaystyle\sum_{l=1}^d\va{h_{l}}^{  a_{l}}\left(\log\left(3+\va{h_l}^{-1}\right)\right)^{\wt{\L_\alpha}(a_l,\delta)}}\right\}<+\infty.
\end{equation}
\end{itemize}
\end{coro}

\begin{proof}[Proof of Theorem~\ref{thm:rect}]
We only give the proof of~\eqref{thm:rect:eq1}; the strategy of the proof remains the same in the case of~\eqref{thm:rect:eq1bis}, except that \eqref{thm:etaincr:eq3bis} has to be used instead of \eqref{thm:etaincr:eq3}. 

Let $T\in(0,+\ii)$ and $h=(h_1,\ldots , h_{k-1},h_k,\ldots, h_d)\in[-T,T]^d$ be arbitrary and fixed. First, we are going to express the operator $\mathbf{\Delta}_{h}$ (see \eqref{def:acc:eq1ter}) in terms of the operators $\Delta_{h_k}^k$, $k\in\{1,\dots,d\}$ (see \eqref{def:acc:eq1}), and of some translation operators. To this end, for any fixed $k\in\{1,\dots,d+1\}$, we denote by $(h)_{k,0}$ the vector of $\R^d$ such that $(h)_{k,0}:=(h_1,\ldots, h_{k-1},0,\ldots, 0)$, with the convention that $(h)_{1,0}$ is the zero vector and that $(h)_{d+1,0}$ is the vector $h$ itself. Also, for any fixed vector $r\in\R^d$, we denote by $\Theta_r$, the translation operator from the space of the real-valued functions on $\R^d$, into itself; so that, when $\tilde{g}$ is such a function, $\Theta_r \tilde{g}$ is then the function defined, for all $x\in\R^d$, as $\big(\Theta_r \tilde {g}\big)(x):=\tilde{g}(x+r)$. One can easily check that $\Theta_r \circ\Delta_{h_{k}}^{k}=\Delta_{h_{k}}^{k}\circ \Theta_r$, for every $k\in\{1,\dots,d\}$, and that
\begin{equation}
\label{addeq:expbigdelta}
\mathbf{\Delta}_{h}=\sum_{k=1}^d \Theta_{(h)_{k,0}}\circ\Delta_{h_{k}}^{k}.
\end{equation}
Now, let $n$ be the same integer as in the statement of Theorem~\ref{thm:rect}, and let $g$ be an arbitrary real-valued continuous function on $\R^d$. Using \eqref{addeq:expbigdelta}, the Multinomial Theorem, the triangle inequality and the inequality $2^n\ge n+1$, we get that
\begin{equation}
\label{incrusual}
\norminf{\mathbf{\Delta}_h^ng}\leq n! \sum_{B\in E_n}\norminfT{{\Delta_{(h)}^B}g}{(n+1)T}\leq n! \sum_{B\in E_n}\norminfT{{\Delta_{(h)}^B}g}{2^nT},
\end{equation}
where the finite set $E_n:=\big\{B=(b_1,\dots,b_d)\in\Z_+^d:\mathrm{l}(B):=b_1+\dots+b_d=n\big\}$, and the operators $\Delta_{(h)}^B$ are defined through \eqref{def:acc:eq1bis}.
%, and $c_1$ is a positive and finite contant which depends only on $n$
Moreover, similarly to \eqref{le:incrgen:eq1}, it can be shown, for each $B\in\Z_+^d$, that
\begin{equation}
\label{incrusual2}
\norminfT{{\Delta_{(h)}^B}g}{2^nT}\leq 2^{\mathrm{l}(B)}\min_{B'\in I(B)}\norminfT{{\Delta_{(h)}^{B'}}g}{2^{\mathrm{l}(B)+n}T}=2^{n}\min_{B'\in I(B)}\norminfT{{\Delta_{(h)}^{B'}}g}{2^{2n}T},
\end{equation}
where the finite set $I(B):=\big\{B'=(b'_1,\dots,b'_d)\in\Z_+^d: \text{ for each $l\in\{1,\dots,d\}$, } b_l'\leq b_l\big\}$.
Next, applying~\eqref{incrusual} and~\eqref{incrusual2} to $g=X^{\eta}(\cdot,\om)$, where $\om\in\Om_1 ^*$ is arbitrary and fixed, we obtain that
\begin{equation}
\label{thm:rect:proof:eq0}
\norminf{\mathbf{\Delta}_h^nX^{\eta}(\cdot,\om)}\leq 2^{n}\,n!\sum_{B\in E_n}\min_{B'\in I(B)}\norminfT{{\Delta_{(h)}^{B'}}X^{\eta}(\cdot,\om)}{2^{2n}T}.
\end{equation}
%On the other hand, it follows from the inequality $n\geq n_0$ that each $B$ belonging to $E_n$ satisfies the following property: the coordinate $b_1,\dots,b_d$ of an arbitrary $B\in E_n$ there is $l\in\{1,\dots,d\}$, such that $b_{l}\geq a_{l}$.
Let us now provide, for any fixed $B\in E_n$, a suitable upper bound for the quantity $\displaystyle\min_{B'\in I(B)}\norminfT{{\Delta_{(h)}^{B'}}X^{\eta}(\cdot,\om)}{2^{2n}T}$. To this end, we set
\begin{equation}
\label{thm:rect:proof:eq0bis}
l_0(B):=\min\big\{l\in\{1,\dots,d\}:b_l\geq a_l\big\}.
\end{equation}
Observe that $l_0(B)$ is well-defined since the inequality $n\geq n_0:=1-d+\sum_{l=1}^d\lceil a_l\rceil$ implies that there exists at least one $l\in\{1,\dots,d\}$ satisfying $b_l\geq a_l$.
%$\big\{l\in\{1,\dots,d\}:b_l\geq a_l\big\}$ is a not empty finite set.
Next, let $B^0:=(b_1^0,\dots,b_d^0)\in \Z_+^d$ be such that 
%\begin{equation}
%\nonumber
%\left\{
%    \begin{array}{l @{\,\,\,\text{ if }\,\,\,} l}
%    b_l'=b_l & \eta_l=0  \\ 
%    b_l'= \lceil  a_l\rceil\ind{l=l_0} & \eta_l=1.
%    \end{array}
%    \right.
%\end{equation}
\begin{equation}
\label{thm:rect:proof:eq0ter}
b_l^0=\lceil  a_l\rceil\ind{\{l=l_0(B)\}}, \text{ \,\,\, for all\,\,\,} l\in\{1,\dots,d\};
\end{equation}
that is 
$b^0_{l_0(B)}=\lceil a_{l_0(B)}\rceil$, and $b_l^0=0$  for all $l\neq l_0(B)$. Notice that $B^0$ belongs to $I(B)$, since \eqref{thm:rect:proof:eq0bis} entails that $b_{l_0(B)}\ge\lceil a_{l_0(B)}\rceil=b^0_{l_0(B)}$.
%$n\ge n_0\ge \max_{1\le l \le d} \,\lceil  a_l\rceil$. 
As a consequence, we have that
\[
%\label{thm:rect:proof:eq0qua}
\min_{B'\in I(B)}\norminfT{{\Delta_{(h)}^{B'}}X^{\eta}(\cdot,\om)}{2^{2n}T}\leq\norminfT{{\Delta_{(h)}^{B^0}}X^{\eta}(\cdot,\om)}{2^{2n}T}.
\]
Thus, it follows from~\eqref{thm:etaincr:eq3}, \eqref{thm:rect:proof:eq0ter}, \eqref{eq:ant-defL3} and \eqref{eq:Ltilde3} that, for any fixed $\delta\in (0,+\infty)$, we have
\begin{eqnarray}
%\nonumber\norminfT{{\Delta_{(h)}^{B^0}}X^{\eta}(\cdot,\om)}{2^{2n}T}
 && \nonumber\min_{B'\in I(B)}\norminfT{{\Delta_{(h)}^{B'}}X^{\eta}(\cdot,\om)}{2^{2n}T} \\
 && \nonumber\le  C_2(\om,B)\prod_{l=1}^d\va{h_l}^{b_l^0(1-\eta_l)}\va{h_l}^{\min(b_l^0,a_l)\eta_l}\left(\log\left(3+\va{h_l}^{-1}\right)\right)^{\eta_l\L_\alpha(a_l,b_l^0,\delta)}\\
&&=C_2(\om,B)\va{h_{l_0(B)}}^{  (1-\eta_{l_0(B)})\lceil a_{l_0(B)}\rceil+\eta_{l_0(B)}a_{l_0(B)}}\left(\log\left(3+\va{h_{l_0(B)}}^{-1}\right)\right)^{\eta_{l_0(B)}\wt{\L_{\alpha}}(a_{l_0(B)},\delta)},
\label{thm:rect:proof:eq1}
\end{eqnarray}
where $C_2(\om,B)$ is a finite constant not depending on $h$. Finally, let $C_3 (\om)$ and $C_4 (\om)$ be the two finite constants defined as
$
C_3(\om):=(2^n\, n!)\times \max\big\{C_2 (\om,B) : B\in E_n\big\}
$
and $C_4 (\om):=\mbox{card} (E_n)\times C_3 (\om)$, where $\mbox{card} (E_n)$ denotes the cardinality of $E_n$. The inequalities \eqref{thm:rect:proof:eq0} and \eqref{thm:rect:proof:eq1}, and the fact that, for all $B\in E_n$, the index $l_0 (B)$ belongs to $\{1,\dots,d\big\}$ imply that
\begin{eqnarray*}
\norminf{\mathbf{\Delta}_h^nX^{\eta}(\cdot,\om)} &\leq & C_3 (\om)\sum_{B\in E_n}\va{h_{l_0(B)}}^{  (1-\eta_{l_0(B)})\lceil a_{l_0(B)}\rceil+\eta_{l_0(B)}a_{l_0(B)}}\left(\log\left(3+\va{h_{l_0(B)}}^{-1}\right)\right)^{\eta_{l_0(B)}\wt{\L_{\alpha}}(a_{l_0(B)},\delta)}\\
&\le & C_4 (\om) \sum_{l=1}^d\va{h_{l}}^{  (1-\eta_{l})\lceil a_{l}\rceil+\eta_{l}a_{l}}\left(\log\left(3+\va{h_l}^{-1}\right)\right)^{\eta_l\wt{\L_{\alpha}}(a_{l},\delta)},
\end{eqnarray*}
which shows that~\eqref{thm:rect:eq1} holds.
\end{proof}

%%%%%%%%%%%%%%%%%%%%%%%%%%%%%%%%%%%%%%%%%%%%%%%%%%
%%%%%%%%%%%%%%%%%%%%%%%%%%%%%%%%%%%%%%%%%%%%%%%%%%
%%%%%%%%%%%%%%%%%%%%%%%%%%%%%%%%%%%%%%%%%%%%%%%%%%

\section{Behaviour at infinity}
\label{secbehav}
Let $f$ be an admissible function, $X$ the field associated with $f$, and $X^{\eta}$ an arbitrary $\eta$-frequency part of $X$, where $\eta=(\eta_1,\ldots, \eta_d)\in\Upsilon:=\{0,1\}^d$ (see Definition~\ref{def:adm}, \eqref{deffield}, Definition~\ref{def:wavrepeta} and Remark~\ref{rem:wavrepeta}). 
The function $f$ may have a singularity at $0$; yet, in the neighbourhood of this point, $f$ is governed by the exponent $a'\in (0,1)$ through the inequality~\eqref{A3}. The main goal of the present section is to draw connections between the exponent $a'$ and the behaviour at  infinity of $X^{\eta}$, that of $X$, and that of their partial derivatives when they exist. The methodology we use is based on the wavelet type random series representations~\eqref{wavrepeta} and \eqref{deffield2} of $X^{\eta}$ and $X$. It is worth mentioning that all the results we obtain are valid on $\Om_1^*$, the "universal" event of probability~$1$  which was introduced in Lemma~\ref{le:eps-JK}. Let us first state them.

\begin{thm}
\label{propo:behav}
The exponents $a'\in (0,1)$ and $a_1,\ldots, a_d\in (0,+\infty)$  are the same as in Definition~\ref{def:adm}.  Let $\eta=(\eta_1,\ldots, \eta_d)\in\Upsilon:=\{0,1\}^d$ and $ b=(b_1,\dots,b_d)\in\Z_+^d$ be arbitrary and such that \eqref{thm:etacvu:eq1} holds~\footnote{Notice that when $\eta=0=(0,\ldots,0)$, then \eqref{thm:etacvu:eq1} holds for any $ b=(b_1,\dots,b_d)\in\Z_+^d$.}. Then, for each fixed  $\delta\in (0,+\infty)$ and $\om\in\Om_1^*$, the following three results are satisfied (with the convention that $0/0=0$).
%\begin{equation}
%\label{propo:behav:eq0}
  %\left\{
%    \begin{array}{l @{\,\,\,\text{ if }\,\,\,} l}
 %   b_l\in\big\{0,1,\dots,[a_l]\big\} &  \eta_l=1 \\ 
 %   b_l\in\Z_+&  \eta_l=0,
 %   \end{array}
%  \right.
%\end{equation}
\begin{enumerate}
\item When $\alpha\in (0,1)$ one has 
\begin{equation}
\sup_{t \in \R^d }\Big\{\va{\partial^b\!X^{\eta}(t,\om)}\Big\}<+\infty \,\,\,\text{ if }\,\,\, \eta\ne 0 \text{ or }b\neq 0, \label{propo:behav:eq1bis}
\end{equation}
and
\begin{equation}
\label{propo:behav:eq2bis}
\sup_{t\in\R^d}\Big\{\norm{t}^{-a'}\left(\log\big(3+\norm{t}\big)\right)^{-1/\alpha-\delta} \va{X^{0}(t,\om)}\Big\}<+\infty.
\end{equation}

\item When $\alpha\in [1,2)$ one has
\begin{equation}
\sup_{t\in\R^d}\Big\{\left(\log\big(3+\norm{t}\big)\right)^{-1/2} \va{\partial^b\!X^{\eta}(t,\om)}\Big\}<+\infty \,\,\,\text{ if }\,\,\, \eta\ne 0 \text{ or }b\neq 0, \label{propo:behav:eq1ter}
\end{equation}
and
\begin{equation}
\label{propo:behav:eq2ter}
\sup_{t\in\R^d}\Big\{ \norm{t}^{-a'}\left(\log\big(3+\norm{t}\big)\right)^{-1/\alpha-\delta} \va{X^{0}(t,\om)}\Big\}<+\infty.
\end{equation}
\item When $\alpha=2$ one has
\begin{equation}
\sup_{t\in\R^d}\Big\{\left(\log\big(3+\norm{t}\big)\right)^{-1/2} \va{\partial^b\!X^{\eta}(t,\om)}\Big\}<+\infty \,\,\,\text{ if }\,\,\, \eta\ne 0 \text{ or }b\neq 0, \label{propo:behav:eq1}
\end{equation}
and
\begin{equation}
\label{propo:behav:eq2}
\sup_{t\in\R^d }\Big\{\norm{t}^{-a'}\left(\log\log\big(3+\norm{t}\big)\right)^{-1/2}\va{X^{0}(t,\om)}\Big\}<+\infty.
\end{equation}
\end{enumerate}
\end{thm}

It easily follows from Remark~\ref{rem:wavrepeta} and Theorem~\ref{propo:behav} that:

\begin{coro}
\label{cor:behav} 
The exponents $a'\in (0,1)$ and $a_1,\ldots, a_d\in (0,+\infty)$  are the same as in Definition~\ref{def:adm}. Let $ b=(b_1,\dots,b_d)\in\Z_+^d$ be arbitrary and such that $b_l<a_l$, for all $l\in\{1,\ldots,d\}$. Then, for each fixed  $\delta\in (0,+\infty)$ and $\om\in\Om_1^*$, the following three results are satisfied (with the convention that $0/0=0$).
%\begin{equation}
%\label{propo:behav:eq0}
  %\left\{
%    \begin{array}{l @{\,\,\,\text{ if }\,\,\,} l}
 %   b_l\in\big\{0,1,\dots,[a_l]\big\} &  \eta_l=1 \\ 
 %   b_l\in\Z_+&  \eta_l=0,
 %   \end{array}
%  \right.
%\end{equation}
\begin{enumerate}
\item When $\alpha\in (0,1)$ one has 
\begin{equation}
\sup_{t \in \R^d }\Big\{\va{\partial^b\!X(t,\om)}\Big\}<+\infty \,\,\,\text{ if }\,\,\, b\neq 0, \label{cor:behav:eq1bis}
\end{equation}
and
\begin{equation}
\label{cor:behav:eq2bis}
\sup_{\norm{t}\ge 1}\Big\{\norm{t}^{-a'}\left(\log\big(3+\norm{t}\big)\right)^{-1/\alpha-\delta} \va{X(t,\om)}\Big\}<+\infty.
\end{equation}

\item When $\alpha\in [1,2)$ one has
\begin{equation}
\sup_{t\in\R^d}\Big\{\left(\log\big(3+\norm{t}\big)\right)^{-1/2} \va{\partial^b\!X(t,\om)}\Big\}<+\infty \,\,\,\text{ if }\,\,\, b\neq 0, \label{cor:behav:eq1ter}
\end{equation}
and
\begin{equation}
\label{coro:behav:eq2ter}
\sup_{\norm{t}\ge 1}\Big\{ \norm{t}^{-a'}\left(\log\big(3+\norm{t}\big)\right)^{-1/\alpha-\delta} \va{X(t,\om)}\Big\}<+\infty.
\end{equation}
\item When $\alpha=2$ one has
\begin{equation}
\sup_{t\in\R^d}\Big\{\left(\log\big(3+\norm{t}\big)\right)^{-1/2} \va{\partial^b\!X(t,\om)}\Big\}<+\infty \,\,\,\text{ if }\,\,\, b\neq 0, \label{cor:behav:eq1}
\end{equation}
and
\begin{equation}
\label{coro:behav:eq2}
\sup_{\norm{t}\ge 1}\Big\{\norm{t}^{-a'}\left(\log\log\big(3+\norm{t}\big)\right)^{-1/2}\va{X(t,\om)}\Big\}<+\infty.
\end{equation}
\end{enumerate}
\end{coro}

\begin{proof}[Proof of Theorem~\ref{propo:behav}]
We restrict to the case $\alpha=2$; the strategy of the proof remains the same in the other cases, except that \eqref{le:eps-JK:eq1}, or \eqref{le:eps-JK:eq2}, has to be used instead of \eqref{le:eps-JK:eq3}. 

{\underline{\em Part I: we show \eqref{propo:behav:eq1} when $\eta\neq0$.}} In view of \eqref{split}, it is enough to prove the existence of a positive finite constant $C_1(\om)$, such that, for all $ t\in\R^d $, one has
\begin{equation}
\label{thm:behav:proof:eq0}
\va{\partial^bY^{\eta}(t,\om)}\leq C_1(\om) \logr{\norm{t}}.
\end{equation} 
It follows from \eqref{thm:etacvu:eq3}, \eqref{le:phi:eq3}, \eqref{le:eps-JK:eq3} and~\eqref{prop:psij:eq2} that
\begin{equation*}
\va{\partial^bY^{\eta}(t,\om)}\leq C_2(\om) \sum_{J\in\Geta}\sum_{K\in\Z^d}\logr{\sum_{l=1}^d\va{j_l}+\sum_{l=1}^d\va{k_l}}\prod_{l=1}^d\frac{2^{(1-\eta_l)j_l(b_l+1/2)}2^{-j_l\eta_l (a_l-b_l)}}{\left(2+\va{2^{j_l}t_l-k_l}\right)^{p_*}},
\end{equation*}
where $C_2(\om)$ is a positive finite constant not depending on $t$.
Next, using~\eqref{le:int1:eq1} and the inequality 
\begin{equation}
\label{thm:behav:proof:eq0bis}
\norm{t}\ge\max_{1\le l\le d} \va{t_l},
\end{equation}
 we get that 
\begin{equation}
\label{thm:behav:proof:eq2}
\va{\partial^bY^{\eta}(t,\om)}\leq C_3(\om)\sum_{J\in\Geta}\logr{\sum_{l=1}^d\va{j_l}+\norm{t}\sum_{l=1}^d2^{j_l}}\prod_{l=1}^d2^{(1-\eta_l)j_l(b_l+1/2)}2^{-j_l\eta_l (a_l-b_l)},
\end{equation}
where $C_3(\om)$ is a positive finite constant not depending on $t$. Finally, in view of \eqref{subadd} and of the inequalities
\begin{equation}
\label{thm:behav:proof:eq2bis}
\norm{t}\sum_{l=1}^d2^{j_l}\le 2\norm{t}\sum_{l=1}^d2^{j_l}\le\norm{t}^2+\bigg(\sum_{l=1}^d2^{j_l}\bigg)^2,
\end{equation}
one can deduce from \eqref{thm:behav:proof:eq2}, \eqref{def:geta:eq1}, and~\eqref{thm:etacvu:eq1} that \eqref{thm:behav:proof:eq0} holds. This implies that \eqref{propo:behav:eq1} is satisfied when $\eta\neq0$.
%$C_4:=C_3\mathbf{c}_2^d\prod_{l=1}^d\sum_{j_l\in\Z_{\eta_l}}2^{(1-\eta_l)j_l/2}2^{-j_l\eta_l a_l}\logr{\va{j_l}+2^{j_l}}$.

{\underline{\em Part II: we show \eqref{propo:behav:eq1} when $\eta=0$ and $b\ne 0$.}} We know from the assumptions that the multi-index $b$ has at least one non vanishing coordinate; it is denoted by $ b_s$. Thus, using \eqref{thm:lfcvu:eq2}, the triangle inequality, \eqref{le:eps-JK:eq3}, \eqref{prop:psij:eq1}, \eqref{le:int1:eq1}, \eqref{thm:behav:proof:eq0bis}, \eqref{subadd}, and \eqref{le:majjj:eq2bis}, one gets, for all $t\in\R^d$, that
\begin{eqnarray*}
\nonumber\va{\partial^b\!X^{0}(t,\om)} &\le &\sum_{J\in\Z_+^d}\sum_{K\in\Z^d}\va{\partial^b\left(\Psi_{-J}\big(2^{-J}\cdot-K\big)-\Psi_{-J}\big(-K\big)\right)\!(t)}\va{\eps_{-J,K}(\om)}\\
\nonumber &=& \sum_{J\in\Z_+^d}\sum_{K\in\Z^d}\va{\partial^b\Psi_{-J}\big(2^{-J}t-K\big)}\va{\eps_{-J,K}(\om)}\prod_{l=1}^d2^{-j_lb_l}\\
&\leq & C_4(\om) \sum_{J\in\Z_+^d}2^{-j_s(1-a')}\left(2^{-j_1}+\dots+2^{-j_d}\right)^{-d/2}\logr{d\norm{t}+\sum_{l=1}^d j_l}\prod_{l=1}^d2^{-j_l/2}\nonumber\\
&\leq & C_5(\om) \sum_{J\in\Z_+^d}\logr{\norm{t}}2^{-j_s(1-a')}\left(2^{-j_1}+\dots+2^{-j_l}\right)^{-d/2}\prod_{l=1}^d2^{-j_l/2}\logr{j_l}\label{thm:behav:proof:eqA}\\
&\leq & C_6(\om) \logr{\norm{t}},
\end{eqnarray*}
where $C_4(\om)$, $C_5(\om)$ and $C_6(\om)$ are positive finite constants not depending on $t$. This shows that~\eqref{propo:behav:eq1} holds when $\eta=0$ and $b\ne 0$.

{\underline{\em Part III: we show \eqref{propo:behav:eq2}.}} First notice that, it can easily be derived from the fact that $X^0(\cdot,\om)$ is an infinitely differentiable function on $\R^d$ vanishing at $0$ (see Proposition~\ref{thm:lfcvu}), that 
\begin{equation}
\label{propo:behav:eq50}
\sup_{\norm{t}\le 2}\Big\{\norm{t}^{-a'}\left(\log\log\big(3+\norm{t}\big)\right)^{-1/2}\va{X^{0}(t,\om)}\Big\}<+\infty.
\end{equation}
So, in the sequel, we fix an arbitrary $t\in\R^d$, and we always assume that $\norm{t} > 2$. 
Let then $\Gamma_{\text{inf}}(t)$ and $\Gamma_{\text{sup}}(t)$ be the two, non-empty and disjoint, sets of indices $J\in\Z_+^d$ defined as
\begin{equation}
\label{propo:behav:proof:eq11}
\Gamma_{\text{sup}}(t):=\left\{J=(j_1,\dots,j_d)\in\Z_+^d: 2^{\min\{j_1,\ldots, j_d\}}>\norm{t}\right\},
\end{equation}
and
\begin{equation}
\label{propo:behav:proof:eq10}
\Gamma_{\text{inf}}(t):=\left\{J=(j_1,\dots,j_d)\in\Z_+^d: 2^{\min\{j_1,\ldots, j_d\}}\leq\norm{t}\right\}.
\end{equation}
Thus, it follows from \eqref{thm:lfcvu:eq2} (with $b=0$) and from the equality $\Z_+^d=\Gamma_{\text{sup}}(t)\cup\Gamma_{\text{inf}}(t)$ (disjoint union) that
\begin{equation}
\label{propo:behav:proof:eq13bis}
X^{0}(t)=X_\text{sup}^{0}(t)+X_\text{inf}^{0}(t),
\end{equation}
where
\begin{equation}
\label{propo:behav:proof:eq12}
X_\text{sup}^{0}(t)=\sum_{(J,K)\in\Gamma_{\text{sup}}(t)\times\Z^d}\big(\Psi_{-J}(2^{-J}t-K)-\Psi_{-J}(-K)\big)\eps_{{-J},K}(\om),
\end{equation}
and
\begin{equation}
\label{propo:behav:proof:eq13}
X_\text{inf}^{0}(t)=\sum_{(J,K)\in\Gamma_{\text{inf}}(t)\times\Z^d}\big(\Psi_{-J}(2^{-J}t-K)-\Psi_{-J}(-K)\big)\eps_{{-J},K}(\om).
\end{equation}
From now on, our goal is to derive appropriate upper-bounds for $X_\text{sup}^{0}(t)$ and $X_\text{inf}^{0}(t)$. 

First, we focus on $X_\text{sup}^{0}(t)$. In view of \eqref{propo:behav:proof:eq11}, when $J=(j_1,\dots,j_d)\in\Gamma_{\text{sup}}(t)$, then, for any $l\in\{1,\dots,d\}$, one has $\va{2^{-j_l}t_l} < 1 $, the $t_l$'s being the coordinates of $t$. Thus, using the triangle inequality, we get that 
\begin{equation}
\label{propo:behav:proof:eq13ter}
\prod_{l=1}^d\big (2+|2^{-j_l}t_l-k_l|\big) > \prod_{l=1}^d\big (1+|k_l|\big), \quad\text{for all $K=(k_1,\ldots, k_d)\in\Z^d$.}
\end{equation}
Next applying, as in \eqref{thm:wavrep:proof:eqA}, the Mean Value Theorem to $\Psi_{-J}(2^{-J}t-K)-\Psi_{-J}(-K)$, and using \eqref{propo:behav:proof:eq12}, \eqref{prop:psij:eq1}, \eqref{propo:behav:proof:eq13ter}, \eqref{le:eps-JK:eq3} and~\eqref{subadd}, we obtain that
\begin{equation}
\label{propo:behav:proof:eq14}
\va{X_\text{sup}^{0}(t)}\leq C_7(\om)\norm{t}\sum_{r=1}^d\sum_{J\in\Gamma_{\text{sup}}(t)}2^{-j_r}\left(2^{-j_1}+\dots+2^{-j_d}\right)^{-a'-d/2}\logr{\sum_{l=1}^d j_l}\prod_{l=1}^d{2^{-j_l/2}},
\end{equation}
where $C_7(\om)$ is a positive finite constant not depending on $t$. Next, for every fixed $m\in\{1,\dots,d\}$, we let $\Gamma_{\text{sup}}^m (t)$ be the subset of $\Gamma_{\text{sup}}(t)$ defined as 
\begin{equation}
\label{propo:behav:proof:eq15}
\Gamma_{\text{sup}}^m (t):=\Big\{J=(j_1,\dots,j_d)\in\Gamma_{\text{sup}}(t): j_m=\min\{j_1,\ldots, j_d\}\Big\}.
\end{equation}
Observe that, in view of \eqref{propo:behav:proof:eq11} and \eqref{propo:behav:proof:eq15}, for each fixed $m\in\{1,\dots,d\}$, one has 
\begin{equation}
\label{propo:behav:proof:eq15bis}
\Gamma_{\text{sup}}^m (t)=\Big\{J=(j_1,\dots,j_d)\in\Z_+^d: \mbox{ for all $l\in\{1,\ldots,d\}$,\, } j_l \ge j_m \ge N(t)+1\Big\},
\end{equation}
where 
\begin{equation}
\label{propo:behav:proof:eq15ter}
N(t):=\big\lfloor\log(\norm{t})/\log(2)\big\rfloor
\end{equation}
 is the integer part of $ \log(\norm{t})/\log(2)$. Also, observe that one has $\Gamma_{\text{sup}}(t)=\bigcup_{m=1}^d \Gamma_{\text{sup}}^m (t)$. Combining this equality with \eqref{propo:behav:proof:eq14} and \eqref{propo:behav:proof:eq15bis}, we get 
\begin{eqnarray}
\label{propo:behav:proof:eq16}
&& \va{X_\text{sup}^{0}(t)} \nonumber\\
&& \leq  d\,C_7(\om)\norm{t}\sum_{m=1}^d\sum_{J\in\Gamma_{\text{sup}}^m (t)}2^{j_m(a'+d/2-1)}\,\logr{\sum_{l=1}^d j_l}\prod_{l=1}^d{2^{-j_l/2}}\nonumber\\
&&= d^2\,C_7(\om)\norm{t}\sum_{J\in\Gamma_{\text{sup}}^1 (t)}2^{j_1 (a'+d/2-1)}\,\logr{\sum_{l=1}^d j_l}\prod_{l=1}^d{2^{-j_l/2}}\nonumber\\
&& =d^2\,C_7(\om)\norm{t}\sum_{j_1=N(t)+1}^{+\infty} 2^{j_1 (a'+d/2-3/2)}\sum_{j_2=j_1}^{+\infty}\ldots\sum_{j_d=j_1}^{+\infty}\logr{j_1+\sum_{l=2}^d j_l}\prod_{l=2}^d{2^{-j_l/2}}.
\end{eqnarray}
Now, we recall a useful inequality (which can easily be derived from \eqref{subadd}): let $\nu$ be an arbitrary fixed positive real number, there exists a finite constant $c_8$, only depending on $\nu$, such that, for all $(q,\theta)\in\Z_+\times \R_+$, one has 
\begin{equation}
\label{propo:behav:proof:eq17}
\sum_{j=q}^{+\ii}2^{-j\nu}\logr{\theta+j}\leq c_8 2^{-q\nu}\logr{\theta+q}.
\end{equation}
Therefore, for each $(j_1,\lambda)\in\Z_+\times \R_+$, one has 
\begin{equation}
\label{propo:behav:proof:eq17bis}
\sum_{j_2=j_1}^{+\infty}\ldots\sum_{j_d=j_1}^{+\infty}\logr{\lambda+\sum_{l=2}^d j_l}\prod_{l=2}^d{2^{-j_l/2}}\le c_8 ' 2^{-j_1(d-1)/2}\,\logr{\lambda+(d-1) j_1},
\end{equation}
where $c_8'$ is a finite constant not depending on $(j_1,\lambda)$. Next, combining~\eqref{propo:behav:proof:eq16} and~\eqref{propo:behav:proof:eq17bis} (with $\lambda=j_1$), we get that,
\begin{equation}
\label{propo:behav:proof:eq18}
\va{X_\text{sup}^{0}(t)}\leq C_9(\om)\norm{t} \sum_{j_1=N(t)+1}^{+\ii}2^{-j_1 (1-a')}\logr{d\, j_1},
\end{equation}
where $C_9(\om)$ is a positive finite constant not depending on $t$. Then, \eqref{propo:behav:proof:eq18}, \eqref{propo:behav:proof:eq17} and \eqref{propo:behav:proof:eq15ter} entail that
\begin{equation}
\label{propo:behav:proof:eq19}
\va{X_\text{sup}^{0}(t)}\leq C_{10}(\om)\norm{t}^{a'}\sqrt{\log\log(3+\norm{t})},
\end{equation}
for some constant $C_{10}(\om)$ not depending on $t$.

Now, we focus on $X_\text{inf}^{0}(t)$. It results from \eqref{propo:behav:proof:eq13} and the triangle inequality that
\begin{equation}
\label{propo:behav:proof:eq20}
\va{X_\text{inf}^{0}(t)}\leq R_\text{inf}^{0}(t)+S_\text{inf}^{0}(t),
\end{equation}
where
\begin{equation}
\label{propo:behav:proof:eq21}
R_\text{inf}^{0}(t)= \sum_{(J,K)\in\Gamma_{\text{inf}}(t)\times\Z^d}\va{\Psi_{-J}(2^{-J}t-K)}\va{\eps_{-J,K}(\om)}
\end{equation}
and
\begin{equation}
\label{propo:behav:proof:eq21bis}
S_\text{inf}^{0}(t)= \sum_{(J,K)\in\Gamma_{\text{inf}}(t)\times\Z^d}\va{\Psi_{-J}(-K)}\va{\eps_{-J,K}(\om)}.
\end{equation}
Next, for every fixed $m\in\{1,\dots,d\}$, we denote by $\Gamma_{\text{inf}}^m (t)$ the subset of $\Gamma_{\text{inf}}(t)$ defined as 
\begin{equation}
\label{propo:behav:proof:eq23}
\Gamma_{\text{inf}}^m (t):=\Big\{J=(j_1,\dots,j_d)\in\Gamma_{\text{inf}}(t): j_m=\min\{j_1,\ldots, j_d\}\Big\}.
\end{equation}
Observe that, in view of \eqref{propo:behav:proof:eq10}, \eqref{propo:behav:proof:eq23} and \eqref{propo:behav:proof:eq15ter}, for each fixed $m\in\{1,\dots,d\}$, one has 
\begin{equation}
\label{propo:behav:proof:eq23bis}
\Gamma_{\text{inf}}^m (t)=\Big\{J=(j_1,\dots,j_d)\in\Z_+^d:  j_m \le N(t)\mbox{ and for all $l\in\{1,\ldots,d\}$,\, } j_l \ge j_m \Big\}.
\end{equation}
Also, observe that one has $\Gamma_{\text{inf}}(t)=\bigcup_{m=1}^d \Gamma_{\text{inf}}^m (t)$. Combining this equality with \eqref{propo:behav:proof:eq21}, \eqref{prop:psij:eq1}, \eqref{le:eps-JK:eq3}, \eqref{le:int1:eq1}, \eqref{propo:behav:proof:eq23bis}, and~\eqref{propo:behav:proof:eq17bis} (where $\lambda=2^{-j_1}\norm{t}+j_1$), we obtain
\begin{eqnarray}
\label{propo:behav:proof:eq22} 
&& R_\text{inf}^{0}(t)\nonumber\\
&& \leq C_{11}(\om)\sum_{J\in\Gamma_{\text{inf}}(t)}\left(2^{-j_1}+\dots+2^{-j_d}\right)^{-a'-d/2}\logr{\sum_{l=1}^d\Big (j_l+2^{-j_l}\va{t_l}\Big)}\prod_{l=1}^d{2^{-j_l/2}}\nonumber\\
&& \leq C_{11}(\om)\sum_{m=1}^d\sum_{J\in\Gamma_{\text{inf}}^m (t)}2^{j_m(a'+d/2)}\logr{d\,2^{-j_m}\norm{t}+\sum_{l=1}^d j_l}\prod_{l=1}^d{2^{-j_l/2}}\nonumber\\
&& =d\, C_{11}(\om)\sum_{J\in\Gamma_{\text{inf}}^1 (t)}2^{j_1(a'+d/2)}\logr{d\,2^{-j_1}\norm{t}+\sum_{l=1}^d j_l}\prod_{l=1}^d{2^{-j_l/2}}\nonumber\\
&& =d\, C_{11}(\om)\sum_{j_1=0}^{N(t)} 2^{j_1(a'+d/2-1/2)}\sum_{j_2=j_1}^{+\infty}\ldots\sum_{j_d=j_1}^{+\infty}\logr{2^{-j_1}\norm{t}+j_1+\sum_{l=2}^d j_l}\prod_{l=2}^d{2^{-j_l/2}}\nonumber\\
&& \leq C_{12}(\om)\sum_{j_1=0}^{N(t)}2^{j_1a'}\logr{d\,2^{-j_1}\norm{t}+d\, j_1},
\end{eqnarray} 
where $C_{11}(\om)$ and $C_{12}(\om)$ are two finite constants not depending on $t$. On the other hand, thanks to the assumption $\norm{t} > 2$, standard computations allow to show that the function $ z\mapsto \sqrt{\log(3+d\,2^{-z}\norm{t}+d\,z)}$ is non-decreasing on $\R_+$. This, in particular, implies that 
\[
\logr{d\,2^{-j_1}\norm{t}+d\, j_1}\le\logr{d\,2^{-N(t)}\norm{t}+d\, N(t)}, \quad\text{for all $j_1\in\{0,\ldots, N(t)\}$,}
\]
and, consequently that
\begin{equation}
\label{propo:behav:proof:eq23ter} 
\sum_{j_1=0}^{N(t)}2^{j_1a'}\logr{d\,2^{-j_1}\norm{t}+d\, j_1}\le c_{13} 2^{N(t)a'}\logr{d\,2^{-N(t)}\norm{t}+d\, N(t)},
\end{equation}
where the finite constant $c_{13}:=2^{a'}\big(2^{a'}-1\big)^{-1}$.
Next, combining \eqref{propo:behav:proof:eq22} and \eqref{propo:behav:proof:eq23ter} with \eqref{propo:behav:proof:eq15ter}, it follows that
\begin{equation}
\label{propo:behav:proof:eq25}  
R_\text{inf}^{0}(t) \leq C_{14}(\om)\norm{t}^{a'}\sqrt{\log\log(3+\norm{t})},
\end{equation}
where $ C_{14}(\om)$ is a finite constant not depending on $t$. Similarly to \eqref{propo:behav:proof:eq25}, it can be shown that 
\begin{equation}
\label{propo:behav:proof:eq26}  
S_\text{inf}^{0}(t) \leq C_{14}(\om)\norm{t}^{a'}\sqrt{\log\log(3+\norm{t})}.
\end{equation}
Next, combining \eqref{propo:behav:proof:eq25} and \eqref{propo:behav:proof:eq26} with \eqref{propo:behav:proof:eq20}, we get that
\begin{equation}
\label{propo:behav:proof:eq27} 
\va{X_\text{inf}^{0}(t)}\leq 2C_{14}(\om) \norm{t}^{a'}\sqrt{\log\log(3+\norm{t})}.
\end{equation}
Finally \eqref{propo:behav:proof:eq27}, \eqref{propo:behav:proof:eq19} and~\eqref{propo:behav:proof:eq13bis} imply that 
\begin{equation}
\label{propo:behav:eq51}
\sup_{\norm{t} > 2}\Big\{\norm{t}^{-a'}\left(\log\log\big(3+\norm{t}\big)\right)^{-1/2}\va{X^{0}(t,\om)}\Big\}<+\infty.
\end{equation}
Then using \eqref{propo:behav:eq50} and \eqref{propo:behav:eq51} we obtain \eqref{propo:behav:eq2}.

%Suppose now that $\alpha\in(0,1)$. Then, following the same lines, one has
%\begin{eqnarray*}
%\va{\partial^bY^{\eta}(t,\om)}&\leq& C_{22}(\om) \sum_{J\in\Geta}\sum_{K\in\Z^d}\prod_{l=1}^d\frac{2^{(1-\eta_l)j_l(b_l+1/\alpha)}2^{-j_l\eta_l (a_l-b_l)}\puiss{\va{j_l}}}{\left(2+\va{2^{j_l}t_l-k_l}\right)^{\Lalpha}}\\
%&\leq& C_{23}(\om)\sum_{J\in\Geta}\prod_{l=1}^d2^{(1-\eta_l)j_l(b_l+1/\alpha)}2^{-j_l\eta_l (a_l-b_l)}\puiss{\va{j_l}}\\
%&<&+\ii.
%\end{eqnarray*}
%
%We now focus on the case $\eta=0 $ and  $ b\neq0$. Then, there is $ s\in\{1,\dots,d\}$ such that $ b_s\neq 0 $. Thus, combining  \eqref{prop:psij:eq1}, \eqref{le:int1:eq1} and~\eqref{subadd} we get that,
%\begin{eqnarray*}
%&&\sum_{J\in\Geta}\sum_{K\in\Z_d}\va{\partial^b\left(\Psi_{-J}\big(2^{-J}\cdot-K\big)-\Psi_{-J}\big(-K\big)\right)\!(t)}\va{\eps_{\alpha,-J,K}(\om)}\\
%&&=\sum_{J\in\Geta}\sum_{K\in\Z^d}\va{\partial^b\Psi_{-J}\big(2^{-J}t-K\big)}\va{\eps_{\alpha,-J,K}(\om)}\prod_{l=1}^d2^{-j_lb_l}\\
%&&\leq C_{24}(\om) \sum_{J\in\Geta}2^{-j_s(1-a')}\left(2^{-j_1}+\dots+2^{-j_d}\right)^{-d/2}\prod_{l=1}^d2^{-j_l/2}\\
%&&<+\ii
%\end{eqnarray*}
\end{proof}

\appendix

\section{Proofs of Proposition~\ref{prop:psij} and Lemma~\ref{le:int1}}
\label{app:psijalpha}

\begin{proof}[Proof of Proposition~\ref{prop:psij}] Let us first assume that $J\in\Z^d$, and show the infinite differentiability of $\PsialphaJ$ and relation \eqref{prop:psij:eq0}. We denote by $\Lambda_{\alpha,J}$ the integrand in \eqref{psialpha}, that is, for all $x\in\R^d$ and  $\xi\in\R^d$, we set,
\begin{equation}
\label{prop:psij:proof:eq1}
\Lambda_{\alpha,J}(x,\xi):=2^{(j_1+\dots+j_d)/\alpha}e^{ix\cdot\xi}\,f(2^J\xi)\wpsi_{0,0}(\xi).
\end{equation}
Observe that $\Lambda_{\alpha,J}$ is an infinitely differentiable function on $\R^d$ with respect to the variable $x$, and that, for any $b\in\Z_+^d$,
\begin{equation}
\label{prop:psij:proof:eq1bis}
\partial_x^{b}\Lambda_{\alpha,J}(x,\xi)=2^{(j_1+\dots+j_d)/\alpha}\,i^{\mathrm{l}(b)}\xi^b e^{ix\cdot \xi}\,f(2^{J}\xi)\wpsi_{0,0}(\xi).
\end{equation}
Thus, in view of a classical rule of differentiation under the integral symbol, in order to show that $\PsialphaJ$ itself is infinitely differentiable on $\R^d$ and satisfies \eqref{prop:psij:eq0}, it is enough to prove that, for any $b\in\Z^d_+$, there exists $G_{J}^{b}\in\Lp{1}{\R^d}$, which does not depend on $x$, such that the inequality:
\begin{equation}
\label{prop:psij:proof:eq2}
\big|\partial_x^{b}\Lambda_{J}(x,\xi)\big|\leq G_{J}^{b}(\xi),
\end{equation}
holds for almost all $\xi\in\R^d$. Recall that $\KK$ is the compact subset of $\R$ defined as $\KK:=\big\{\lambda\in\R : 2\pi/3\leq |\lambda|\le 8\pi/3\big\}$; also recall that $\wpsi_{0,0}$ is a $C^\infty$ function on $\R^d$ with a compact support included in $\KK^d$. Thus the smoothness assumption on the function $f$ (that is $(\H_1)$ in Definition~\ref{def:adm}) implies that the supremum $\big\|f(2^j\cdot)\wpsi_{0,0}(\cdot)\big\|_\infty:=\sup_{\xi\in\KK^d}\big|f(2^J\xi)\wpsi_{0,0}(\xi)\big |$ is finite. Then, it turns out that a function $G_{J}^{b}$, belonging to $\Lp{1}{\R^d}$ and satisfying \eqref{prop:psij:proof:eq2}, can simply be obtained by setting, for all $\xi\in\R^d$, 
$$
G_{J}^{b}(\xi)=2^{(j_1+\dots+j_d)/\alpha}\Big(\frac{8\pi}{3}\Big)^{\mathrm{l}(b)}\big\|f(2^j\cdot)\wpsi_{0,0}(\cdot)\big\|_\infty\ind{\KK^d}(\xi).
$$

Let us now prove that parts $(i)$ and $(ii)$ of the proposition hold. For the sake of simplicity, we restrict to the case where $x=(x_1,\dots, x_d)\in\R_+^d$; the other cases can be treated similarly. It easily follows from \eqref{prop:psij:eq0}, \eqref{eq3:blm} and \eqref{eq4:blm} that, for every $T\in(0,+\infty)$, $J\in\Z^d$ and $x\in\R_+^d$,
\begin{equation}
\label{prop:psij:proof:eq2bis}
\va{\partial^{b}\PsiJ(x)}=2^{(j_1+\dots+j_d)/\alpha}\va{\int_{\KK^d}\left(\prod_{l=1}^d e^{i(1+T+x_l)\xi_l}\,\widehat{\Phi}_l(\xi_l)\right)f(2^{J}\xi)\mathrm d\xi},
\end{equation} 
where $\widehat{\Phi}_l(\xi_l):=e^{-i(1+T)\xi_l}\,\xi_l^{b_l}\,\wh{\psi^1}(\xi_l)$. Next, we set  $R_{J}(\xi):=f(2^{J}\xi)\prod_{l=1}^d\widehat{\Phi}_l(\xi_l)$, for all $\xi\in \big(\R\setminus\{0\}\big)^d$. Observe that, similarly to $\wh{\psi^1}$ (see the beginning of Section~\ref{secwavrep}), $\widehat{\Phi}_l$ is a $C^\infty$ function on $\R$ having a compact support included in $\KK$. Thus, using the condition $(\H_1)$ in Definition~\ref{def:adm}, it turns out that the partial derivative $\partial^{(p_*,\dots,p_*)}R_J$ is a well-defined continuous function on $\big(\R\setminus\{0\}\big)^d$ having a compact support included in $\KK^d$. Hence, integrating by parts in \eqref{prop:psij:proof:eq2bis}, we obtain that
\begin{eqnarray}
\label{prop:psij:proof:eq3}
\va{\partial^{b}\PsiJ(x)}&=&2^{(j_1+\dots+j_d)/\alpha}\va{\int_{\KK^d}\left( \left(\partial^{(p_*,\dots,p_*)}R_{J}\right)\!\!(\xi)\prod_{l=1}^d \frac{e^{i(1+T+x_l)\xi_l}}{(1+T+x_l)^{p_*}}\right)\mathrm d\xi}\nonumber\\
&\le & c_1 \frac{2^{(j_1+\dots+j_d)/\alpha}}{\prod_{l=1}^d(1+T+x_l)^{p_*}} \sup_{\xi\in\KK^d}\va{ \left(\partial^{(p_*,\dots, p_*)}R_{J}\right)\!\!(\xi)},
\end{eqnarray}
where the constant $c_1>0$ is the Lebesgue measure of $\KK^d$. On the other hand, using the Leibniz formula , we get, for every $\xi\in\big(\R\setminus\{0\}\big)^d$, that
\begin{equation}
\label{prop:psij:proof:eq4}
\left(\partial^{(p_*,\dots,p_*)} R_J\right)\!\!(\xi)=\sum_{p_1=0}^{p_*}\dots\sum_{p_d=0}^{p_*}\left(\partial^{(p_1,\dots,p_d)}\! f\right)\!\!(2^{J}\xi)\prod_{l=1}^d\binom{p_*}{p_l}2^{j_l p_l}\widehat{\Phi}_l ^{(p_*-p_l)}(\xi_l),
%\left(\partial^{p_*-p_l}\widehat{\Phi}_l\right)\!\!(\xi_l).
\end{equation}
where $\widehat{\Phi}_l ^{(p_*-p_l)}$ is the derivative of order $p_*-p_l$ of $\widehat{\Phi}_l$.  In view of \eqref{prop:psij:proof:eq3}, it turns out that for deriving \eqref{prop:psij:eq1}, it is enough to show that
\begin{equation}
\label{prop:psij:proof:eq3bis}
\sup_{J\in\Z_+^d} \sup_{\xi\in\KK^d}\Big\{ \left(2^{-j_1}+\dots+2^{-j_d}\right)^{a'+d/\alpha} 
\va{\left(\partial^{(p_*,\dots,p_*)} R_{-J}\right)\!\!(\xi)}\Big\}<+\infty,
\end{equation}
and for deriving \eqref{prop:psij:eq2}, it is enough to show that, for all $\eta\in\Upss$,
\begin{equation}
\label{prop:psij:proof:eq3bisbis}
\sup_{J\in\Geta} \sup_{\xi\in\KK^d}\left\{ \prod_{l=1}^d2^{j_l/\alpha}2^{-(1-\eta_l)j_l/\alpha}2^{j_l\eta_l a_l}  
\va{\left(\partial^{(p_*,\dots,p_*)} R_{J}\right)\!\!(\xi)}\right\}<+\infty;
\end{equation}
recall that the sets $\Upss$ and $\Geta$ are defined in \eqref{def:upss} and \eqref{def:geta:eq1}, respectively.

We now focus on the proof of \eqref{prop:psij:proof:eq3bis}. In view of \eqref{prop:psij:proof:eq4} and of the fact that the $\widehat{\Phi}_l ^{(p_*-p_l)}$'s, 
%$\partial^{p_*-p_l}\widehat{\Phi}_l$'s, 
$l=1,\dots,d$ are bounded functions on $\KK$, \eqref{prop:psij:proof:eq3bis} can be obtained by showing that 
 \begin{equation}
\label{prop:psij:proof:eq3ter}
\sup_{p\in\{0,1,2,\ldots, p_*\}^d} \sup_{J\in\Z_+^d} \sup_{\xi\in\KK^d}\Big\{ \left(2^{-j_1}+\dots+2^{-j_d}\right)^{a'+d/\alpha} 2^{-(j_1 p_1+\dots+j_d p_d)}
\va{\left(\partial^{p} f\right)\!(2^{-J}\xi)}\Big\}<+\infty.
\end{equation}
%
%Notice that in view of assumption $ \H_1$, the function $R$ has continuous partial derivatives of order $(2,\ldots,2)$, thus by the Leibniz formula, one has
%\begin{eqnarray}
%&&\label{prop:psij:proof:eq4}\left(\partial^{(2,\dots,2)} R\right)\!\!(\xi)\\
%&&\nonumber=\sum_{p_1=0}^2\dots\sum_{p_d=0}^2\left(\partial^{(p_1,\dots,p_d)}f\right)\!\!(2^{-J}\xi)\prod_{l=1}^d{2\choose p_l}2^{-j_lp_l}\left(\partial^{2-p_l}\widehat{\Phi}_l\right)\!\!(\xi_l).
%\end{eqnarray}
%Notice also that in view of assumption $\H_2$, the support of $ R$ is included in $\KK^d$; thus in the sequel we suppose that $\xi$ belongs to $\KK^d$. Since the $j_1,\dots,j_d$ are non-negative, one obtains,
Observe that, for any $\xi\in\KK^d$ and $J\in\Z_+^d$, one has $\norm{2^{-J}\xi}\leq8\pi\sqrt{d}/3$. Thus, assuming that $p\in\{0,1,2,\ldots, p_*\}^d$ is arbitrary and using \eqref{A3}, one gets that
\begin{equation}
\label{prop:psij:proof:eq3qua}
\va{\partial^p\!f(2^{-J}\xi)}\leq c_2\left(2^{-2j_1}\xi_1^2+\dots+2^{-2j_d}\xi_d^2\right)^{-\frac{a'}{2}-\frac{d}{2\alpha}-\frac{\mathrm{l}(p)}{2}},
\end{equation}
where $c_2$ denotes the constant $c'$ in \eqref{A3} which does not depend on $p$, $J$, and $\xi$. On the other hand, the fact that $\xi\in\KK^d$ implies that 
\begin{equation}
\label{prop:psij:proof:eq3qua-bis}
\min_{1\le l \le d}\va{\xi_l}\geq 2\pi/3\geq 1.
\end{equation}
It follows from these inequalities and from the equality $\mathrm{l}(p)=p_1+\dots+p_d$ that
\begin{eqnarray}
\label{prop:psij:proof:eq3cin}
\left(2^{-2j_1}\xi_1^2+\dots+2^{-2j_d}\xi_d^2\right)^{-\frac{a'}{2}-\frac{d}{2\alpha}-\frac{\mathrm{l}(p)}{2}} &\leq & \left(2^{-2j_1}+\dots+2^{-2j_d}\right)^{-\frac{a'}{2}-\frac{d}{2\alpha}-\frac{\mathrm{l}(p)}{2}}\nonumber\\
&=& \left(2^{-2j_1}+\dots+2^{-2j_d}\right)^{-\frac{a'}{2}-\frac{d}{2\alpha}}\prod_{l=1}^d\left(2^{-2j_1}+\dots+2^{-2j_d}\right)^{-\frac{p_l}{2}}\nonumber\\
&\leq & c_3\left(2^{-j_1}+\dots+2^{-j_d}\right)^{-a'-\frac{d}{\alpha}}2^{j_1 p_1+\dots+j_d p_d},
\end{eqnarray}
where $c_3>0$ is a constant only depending on $d$, $a'$ and $\alpha$. \eqref{prop:psij:proof:eq3ter}  results from \eqref{prop:psij:proof:eq3qua} and \eqref{prop:psij:proof:eq3cin}. 

We now focus on the proof of \eqref{prop:psij:proof:eq3bisbis}, where $\eta\in\Upss$ is arbitrary and fixed. In view of  \eqref{prop:psij:proof:eq4} and of the fact that the $\widehat{\Phi}_l ^{(p_*-p_l)}$'s,
%$\partial^{p_*-p_l}\widehat{\Phi}_l$'s, 
$l=1,\dots,d$, are bounded functions on $\KK$, \eqref{prop:psij:proof:eq3bisbis} can be obtained by showing that
\begin{equation}
\label{prop:psij:proof:eq7}
\sup_{p\in\{0,1,2,\ldots, p_*\}^d} \sup_{J\in\Geta} \sup_{\xi\in\KK^d}\Big\{2^{j_1p_1+\dots+j_d p_d}
\va{\left(\partial^{p} f\right)\!(2^{J}\xi)} \prod_{l=1}^d2^{j_l/\alpha}2^{-(1-\eta_l)j_l/\alpha}2^{j_l\eta_l a_l} \Big\}<+\infty.
\end{equation}
Let $p=(p_1,\dots,p_d)\in\{0,1,2,\ldots, p_*\}^d$, $J=(j_1,\dots, j_d)\in\Geta$ and $\xi=(\xi_1,\dots,\xi_d)\in\KK^d$ be arbitrary. Observe that, we know from the definition of $\Geta$ (see \eqref{def:geta:eq1} and \eqref{def:geta:eq2}) that $J$ has at least one positive coordinate, let us say $j_r$. Therefore, using \eqref{prop:psij:proof:eq3qua-bis}, one gets that $\norm{2^J\xi}\geq\va{2^{j_r}\xi_r}\geq 2\pi/3$. Then, it follows from \eqref{A2} that
\begin{equation}
\label{prop:psij:proof:eq9}
\va{\partial^p\!f(2^J\xi)}\leq c_4 \prod_{l=1}^d\left(1+2^{j_l}\va{\xi_l}\right)^{-a_l-\frac{1}{\alpha}-p_l},
\end{equation}
where $c_4$ denotes the constant $c$ in~\eqref{A2} which does not depend on $p$, $J$, and $\xi$.
We now provide a convenient upper bound for the right hand side in \eqref{prop:psij:proof:eq9}. To this end, we notice that $\{1,\dots,d\}=\mathbb{L}_+\cup\mathbb{L}_-$, where the disjoint sets $\mathbb{L}_+$ and $\mathbb{L}_-$ are defined by $\mathbb{L}_+=\left\{l\in\{1,\dots,d\}:  \eta_l=1\right\}$ and $\mathbb{L}_-=\left\{l\in\{1,\dots,d\}:  \eta_l=0\right\}$.
Then, using \eqref{prop:psij:proof:eq3qua-bis} and the fact that $-j_l\ge 0$ when $l\in\mathbb{L}_-$, one obtains that 
\begin{equation}
\label{prop:psij:proof:eq11}
\prod_{l\in\mathbb{L}_+}\left(1+2^{j_l}\va{\xi_l}\right)^{-a_l-\frac{1}{\alpha}-p_l} \leq \prod_{l\in\mathbb{L}_+}2^{-j_l \left(a_l+\frac{1}{\alpha}+p_l\right)}\le 2^{-(j_1p_1+\dots+j_d p_d)}\prod_{l=1}^d2^{-j_l\eta_l \left(a_l+\frac{1}{\alpha}\right)}.
\end{equation}
On the other hand, one clearly has that
\begin{equation}
\label{prop:psij:proof:eq11bis}
\prod_{l\in\mathbb{L}_-}\left(1+2^{j_l}\va{\xi_l}\right)^{-a_l-\frac{1}{\alpha}-p_l}\leq 1,
\end{equation}
with the convention that $\prod_{l\in\mathbb{L}_-}\dots=1$, when $\mathbb{L}_-$ is the empty set. Next, combining \eqref{prop:psij:proof:eq11} and \eqref{prop:psij:proof:eq11bis}, it follows that:
\begin{equation}
\label{prop:psij:proof:eq12}
\prod_{l=1}^d\left(1+2^{j_l}\va{\xi_l}\right)^{-a_l-\frac{1}{\alpha}-p_l}\le 2^{-(j_1p_1+\dots+j_d p_d)}\prod_{l=1}^d 2^{-\eta_l j_l/\alpha}\, 2^{-j_l\eta_l a_l}= 2^{-(j_1p_1+\dots+j_d p_d)}\prod_{l=1}^d 2^{- j_l/\alpha}\,2^{(1-\eta_l) j_l/\alpha}\, 2^{-j_l\eta_l a_l}.
\end{equation}
Finally \eqref{prop:psij:proof:eq7} results from \eqref{prop:psij:proof:eq9} and \eqref{prop:psij:proof:eq12}.
\end{proof}

\begin{proof}[Proof of Lemma~\ref{le:int1}] One denotes by $\lfloor v\rfloor$ the integer part of $v$, and one sets $w(v)=v-\lfloor v\rfloor$. Then, using the triangle inequality and the inequality $\va{\lfloor v\rfloor}\le \va{v}+1$, one obtains that
\begin{equation}
\label{le:int1:proof:eq3}
\sum_{k\in\Z}\frac{\sqrt{\log(3+\tea+\va{k})}}{\left(2+\va{v-k}\right)^{p_*}}
=\sum_{k\in\Z}\frac{\sqrt{\log(3+\tea+\va{k+\lfloor v\rfloor})}}{\left(2+\va{v-\lfloor v\rfloor-k}\right)^{p_*}}\leq\sum_{k\in\Z}\frac{\sqrt{\log(3+\tea+\va{k}+1+\va{v})}}{\left(2+\va{w(v)-k}\right)^{p_*}}.
\end{equation}
Next, let $c$ be the constant defined as:
\begin{equation}
\label{le:int1:proof:eq2}
c:=2\sup_{w\in [0,1]}\Bigg\{\sum_{k\in\Z}\frac{\sqrt{\log{(4+\va{k})}}}{\left(2+\va{w-k}\right)^{p_*}}\Bigg\}.
\end{equation}
Observe that \eqref{eq:bs} and the inequality $2+\va{w-k}\geq 1+\va{k}$, for all $(k,w)\in\Z\times[0,1]$, imply 
that $c$ is finite. Also, observe that, it follows from \eqref{subadd}, the fact that $w(v)\in [0,1]$, and \eqref{le:int1:proof:eq2} that
\begin{equation}
\label{le:int1:proof:eq4}
\sum_{k\in\Z}\frac{\sqrt{\log(3+\tea+\va{k}+1+\va{v})}}{\left(2+\va{w(v)-k}\right)^{p_*}} \leq 2\sum_{k\in\Z}\frac{\sqrt{\log(4+\va{k})}\sqrt{\log(3+\tea+\va{v})}}{\left(2+\va{w(v)-k}\right)^{p_*}} \leq  c\sqrt{\log(3+\tea+\va{v})}.
\end{equation}
Finally combining \eqref{le:int1:proof:eq3} and \eqref{le:int1:proof:eq4}, one gets \eqref{le:int1:eq1}.
%The inequality \eqref{le:int1:eq2} is a straightforward consequence of the fact that for every $T>0$,
%\[
%\norminf{\sqrt{\log(3+\va{j}+2^{j}\va{\,\cdot\,})}}\leq \sqrt{\log(3+\va{j}+2^{j}T)};
%\]
%and the inequality \eqref{le:int1:eq3} is a straightforward consequence of
%%the fact that for every $(t,j)\in\R^d\times\R^d$, 
%%\[
%%(3+\va{j}+2^{j}\va{t})\leq (3+\va{j}+2^j)(3+\va{t}),
%%\]
%%and
%the sub-additivity of the function $\sqrt{\cdot}$.
%%Since for all positive constant $a$, $\sqrt{\log(4+a)}\geq\sqrt{\log(3+a)}\geq 1$,
%%one has, for each $t_l\in[-T,T]$, $n \in\Z_+$, $m\in\N$,
%%\begin{eqnarray*}
%%&&\sum_{k\in\Z}\frac{\sqrt{\log(3+\va{j}+\va{k})}}{\left(2+\va{2^{j}t-k}\right)^2}\\
%%&&\leq2\sqrt{\log(3+\va{j}+2^{j}T)}\sum_{k\in\Z}\frac{\sqrt{\log(4+\va{k})}}{\left(2+\va{2^{j}t-[2^{j}t]-k}\right)^2}.
%%\end{eqnarray*}
%%Moreover, the series $\sum_{k\in\Z}\frac{\sqrt{\log(4+\va{k})}}{\left(2+\va{\cdot-k}\right)^2}$, of continuous functions on the compact interval $[0,1]$, converges uniformly on $[0,1]$; thus, 
%%\begin{equation*}
%%\sup_{x\in[0,1]}\left\{\sum_{k\in\Z}\frac{\sqrt{\log(4+\va{k})}}{\left(2+\va{x-k}\right)^2}\right\}<+\ii,
%%\end{equation*}
%%and Lemma~\ref{le:int1} holds.
\end{proof}

\section{Proofs of Lemma~\ref{le:majjj} and Proposition~\ref{prop:decompo}}
\label{app:decompo}

\begin{proof}[Proof of Lemma \ref{le:majjj}]
%Let denote by $ \mathfrak{S}_d$\footnote{The set of permutations of $ \{1,\dots,d\} $.}
Assume that the real numbers $a'\in(0,1)$, $\alpha\in(0,2]$, and $\delta>0$ are arbitrary and fixed. Also assume that the positive integer $d$ and $r\in\{1,\dots,d\}$ are arbitrary and fixed. 
Then, for any fixed $r'\in\{1,\dots,d\}$, let $\Gamma_{r'}$ be the set defined as $$\Gamma_{r'}:=\left\{J=(j_1,\dots,j_d)\in\Z_+^d: j_{r'}=\min\{j_1,\dots,j_d\}\right\},$$ and let $S_{r,r'}$ be the positive quantity defined as
\[
S_{r,r'}=\sum_{J\in\Gamma_{r'}}2^{-j_r(1-a')}\left(2^{-j_1}+\dots+2^{-j_d}\right)^{-d/\alpha}\prod_{l=1}^d2^{-j_l/\alpha}\logr{j_l}(1+j_l)^{1/\alpha+\delta}.
\]
The fact that $\Z_+^d=\bigcup_{r'=1}^d \Gamma_{r'}$ implies that
\[
\sum_{J\in\Z_+^d}2^{-j_r(1-a')}\left(2^{-j_1}+\dots+2^{-j_d}\right)^{-d/\alpha}\prod_{l=1}^d2^{-j_l/\alpha}\logr{j_l}(1+j_l)^{1/\alpha+\delta}\leq\sum_{r'=1}^dS_{r,r'}.
\]
On the other hand, standard computations, relying on the definitions of $\Gamma_{r'}$ and $S_{r,r'}$, allow to obtain, for each $r'\in\{1,\dots,d\}$, that 
\[
S_{r,r'}\le \sum_{n=0}^{+\infty} 2^{-n(1+1/\alpha-a'-d/\alpha)}\logr{n}(1+n)^{1/\alpha+\delta}\bigg(\sum_{m=n}^{+\infty}2^{-m/\alpha }\logr{m}(1+m)^{1/\alpha+\delta}\bigg)^{d-1}.
\]
Thus, in order to derive~\eqref{le:majjj:eq1}, it is enough to show that
\[
\sum_{n=0}^{+\infty} 2^{-n(1+1/\alpha-a'-d/\alpha)}\logr{n}(1+n)^{1/\alpha+\delta}\bigg(\sum_{m=n}^{+\infty}2^{-m/\alpha }\logr{m}(1+m)^{1/\alpha+\delta}\bigg)^{d-1}<+\ii.
\]
This can easily be obtained by making use of the inequality
\begin{equation}
\label{le:majjj:proof:eq3}
\sum_{m=n}^{+\ii}2^{-m/\alpha }\logr{m}(1+m)^{1/\alpha+\delta}\leq c 2^{-n/\alpha}\logr{n}(1+n)^{1/\alpha+\delta},
\end{equation}
which holds for any non-negative integer $n$ and for some finite constant $c$ only depending on $\alpha$ and $\delta$. The proof of \eqref{le:majjj:proof:eq3} has been omitted since it is not difficult.
\end{proof}

The proof of Proposition~\ref{prop:decompo} is devided into the following two steps which will be obtained separately.
\begin{enumerate}[\bf{Step} 1.]
\item We show that, for every fixed $t\in\R^d$, there exists $\wt{F}(t,\cdot)$ in $\Lp{\alpha}{\R^d}$ such that, 
for any increasing sequence $(\mathcal{D}_n)_{n\in\N}$ of finite subsets of~$\Z^d\times\Z^d$ which satisfies $\bigcup_{n\in\N}\mathcal{D}_n=\Z^d\times\Z^d$, one has
\begin{equation}
\label{prop:decompo:proof:eq1}
\lim_{n\to+\ii}\Delta_{\alpha}\left(\sum_{(J,K)\in\mathcal{D}_n}\big(\PsialphaJ(2^Jt-K)-\PsialphaJ(-K)\big)\overline{\wh{\psi}_{\alpha,J,K}(\cdot)},\wt{F}(t,\cdot)\right)=0.
\end{equation}
\item We show that, for all $t\in\R^d$ and almost all $\xi\in\R^d$,
$
F(t,\xi)=\wt{F}(t,\xi).
$
%\item We finish the proof using a property of $\alpha$-stable integral.
\end{enumerate}

%\sum_{(J,K)\in\Z^d\times\Z^d}\Delta_{\alpha}\left(\big(\PsialphaJ(2^Jt-K)-\PsialphaJ(-K)\big)\overline{\wh{\psi}_{\alpha,J,K}(\cdot)},0\right)<+\ii
\begin{proof}[Proof of Proposition~\ref{prop:decompo} (Step 1)] In view of Lemma~\ref{rem:conv} and \eqref{prop:geta}, it is enough to show that, for all fixed $t\in\R^d$ and $\eta\in\Upsilon$, one has
\begin{equation}
\label{prop:decompo:proof:eq3}
\sum_{(J,K)\in\Geta\times\Z^d}\Delta_{\alpha}\left(\big(\PsialphaJ(2^Jt-K)-\PsialphaJ(-K)\big)\overline{\wh{\psi}_{\alpha,J,K}(\cdot)},0\right)<+\ii.
\end{equation}
We will study the following 4 cases: 
$$
\alpha\in(0,1) \text{ and } \eta=0, \quad\alpha\in [1,2] \text{ and }\eta=0,\quad \alpha\in(0,1)\text{ and } \eta\ne 0, \quad \alpha\in [1,2] \text{ and } \eta\ne 0.
$$
\indent\underline{Case 1: $\alpha\in(0,1)$ and $\eta=0$.}
Notice that, in this case, one has $J\in\Z^d_{(0)}$, so it can be rewritten as $J=-J'$, where $J'$ belongs to $\Z_+^d$. In the sequel $J'$ is denoted by $J$. Then \eqref{rem:Lalpha:eq1}, \eqref{eq4:blmalpha} and the change of variable $\eta=2^{-J}\xi$ imply that, for all $K\in\Z^d$, one has
\begin{equation}
\label{prop:decompo:proof:eq4}
\Delta_{\alpha}\left(\big(\Psi_{\alpha,-J}(2^{-J}t-K)-\Psi_{\alpha,{-J}}(-K)\big)\overline{\wh{\psi}_{\alpha,{-J},K}(\cdot)},0\right)=c_1\va{\Psi_{\alpha,{-J}}(2^{-J}t-K)-\Psi_{\alpha,{-J}}(-K)}^{\alpha},
\end{equation}
where the constant $c_1:=\left(\int_{\R}|\wh{\psi^1}(\eta)|^\alpha\,\mathrm d\xi\right)^d$ is finite. Next, let $T:=\max_{1\leq l\leq d}{\va{t_l}}$. 
Using the Mean Value Theorem and the triangle inequality, we get that,
\begin{equation}
\label{prop:decompo:proof:eq5}
\va{\Psi_{\alpha,-J}(2^{-J}t-K)-\Psi_{\alpha,-J}(-K)}\leq T\sum_{r=1}^d2^{-j_r}\sup_{s\in [-T,T]^d}\bigg |\frac{\partial\Psi_{\alpha,-J}}{\partial x_r}\big(2^{-J}s-K\big)\bigg|,
\end{equation}
%where $T:=\max_{1\leq l\leq d}{\va{t_l}}$.
% it follows from inequalities \eqref{le:eps-JK:eq1} and \eqref{prop:psij:eq1}, for all $(J,K)\in\Z_+^d\times\Z^d$, there exists $\theta_{J,K}(t):=(\theta_1,\dots,\theta_d)\in(-T,T)$ such that:
%\begin{eqnarray}
%\nonumber&&\va{\Psi_{-J}(2^{-J}t-K)-\Psi_{-J}(-K)}\va{\eps_{-J,K}(\om)}\\
%&&\nonumber\leq C_1(\om) T\sum_{r=1}^d2^{-j_r}\va{\frac{\partial\Psi_{-J}}{\partial x_r}\big(2^{-J}\theta_{J,K}(t)-K\big)}\prod_{s=1}^d\logr{\va{j_s}+\va{k_s}}\\
%&&\leq C_3(\om)\sum_{r=1}^d2^{-j_r}\left(\sum_{l=1}^d2^{-j_l}\right)^{{-a'-d/2}}\prod_{s=1}^d\frac{{2^{-j_s/2}\logr{\va{j_s}+\va{k_s}}}}{\left(1+T+\va{2^{-j_s}\theta_s-k_s}\right)^2}\label{thm:wavrep:proof:eq1bis},
%\end{eqnarray}
%where $C_3(\om):=Tc_2C_1(\om)$: $C_1(\om)$ being the constant in \eqref{le:eps-JK:eq1}, and $c_2$, the constant in \eqref{prop:psij:eq1}.
Moreover, combining \eqref{prop:psij:eq1} with the inequality, 
$$1+T+\va{2^{-j_l}s_l-k_l}\geq1+\va{k_l},\quad\mbox{for all $l\in\{1,\dots,d\}$ and $s_l\in [-T,T]$}, $$ 
we obtain, for every $r\in\{1,\dots,d\}$, that
\begin{equation}
\label{prop:decompo:proof:eq6}
2^{-j_r}\sup_{s\in [-T,T]^d}\bigg |\frac{\partial\Psi_{\alpha,-J}}{\partial x_r}\big(2^{-J}s-K\big)\bigg|\leq c_2\frac{2^{-j_r(1-a')}\left(2^{-j_1}+\dots+2^{-j_d}\right)^{-d/\alpha}\prod_{l=1}^d2^{-j_l/\alpha}}{\prod_{l=1}^d\big(1+\va{k_l}\big)^{\Lalpha}},
\end{equation}
where $c_2$ is a constant not depending on $(J,K)$. On the other hand \eqref{eq:bs} implies that
\begin{equation}
\label{prop:decompo:proof:eq6bis}
\sum_{K\in\Z^d}\,\prod_{l=1}^d \big(1+\va{k_l}\big)^{-\alpha \Lalpha}<+\ii.
\end{equation}
Finally, using~\eqref{prop:decompo:proof:eq4} to~\eqref{prop:decompo:proof:eq6bis}, and the same arguments as in the proof of~\eqref{le:majjj:eq1}, we get~\eqref{prop:decompo:proof:eq3}.

\underline{Case 2: $\alpha\in[1,2]$ and $\eta=0$.} The proof follows the same lines as in the case 1, except that one has to use \eqref{rem:Lalpha:eq2} instead of \eqref{rem:Lalpha:eq1}.
%; thus \eqref{prop:decompo:proof:eq4} has to be replaced by 
%\[
%\Delta_{\alpha}\left(\big(\Psi_{\alpha,-J}(2^{-J}t-K)-\Psi_{\alpha,{-J}}(-K)\big)\overline{\wh{\psi}_{\alpha,{-J},K}(\cdot)},0\right)=c_3\va{\Psi_{\alpha,{-J}}(2^{-J}t-K)-\Psi_{\alpha,{-J}}(-K)},
%\]
%where the constant $c_3:=\left(\int_{\R}|\wh{\psi^1}(\eta)|^\alpha\,\mathrm d\xi\right)^{d/\alpha}$. 

\underline{Case 3: $\alpha\in(0,1)$ and $\eta\neq 0$.}
It follows from~\eqref{rem:Lalpha:eq1}, the triangle inequality, and the sub-additivity on $[0,+\ii)$ of the function $z\mapsto z^{\alpha}$, that, for all $(J,K)\in\Geta\times\Z^d$, one has
\begin{eqnarray*}
\nonumber\Delta_{\alpha}\left(\big(\PsialphaJ(2^{J}t-K)-\Psi_{\alpha,{J}}(-K)\big)\overline{\wh{\psi}_{\alpha,{J},K}(\cdot)},0\right)&=&c_1\va{\Psi_{\alpha,{J}}(2^{J}t-K)-\Psi_{\alpha,{J}}(-K)}^{\alpha}\\
&\leq & c_1\va{\Psi_{\alpha,{J}}(2^{J}t-K)}^{\alpha}+\va{\Psi_{\alpha,{J}}(-K)}^{\alpha}\\
&\le & c_3 \prod_{l=1}^d 2^{(1-\eta_l)j_l}2^{-j_l\eta_l a_l \alpha}\left(\frac{1}{\big(2+\va{2^{j_l}t_l-k_l}\big)^{\alpha \Lalpha}}+\frac{1}{\big(2+\va{k_l}\big)^{\alpha \Lalpha}}\right).
\end{eqnarray*}
Notice that $c_3$ is a constant not depending on $(J,K)$. Also notice that the last inequality is obtained by using \eqref{prop:psij:eq2} in the case where $T=1$.
Next, this inequality, \eqref{def:geta:eq1}, \eqref{def:geta:eq2}, and \eqref{eq:bs} yield that
\begin{eqnarray*}
&&\nonumber \sum_{(J,K)\in\Geta\times\Z^d}\Delta_{\alpha}\left(\big(\PsialphaJ(2^{J}t-K)-\Psi_{\alpha,{J}}(-K)\big)\overline{\wh{\psi}_{\alpha,{J},K}(\cdot)},0\right)\\
&&\nonumber\leq c_3 \sum_{J\in\Geta}\prod_{l=1}^d 2^{(1-\eta_l)j_l}2^{-j_l\eta_l a_l \alpha}\left(\sum_{k_l\in\Z}\frac{1}{\big(2+\va{2^{j_l}t_l-k_l}\big)^{\alpha \Lalpha}}+\sum_{k_l\in\Z}\frac{1}{\big(2+\va{k_l}\big)^{\alpha \Lalpha}}\right)\\
&&\nonumber=c_3\sum_{J\in\Geta}\prod_{l=1}^d 2^{(1-\eta_l)j_l}2^{-j_l\eta_l a_l \alpha}\left(\sum_{k_l\in\Z}\frac{1}{\big(2+\va{2^{j_l}t_l-\lfloor2^{j_l}t_l\rfloor-k_l}\big)^{\alpha \Lalpha}}+\sum_{k_l\in\Z}\frac{1}{\big(2+\va{k_l}\big)^{\alpha \Lalpha}}\right)\\
&&\nonumber\leq 2^d c_3\prod_{l=1}^d\left\{\left(\sum_{j_l\in\Z_{\eta_l}} 2^{(1-\eta_l)j_l}2^{-j_l\eta_l a_l\alpha}\right)\left(\sum_{k_l\in\Z}\frac{1}{\big(1+\va{k_l}\big)^{\alpha \Lalpha}}\right)\right\}<+\ii,
\end{eqnarray*}
which show that~\eqref{prop:decompo:proof:eq3} holds.

\underline{Case 4: $\alpha\in[1,2]$ and $\eta\neq0$.} The proof follows the same lines as in the case 3, except that one has to use \eqref{rem:Lalpha:eq2} instead of \eqref{rem:Lalpha:eq1}.
%Using \eqref{rem:Lalpha:eq2}, the triangle inequality, \eqref{prop:psij:eq2}, and the equality $ \Lalpha=2$, we get that,
%\begin{eqnarray}
%&&\nonumber\sum_{(J,K)\in\Z^d\times\Z^d}\Delta_{\alpha}\left(\big(\PsialphaJ(2^{J}t-K)-\Psi_{\alpha,{J}}(-K)\big)\overline{\wh{\psi}_{\alpha,{J},K}(\cdot)},0\right)\\
%&&\nonumber\leq \sum_{(J,K)\in\Z^d\times\Z^d}\va{\PsialphaJ(2^Jt-K)}+\va{\Psi_{\alpha,{J}}(-K)}\\
%&&\nonumber\leq c_7 \prod_{l=1}^d\left\{\left(\sum_{j_l\in\Z_{\eta_l}} 2^{(1-\eta_l)j_l/\alpha}2^{-j_l\eta_l a_l}\right)\left(\sum_{k_l\in\Z^d}\frac{1}%{\big(1+\va{k_l}\big)^{2}}\right)\right\}\\
%&& \label{prop:decompo:proof:eq10}<+\ii,
%\end{eqnarray}
%where $c_7$ is a positive and finite constant. This shows that \eqref{prop:decompo:proof:eq3} holds.
\end{proof}

\begin{proof}[Proof of of Proposition~\ref{prop:decompo} (Step 2)]  For any fixed $m\in\N$, we denote by $\Theta_m$ the closed subset of $\R^d$ defined as 
\begin{equation}
\label{prop:decompo:proof:eq11}
\Theta_m:=\Big\{\xi=(\xi_1,\ldots,\xi_d)\in\R^d:\min\big\{|\xi_1|,\ldots,|\xi_d|\big\}\ge 2^{-m+1}\pi/3\Big\}.
\end{equation}
In view of \eqref{prop:decompo:eq2} and Definition~\ref{def:adm}, it can easily be seen that, for any fixed $t\in\R^d$, the function $F(t,\cdot)\ind{\Theta_m}(\cdot):\xi\mapsto F(t,\xi)\ind{\Theta_m}(\xi)$ belongs to the Hilbert space $\Lp{2}{\R^d}$ . Therefore, using the fact that $\{\psi_{J,K}:  (J,K)\in\Z^d\times\Z^d\}$ is an orthonormal basis of this space, similarly to \eqref{eq1:decomp}, one gets that
\begin{equation}
\label{prop:decompo:proof:eq12}
\lim_{n\to+\ii}\int_{\R^d}\bigg |F(t,\xi)\ind{\Theta_m}(\xi)-\sum_{(J,K)\in\mathcal{D}_n}w_{J,K}(t)\overline{\wh{\psi}_{{J},K}(\xi)}\bigg|^2\,\mathrm d\xi=0,
\end{equation}
%Recall that (see~\eqref{eq:wavs} and~\eqref{eq:wavs2}),
where
\begin{equation}
\label{prop:decompo:proof:eq13}
w_{J,K}(t):=\int_{\R^d}F(t,\xi)\ind{\Theta_m}(\xi)\,\mathrm{d}\xi=\int_{\Theta_m}\left(e^{it\cdot\xi}-1\right)f(\xi)\wpsiJK(\xi)\,\mathrm{d}\xi,
\end{equation}
and $(\mathcal{D}_n)_{n\in\N}$ is an arbitrary increasing sequence of finite subsets of~$\Z^d\times\Z^d$ such that $\bigcup_{n\in\N}\mathcal{D}_n=\Z^d\times\Z^d$. Next, we denote by $\mathcal{C}_m$ the compact subset of $\Theta_m$ defined as
\begin{equation}
\label{prop:decompo:proof:eq13ter}
\mathcal{C}_m:=\Big\{\xi=(\xi_1,\ldots,\xi_d)\in\R^d: 2^{m+3}\pi/3\ge\max\big\{|\xi_1|,\ldots,|\xi_d|\big\}\ge\min\big\{|\xi_1|,\ldots,|\xi_d|\big\}\ge 2^{-m+3}\pi/3\Big\}. 
\end{equation}
Let us show that, for all $(J,K)\in\Z^d\times\Z^d$ and $\xi\in\mathcal{C}_m$, one has
\begin{equation}
\label{prop:decompo:proof:eq13bis}
w_{J,K}(t)\overline{\wh{\psi}_{{J},K}(\xi)}=\big(\PsiJ(2^Jt-K)-\PsiJ(-K)\big)\overline{\wh{\psi}_{J,K}(\xi)},
\end{equation}
where the function $\PsiJ$ is as in \eqref{eq:psiJ}. To this end, we will study the following two cases: $\min\{j_1,\ldots, j_d\}< -m$ and $\min\{j_1,\ldots, j_d\}\ge -m$, where the integers $j_1, \ldots, j_d$ are the coordinates of $J$, that is $J=(j_1,\ldots, j_d)$. In the first case 
$\min\{j_1,\ldots, j_d\}< -m$, using \eqref{eq4:blm} and \eqref{prop:decompo:proof:eq13ter}, one gets that $\wpsiJK(\xi)=0$, for each $\xi\in\mathcal{C}_m$; therefore 
\eqref{prop:decompo:proof:eq13bis} holds. In the second case $\min\{j_1,\ldots, j_d\}\ge -m$, it follows from \eqref{eq4:blm} and  \eqref{prop:decompo:proof:eq11} that $\supp\,\wpsiJK\subset \Theta_m$. Thus, \eqref{prop:decompo:proof:eq13}, \eqref{eq3:blm}, the 
change of variable $(\eta_1,\ldots,\eta_d)=(2^{-j_1}\xi_1,\ldots , 2^{-j_d}\xi_d)$, and \eqref{eq:psiJ} imply that
$
w_{J,K}(t)=\PsiJ(2^Jt-K)-\PsiJ(-K).
$
Therefore \eqref{prop:decompo:proof:eq13bis} is satisfied.

Next, using \eqref{prop:decompo:proof:eq13bis}, \eqref{normaliz}, \eqref{prop:decompo:proof:eq12}, and the inclusion $\mathcal{C}_m\subset\Theta_m$ one gets that
\[
\lim_{n\to+\ii}\int_{\mathcal{C}_m}\bigg |F(t,\xi)-\sum_{(J,K)\in\mathcal{D}_n}\left(\PsialphaJ\left(2^Jt-K\right)-\PsialphaJ\left(-K\right)\right)\overline{\wh{\psi}_{\alpha,{J},K}(\xi)}\bigg|^2\,\mathrm d\xi=0.
\]
Then the H\"older inequality, combined with the fact that $\mathcal{C}_m$ has a finite Lebesgue measure, implies that
\begin{equation}
\label{prop:decompo:proof:eq14}
\lim_{n\to+\ii}\int_{\mathcal{C}_m}\Big |F(t,\xi)-\sum_{(J,K)\in\mathcal{D}_n}\left(\PsialphaJ\left(2^Jt-K\right)-\PsialphaJ\left(-K\right)\right)\overline{\wh{\psi}_{\alpha,{J},K}(\xi)}\Big|^\alpha\,\mathrm d\xi=0.
\end{equation}
On the other hand, \eqref{prop:decompo:proof:eq1} entails that,
\begin{equation}
\label{prop:decompo:proof:eq15}
\lim_{n\to+\ii}\int_{\mathcal{C}_m}\va{\wt{F}(t,\xi)-\sum_{(J,K)\in\mathcal{D}_n}\big(\PsialphaJ(2^Jt-K)-\PsialphaJ(-K)\big)\overline{\wh{\psi}_{\alpha,J,K}(\xi)}}^\alpha\,\mathrm d\xi=0.
\end{equation}
Finally, it follows from~\eqref{prop:decompo:proof:eq14}, and~\eqref{prop:decompo:proof:eq15} that, for all $m\in\N$ and for almost all $\xi\in\mathcal{C}_m$, one has $\wt{F}(t,\xi)={F}(t,\xi)$; this amounts to saying that $\wt{F}(t,\xi)={F}(t,\xi)$, {\em for almost all $\xi\in\R^d$},  since $\bigcup_{m\in\N}\mathcal{C}_m=(\R\setminus\{0\})^d$.
\end{proof}

\section{Proof of Lemma~\ref{le:eps-JK}}
\label{app:lepage}
In order to show that Lemma~\ref{le:eps-JK} holds, we need two preliminary results. The following proposition provides, when $\alpha\in (0,2)$, a LePage series representation of the complex-valued $\alpha$-stable process
\[
\bigg\{\int_{\R^d}\overline{\wpsialphaJK(\xi)}\,\mathrm d\wt{M}_{\alpha}(\xi): (J,K)\in\Z^d\times\Z^d\bigg\}.
\]
Its proof has been omitted since it is rather similar to that of Theorem 4.2 in~\cite{kono91lepage}.

\begin{propo}
\label{le:lepage}
We assume that the stability parameter $\alpha$ belongs to the open interval $(0,2)$, and we set
\begin{equation}
\label{eq1:prop-lepage}
a(\alpha):=\left(\int_0^{+\ii}x^{-\alpha}\sin(x)\,\mathrm{d} x\right)^{-1/\alpha}.
\end{equation}
Let $\{\kappa^m: m\in\N\}$, $\{\Gamma_m: m\in\N\}$, and $\{g_m: m\in\N\}$ be three arbitrary mutually independent sequences of random variables, defined on the same probability space $(\Om,\mathcal{G},\pr)$, having the following properties.
\begin{itemize}
\item The $\kappa^m$'s, $m\in\N$, are $\R^d$-valued, independent, identically distributed and absolutely continuous, with a probability density function, denoted by $\phi$, such that the measure $\phi(\xi)\mathrm d\xi$ is equivalent to the Lebesgue measure $\mathrm d\xi$ on $\R^d$.
\item The $\Gamma_m$'s, $m\in\N$, are Poisson arrival times with unit rate; that is, for all $m\in\N^*$, one has 
\begin{equation}
\label{eq2:prop-lepage}
\Gamma_m=\sum_{n=1}^m \nu_n, 
\end{equation}
where $(\nu_n)_{n\in\N}$ denotes a sequence of independent exponential random variables with the same parameter equal to $1$.
\item The $g_m$'s, $m\in\N$, are complex-valued, independent, identically distributed, rotationally invariant\footnote{That is, for all fixed $m\in\N$ and $\theta\in\R$, the random variables $e^{i\theta}g_m$ and $g_m$ have the same distribution.}  and satisfy $\esp[\va{\Ree(g_m)}^\alpha]=1$.
\end{itemize}
On the other hand, for every fixed $(J,K)\in\Z^d\times\Z^d$, let $\wpsialphaJK$ be the function defined in \eqref{eq4:blm}.

Then, the random series of complex numbers 
\[
\sum_{m=1}^{+\ii}g_m\Gamma_m^{-1/\alpha}\phi(\kappa^m)^{-1/\alpha}\overline{\wpsialphaJK(\kappa^m)}
\]
is almost surely convergent. Moreover, the stochastic processes 
\[
\left\{a(\alpha)\sum_{m=1}^{+\ii}g_m\Gamma_m^{-1/\alpha}\phi(\kappa^m)^{-1/\alpha}\overline{\wpsialphaJK(\kappa^m)}:\,\, (J,K)\in\Z^d\times\Z^d\right\}
\quad \text{and}\quad\Bigg\{\int_{\R^d}\overline{\wpsialphaJK(\xi)}\,\mathrm d\wt{M}_{\alpha}(\xi): (J,K)\in\Z^d\times\Z^d\Bigg\}
\]
have the same distribution. These two processes are identified throughout our article.
\end{propo}

\begin{lemme}
\label{le:gauss}
There exists a positive constant $c$ such that for any sequence of complex-valued centered~\footnote{That is satisfying $\esp(G_{J,K})=0$, for all $(J,K)\in\Z^d\times\Z^d$.} Gaussian random variables $\left\{G_{J,K}:(J,K)\in\Z^d\times\Z^d\right\}$, defined on $(\Om,\mathcal{G},\pr)$, one has
\begin{equation}
\label{le:gauss:eq1}
\esp\left\{\sup_{(J,K)\in\Z^d\times\Z^d}\left(\frac{\va{G_{J,K}}}{\logr{\sum_{l=1}^d\big (\va{j_l}+\va{k_l}\big)}}\right)\right\}\leq c\sqrt{\sup_{(J,K)\in\Z^d\times\Z^d}\esp\left[\va{G_{J,K}}^2\right]},
\end{equation}
where the $j_l$'s and $k_l$'s respectively denote the coordinates of $J$ and $K$.
\end{lemme}
\begin{proof}
We set,
\begin{equation}
\label{le:gauss:proof:eq1}
\SG:=\sqrt{\sup_{(J,K)\in\Z^d\times\Z^d}\esp\left[\va{G_{J,K}}^2\right]} \text{ \,\,\,\,and, for all $(J,K)\in\Z^d\times\Z^d$,\,\,\,\,} \bJK:=\logr{\sum_{l=1}^d\big(\va{j_l}+\va{k_l}\big)}.
\end{equation}
Clearly the lemma holds when $\SG=0$, and also when $\SG=+\ii$. Thus, in the sequel, we assume that $0<\SG<+\ii$. Using the fact that the expectation of an arbitrary non-negative random variable $Z$ can be expressed as $
%\begin{equation*}
\esp[Z]=\int_{0}^{+\ii}\pr(Z>x)\,\mathrm dx,
%\end{equation*}
$ we get that
\begin{eqnarray}
\label{le:gauss:proof:eq2}
\esp\left[\sup_{(J,K)\in\ZZ}\left(\frac{\va{\GJK}}{\SG\bJK}\right)\right] &=&\int_{0}^{+\ii}\pr\left(\sup_{(J,K)\in\ZZ}\left( \frac{\va{\GJK}}{\SG\bJK}\right)>x\right)\,\mathrm dx\nonumber\\
&\leq& 2^{d+1}+\int_{2^{d+1}}^{+\ii}\pr\left(\sup_{(J,K)\in\ZZ}\left(\frac{\va{\GJK}}{\SG\bJK}\right)>x\right)\,\mathrm dx\nonumber\\
&\leq&  2^{d+1}+\sum_{(J,K)\in\ZZ}\int_{2^{d+1}}^{+\ii}\pr\left(\frac{\va{\GJK}}{\SG\bJK}>x\right)\,\mathrm dx,
\end{eqnarray}
where the last inequality follows from the equality 
\[
 \bigg\{\om\in\Om: \sup_{(J,K)\in\ZZ}\left(\frac{\va{\GJK(\om)}}{\SG\bJK}\right)>x\bigg\}=\bigcup_{(J,K)\in\ZZ}\bigg\{\om\in\Om: \frac{\va{\GJK(\om)}}{\SG\bJK}>x\bigg\}.
\]
Next, denoting by $\Ree(\GJK)$ and $\Imm(\GJK)$ the real and the imaginary parts of $\GJK$, then, in view of the equality $|\GJK|=\sqrt{\va{\Ree(\GJK)}^2+\va{\Imm(\GJK)}^2}$, for all $x\ge 2^{d+1}$, one has
\begin{equation}
\label{le:gauss:proof:eq3}
\pr\left(\frac{\va{\GJK}}{\SG\bJK}>x\right)\leq\pr\left(\frac{\va{\Ree(\GJK)}}{\SG\bJK}>2^{-1/2}x\right)+\pr\left(\frac{\va{\Imm(\GJK)}}{\SG\bJK}>2^{-1/2}x\right).
\end{equation}
Now, we are going to show that
\begin{equation}
\label{le:gauss:proof:eq4}
\pr\left(\frac{\va{\Ree(\GJK)}}{\SG\bJK}>2^{-1/2}x\right)\leq\exp\left(-2^{-2}\,\bJK^2\,x^2\right);
\end{equation}
similarly, it can be shown that
\begin{equation}
\label{le:gauss:proof:eq7}
\pr\left(\frac{\va{\Imm(\GJK)}}{\SG\bJK}>2^{-1/2}x\right)\leq\exp\left(-2^{-2}\,\bJK^2\,x^2\right).
\end{equation} 
We set
\begin{equation*}
\sigma(\GJK):=\sqrt{\esp\left[\va{\Ree(\GJK)}^2\right]};
\end{equation*}
observe that, in view of the first equality in \eqref{le:gauss:proof:eq1}, one has
\begin{equation}
\label{le:gauss:proof:eq4bis}
\SG\ge\sigma(\GJK).
\end{equation}
It is clear that \eqref{le:gauss:proof:eq4} holds when $\sigma(\GJK)=0$, since $\Ree(\GJK)$ is then vanishing almost surely. So, in the sequel  we assume that $\sigma(\GJK)>0$. Hence $\Ree(\GJK)/\sigma(\GJK)$ is a well-defined real-valued standard Gaussian random variable. Therefore, using \eqref{le:gauss:proof:eq4bis} and the fact that $2^{-1/2}\bJK x\ge 2^d\,\sqrt{2\log 3}\ge 1$, we get that
\begin{eqnarray*}
\pr\left(\frac{\va{\Ree(\GJK)}}{\SG\bJK}>2^{-1/2}x\right) &\leq & \pr\left(\frac{\va{\Ree(\GJK)}}{\sigma(\GJK)\bJK}>2^{-1/2}x\right)\\
&\leq &  \int_{2^{-1/2}\bJK x}^{+\infty}e^{-y^2/2}\,\mathrm dy\leq\int_{2^{-1/2}\bJK x}^{+\infty} ye^{-y^2/2}\,\mathrm dy
=\exp\left(-2^{-2}\,\bJK^2\,x^2\right),
\end{eqnarray*}
which shows that \eqref{le:gauss:proof:eq4} holds.

Next putting together \eqref{le:gauss:proof:eq3}, \eqref{le:gauss:proof:eq4}, \eqref{le:gauss:proof:eq7} and the inequalities $2^{-2}\,b_{J,K}^2\,x\ge 2^{d-1}\log 3\ge 1$, we obtain that
\begin{equation}
\label{le:gauss:proof:eq8}
\int_{2^{d+1}}^{+\ii}\pr\left(\frac{\va{\GJK}}{\SG\bJK}>x\right)\,\mathrm dx\leq 2\int_{2^{d+1}}^{+\ii}2^{-2}\,\bJK^2\,x\exp\left(-2^{-2}\,\bJK^2\,x^2\right)\,\mathrm dx=\exp\left(-2^{2d}\,\bJK^2\right).
\end{equation}

Finally, in view of \eqref{le:gauss:proof:eq1}, \eqref{le:gauss:proof:eq2} and \eqref{le:gauss:proof:eq8}, it turns out that in order to obtain \eqref{le:gauss:eq1} it is enough to show that 
\[
\sum_{(J,K)\in\ZZ}\bigg (3+\sum_{l=1}^d\big(\va{j_l}+\va{k_l}\big)\bigg)^{-4^{d}}<+\infty.
\]
This can be shown by noticing that $4^d\ge 4d$ and that  
\begin{eqnarray*}
\bigg (3+\sum_{l=1}^d\big(\va{j_l}+\va{k_l}\big)\bigg)^{-4^{d}} &\le & \bigg (3+\sum_{l=1}^d\big(\va{j_l}+\va{k_l}\big)\bigg)^{-4d} \\
&=&\prod_{m=1}^d  \bigg (3+\sum_{l=1}^d\big(\va{j_l}+\va{k_l}\big)\bigg)^{-4} \\
&\le & \prod_{m=1}^d  \bigg (3+\big(\va{j_m}+\va{k_m}\big)\bigg)^{-4} \le \prod_{m=1}^d \big(3+\va{j_m}\big)^{-2} \big(3+\va{k_m}\big)^{-2}.
\end{eqnarray*}
\end{proof}

We are now in the position to prove Lemma~\ref{le:eps-JK}.
\begin{proof}[Proof of Lemma~\ref{le:eps-JK}]
First we recall that the third result provided by Lemma~\ref{le:eps-JK} (in other words the inequality \eqref{le:eps-JK:eq3} which holds in the Gaussian case $\alpha=2$) is rather classical. We will skip its proof; it can be found in e.g.~\cite{AT03}. In all the sequel, we assume that $\alpha\in (0,2)$. Notice that, in view of \eqref{eq:eps-stable}, for all $(J,K)\in\Z^d\times\Z^d$, one clearly has 
\begin{equation}
\label{le:eps-JK:proof:eq0}
|\epsalphaJK|\le\Big |\int_{\R^d}\overline{\wpsialphaJK(\xi)}\,\mathrm d\wt{M}_{\alpha}(\xi)\Big|.
\end{equation}
Thus, in order to get \eqref{le:eps-JK:eq1} and \eqref{le:eps-JK:eq2}, it is enough to show that these two inequalities are satisfied when 
$\epsalphaJK$ in them is replaced by $\int_{\R^d}\overline{\wpsialphaJK(\xi)}\,\mathrm d\wt{M}_{\alpha}(\xi)$. The advantage of this strategy is that we know from Proposition~\ref{le:lepage} that, for each $(J,K)\in\Z^d\times\Z^d$,
\begin{equation}
\label{le:eps-JK:proof:eq1}
\int_{\R^d}\overline{\wpsialphaJK(\xi)}\,\mathrm d\wt{M}_{\alpha}(\xi)=a(\alpha)\sum_{m=1}^{+\ii} g_m\Gamma_m^{-1/\alpha}\phi(\kappa^m)^{-1/\alpha}\overline{\wpsialphaJK(\kappa^m)};
\end{equation} 
moreover, we can and will assume that the $g_m$'s, $m\in\N$, are complex-valued centered Gaussian random variables, and that the function $\phi$ is such that for all $\xi=(\xi_1,\dots,\xi_d)\in\big(\R\setminus\{0\}\big)^d$, one has
\[
\phi(\xi):=\left(\frac{\epsilon}{4}\right)^d\,\prod_{l=1}^d\va{\xi_l}^{-1}\left(1+\va{\log{\va{\xi_l}}}\right)^{-1-\epsilon},
\]
where $\epsilon$ is an arbitrary  fixed positive real number. Therefore, using~\eqref{eq4:blmalpha} and~\eqref{eq6:blmalpha}, we obtain, for every $m\in\N$ and $(J,K)\in\Z^d\times\Z^d$, that
\begin{eqnarray}
\nonumber\va{\phi(\kappa^m)^{-1/\alpha}\overline{\wpsialphaJK(\kappa^m)}}&\leq&\left(\frac{\epsilon}{4}\right)^{-d/\alpha}\prod_{l=1}^d\va{2^{-j_l}\kappa_l^m}^{1/\alpha}\left(1+\va{j_l}+\va{\log{\va{2^{-j_l}\kappa_l^m}}}\right)^{(1+\epsilon)/\alpha}\va{\wh{\psi^1}(2^{-j_l}\kappa_l^m)}\\
&\leq&\label{le:eps-JK:proof:eq3} c_1 \prod_{l=1}^d\left(1+\va{j_l}\right)^{(1+\epsilon)/\alpha},
\end{eqnarray}
where $c_1$ is a deterministic constant not depending on $(J,K)$ and $m$. On the other hand, in view of the Gaussianity assumption 
on the $g_m$'s, $m\in\N$, it can be derived from the Borel-Cantelli Lemma that, almost surely, for all $m\in\N$, one has 
\begin{equation}
\label{le:eps-JK:proof:eq4}
|g_m|\leq C_2\logr{m},
\end{equation}
where $C_2$ is a finite random variable not depending on $(J,K)$ and $m$. Also, observe that, in view of \eqref{eq2:prop-lepage}, it results
from the strong law of large number, that almost surely, for any $m\in\N$, the Poisson arrival time $\Gamma_m$ satisfies 
\begin{equation}
\label{le:eps-JK:proof:eq5}
C_3 m\leq\Gamma_m\leq C_4 m,
\end{equation}
where $C_3$ and $C_4$ are two positive finite random variables not depending on $(J,K)$ and $m$. Next, we suppose for a while that $\alpha\in (0,1)$, then the random variable 
\[
C_5:=a(\alpha)c_1C_2C_3^{-1/\alpha}\sum_{m=1}^{+\ii}m^{-1/\alpha}\logr{m}
\]
is almost surely finite; moreover, it follows from the triangle inequality and from~\eqref{le:eps-JK:proof:eq1}~to~\eqref{le:eps-JK:proof:eq5} that, almost surely, for all $(J,K)\in\ZZ$, one has
\begin{equation*}
\Big |\int_{\R^d}\overline{\wpsialphaJK(\xi)}\,\mathrm d\wt{M}_{\alpha}(\xi)\Big|\leq  a(\alpha)\sum_{m=1}^{+\ii}\va{g_m}\Gamma_m^{-1/\alpha}\phi(\kappa^m)^{-1/\alpha}\va{\overline{\wpsialphaJK(\kappa^m)}}\leq C_5\prod_{l=1}^d\left(1+\va{j_l}\right)^{(1+\epsilon)/\alpha}.
\end{equation*}
These inequalities combined with \eqref{le:eps-JK:proof:eq0} show that \eqref{le:eps-JK:eq1} holds.

From now on, we assume that $\alpha\in[1,2)$ and our goal is to derive~\eqref{le:eps-JK:eq2}; notice that the previous strategy has to be modified since $C_5$ is no longer finite. Let $\mathcal{F}_{\Gamma,\kappa}$ be the sub $\sigma-$field of $\mathcal{G}$ generated by the two sequences of random variables $ \left\{\Gamma_m:m\in\N\right\}$ and $ \left\{\kappa^m:m\in\N\right\}$. We denote by $ \esp_{\Gamma,\kappa}[\,\cdot\,] $ the conditional expectation operator with respect to $\mathcal{F}_{\Gamma,\kappa}$; recall that 
$\esp(\,\cdot\,)$ denotes the classical expectation operator. We know from~\eqref{le:eps-JK:proof:eq1} that conditionally to $\mathcal{F}_{\Gamma,\kappa}$, for any arbitrary $ (J,K)\in\ZZ, $ the random variable
\begin{equation}
\label{le:eps-JK:proof:eq6bis}
G_{J,K}:=\Big(\prod_{l=1}^d\left(1+\va{j_l}\right)^{-(1+\epsilon)/\alpha}\Big) \int_{\R^d}\overline{\wpsialphaJK(\xi)}\,\mathrm d\wt{M}_{\alpha}(\xi)
\end{equation}
 has a centered Gaussian distribution over $\C$. Then, assuming that $c_6$ denotes the constant $c$ in \eqref{le:gauss:eq1}, one can derive from Lemma~\ref{le:gauss} that the following inequality holds almost surely:
\begin{equation}
\label{le:eps-JK:proof:eq6ter}
\esp_{\Gamma,\kappa}\left [\sup_{(J,K)\in\Z^d\times\Z^d}\Bigg(\frac{\va{G_{J,K}}}{\logr{\sum_{l=1}^d\big (\va{j_l}+\va{k_l}\big)}}\Bigg)\right ]\leq c_6\sqrt{\sup_{(J,K)\in\Z^d\times\Z^d}\esp_{\Gamma,\kappa}\left[\va{G_{J,K}}^2\right]}.
\end{equation}
Next, using the fact that $\esp(\,\cdot\,)=\esp \big (\esp_{\Gamma,\kappa}[\,\cdot\,]\big)$,  Cauchy-Schwarz inequality, and \eqref{le:eps-JK:proof:eq6ter}, one obtains that 
\begin{eqnarray}
\label{le:eps-JK:proof:eq8}
&&\nonumber\esp\left (\sqrt{\sup_{(J,K)\in\Z^d\times\Z^d}\Bigg(\frac{|G_{J,K}|}{\logr{\sum_{l=1}^d\va{j_l}+\va{k_l}}}\Bigg)}\right)=\esp\left (\esp_{\Gamma,\kappa}\left [\sqrt{\sup_{(J,K)\in\Z^d\times\Z^d}\Bigg(\frac{|G_{J,K}|}{\logr{\sum_{l=1}^d\va{j_l}+\va{k_l}}}\Bigg)}\right]\right)\\
&&\leq \esp\left(\sqrt{\esp_{\Gamma,\kappa}\Bigg [\sup_{(J,K)\in\Z^d\times\Z^d}\Bigg (
\frac{|G_{J,K}|}{\logr{\sum_{l=1}^d\va{j_l}+\va{k_l}}}\Bigg)\Bigg]}\right)\leq\sqrt{c_6}\,\esp\left (\bigg (\sup_{(J,K)\in\Z^d\times\Z^d}\esp_{\Gamma,\kappa}\Big[\va{G_{J,K}}^2\Big]\bigg)^{1/4}\right ).\nonumber\\
\end{eqnarray}
On the other hand, \eqref{le:eps-JK:proof:eq1} and~\eqref{le:eps-JK:proof:eq6bis} imply that, one has, almost surely, for any arbitrary $(J,K)\in\Z^d\times\Z^d$,
\begin{equation*}
\esp_{\Gamma,\kappa}\left[\va{G_{J,K}}^2\right]=c_7\Big(\prod_{l=1}^d\left(1+\va{j_l}\right)^{-2(1+\epsilon)/\alpha}\Big)\sum_{m=1}^{+\ii}\Gamma_m^{-2/\alpha}\phi(\kappa^m)^{-2/\alpha}\va{\wpsialphaJK (\kappa^m)}^2,
\end{equation*}
where the deterministic constant $c_7:=a(\alpha)^2\,\esp\big(|g_1|^2\big)$ does not depend on $(J,K)$. Then, using~\eqref{le:eps-JK:proof:eq3}, one gets, almost surely, that
\begin{equation}
\label{le:eps-JK:proof:eq6}
\sup_{(J,K)\in\Z^d\times\Z^d}\esp_{\Gamma,\kappa}\left[\va{G_{J,K}}^2\right]\leq c_8\sum_{m=1}^{+\ii}\Gamma_m^{-2/\alpha},
\end{equation}
where the deterministic constant $c_8:=c_1^2c_7$. Finally, in view of  \eqref{le:eps-JK:proof:eq0}, \eqref{le:eps-JK:proof:eq6bis}, \eqref{le:eps-JK:proof:eq8} and \eqref{le:eps-JK:proof:eq6}, it turns out that \eqref{le:eps-JK:eq2} can be obtained by showing that
\begin{equation}
\label{le:eps-JK:proof:eq9}
\esp\left(\bigg(\sum_{m=1}^{+\infty}\Gamma_{m}^{-2/\alpha}\bigg)^{1/4}\right)<+\infty.
\end{equation}
We know from Remark~4 on page 29 in~\cite{SamTaq}, that the positive random variable $\sum_{m=1}^{+\infty}\Gamma_{m}^{-2/\alpha}$ has
a stable distribution with a stability parameter equal to $\alpha/2$. Thus combining the fact that $\alpha/2>1/4$ with the Property 1.2.16 on page 18 in~\cite{SamTaq}, one gets \eqref{le:eps-JK:proof:eq9}.
\end{proof}

\section*{Acknowledgements}
The authors are very grateful to the anonymous two referees for their careful reading of the article. This work has been partially supported by ANR-11-BS01-0011 (AMATIS), GDR 3475 (Analyse Multifractale), and ANR-11-LABX-0007-01 (CEMPI).

%\newpage
\bibliography{biblio-de-Ref}
\bibliographystyle{abbrv}
\end{document}